\documentclass[10pt]{amsart}
\usepackage{latexsym, amsmath, amssymb, longtable, booktabs, amscd, graphicx}

\usepackage[english]{babel}
\usepackage[utf8]{inputenc}

\usepackage{hyperref}
\usepackage[all]{xy}
\usepackage{color}
\theoremstyle{plain}
\usepackage[margin=1in]{geometry}
\usepackage [english]{babel}
\usepackage[autostyle]{csquotes}
\usepackage{relsize}
\usepackage{dsfont}
\usepackage{bbm}
\usepackage{lipsum}
\usepackage{enumerate}   
\usepackage{enumitem}
\usepackage{tabstackengine}
\usepackage{tikz}
\allowdisplaybreaks

\usepackage{hyperref}
\usepackage{enumerate}
\usepackage{enumitem}
\usepackage{mathtools}

\usepackage{tikz}
\usepackage{pst-node}

\usepackage[all]{xy}
\numberwithin{equation}{section}
\newcommand{\ind}{\operatorname{ind}}

\newcommand\GL{{\mathrm{GL}}}

\newcommand{\Q}{\mathbb{Q}}
\newcommand{\Z}{\mathbb{Z}}

\newcommand{\F}{\mathbb{F}}

\newcommand\Sym{{\mathrm {Sym}}}

\newcommand\iso{{\> \simeq \>}}
 
\newcommand\linee{{
        {\begin{tiny}
        \begin{xymatrix}{
         \bullet  \ar@{-}[d] \\ 
         \bullet  \ar@{-}[d] \\
         \bullet} 
       \end{xymatrix}
       \end{tiny}} }}  
\newcommand\diamondd{{
        \begin{tiny}
        \begin{xymatrix}{
         & \bullet  \ar@{-}[dl] \ar@{-}[dr] &  \\ 
         \bullet  \ar@{-}[dr] &   &   \bullet  \ar@{-}[dl] \\
         & \bullet & } 
       \end{xymatrix}
       \end{tiny} }}

\newtheorem{thm}{Theorem}[section]
\newtheorem{theorem}[thm]{Theorem}
\newtheorem{cor}[thm]{Corollary}

\newtheorem{prop}[thm]{Proposition}
\newtheorem{lemma}[thm]{Lemma}
\theoremstyle{definition}

\theoremstyle{remark}
\newtheorem{remark}[thm]{Remark}

\theoremstyle{remark}

\begin{document}

\title[Reductions of Galois Representations]{Reductions of Galois representations of Slope $\frac{3}{2}$}

\author[E. Ghate]{Eknath Ghate} 
\address{School of Mathematics, Tata Institute of Fundamental Research, Homi Bhabha Road, Mumbai-5, India}
\email{eghate@math.tifr.res.in}

\author[V. Rai]{Vivek Rai}
\address{School of Mathematics, Indian Institute of Science Education and Research, Pashan, Pune-8, India}
\email{vvkrai@iiserpune.ac.in}

\begin{abstract}
We prove a zig-zag conjecture describing the reductions
of irreducible crystalline two-dimensional representations of $G_{\Q_p}$ of 
slope $\frac{3}{2}$ and exceptional weights. 
%These weights are two more than twice the
%slope mod $(p-1)$ and are the hardest to treat. 
%The conjecture matches with recent results of Buzzard-Gee for slope $\frac{1}{2}$ and 
%Bhattacharya-Ghate-Rozensztajn for slope 1. 
%We give evidence towards this conjecture by proving it for the case of slope $\frac{3}{2}$. 
This along with previous works completes the description of the reduction 
for all slopes less than $2$. The proof involves computing the reductions of the
Banach spaces attached by the $p$-adic LLC to these representations, followed by 
an application of the mod $p$ LLC to recover the reductions of these representations.
\end{abstract}

\subjclass[2010]{Primary: 11F80} %, Secondary: 11F80, 11F30}
\keywords{Reductions of crystalline Galois representations, mod $p$ Local Langlands Correspondence.}
%\date{September 2015}
\maketitle

%%%%%%%%%%%%%%%%%%%%%%%%%%%%%%%%%%%%%%%%%%%%%%%%%%%%%%%%%%%%%%%%%%%%%%%%%%%%%%%%%%%%%%%%%%%%%%%%%%%%%%%%%%%%%%%%%%%%%%%%%%%%%%%%%%%%%%%%%%%%%%%%%%%%%

%%%%%

\section{Introduction}
\label{sectionintro}

Let $p$ be an odd prime. This paper is concerned with computing the
reductions of certain crystalline two-dimensional representations of the
local Galois group $G_{\Q_p}$. This problem is classical and important in view of 
its applications to Galois representations attached to modular forms.  

\subsection{Main Result} Let $k \geq 2$ be an integer and let $a_p$ lie in a finite extension $E$ of $\Q_p$. Assume that $v(a_p) > 0$,
where $v$ is the $p$-adic valuation of $\Q_p$, normalized so that $v(p) = 1$. 
Let $V_{k,a_p}$ be the irreducible two-dimensional crystalline representation of $G_{\Q_p}$ defined over $E$ of weight $k \geq 2$ and positive slope 
$v(a_p)$.
Let $\bar{V}_{k,a_p}$ be the {\it semisimplification} of the 
reduction of $V_{k,a_p}$ modulo the maximal ideal of the ring of integers of $E$. It is a two-dimensional semisimple representation of
$G_{{\mathbb Q}_p}$ defined over $\bar{\mathbb F}_p$, and is independent of the choice of the lattice used to define the reduction.  
%If $f = \sum_{n=1}^\infty a_n q^n$ is a primitive cusp form of weight $k \geq 2$, level coprime to $p$, and (for example) trivial nebentypus character,
%then the local Galois representation $\rho_f |_{G_{\Q_p}}$ attached to $f$ and $p$ is isomorphic to $V_{k,a_p}$,  at least if $a_p^2 \neq 4 p^{k-1}$, 
%so the reduction $\bar{\rho}_f |^{ss}_{G_{\Q_p}}$ is isomorphic to $\bar{V}_{k,a_p}$.

For simplicity, we often write $v$ for the slope $v(a_p)$. 
The reduction $\bar{V}_{k,a_p}$ is known by classical work of Fontaine and 
Edixhoven \cite{Edixhoven92} and its subsequent extension by Breuil \cite{Br03} for small weights $k \leq 2p+1$, and for all large slopes 
$v > \lfloor \frac{k-1}{p-1} \rfloor$
by Berger-Li-Zhu \cite{BLZ}. There has been a spate of recent work  computing the reduction $\bar{V}_{k,a_p}$ for small slopes. 
Buzzard-Gee \cite{BG09}, \cite{BG13} treated the case of slopes $v$ in $(0,1)$.
%though the case of slope $v = \frac{1}{2}$ was treated 
%subsequently in \cite{BG13}. 
The case of slopes $v$ in $(1,2)$ was treated in \cite{GG15}, \cite{Bhattacharya-Ghate}, 
but only under an assumption when $v = \frac{3}{2}$. The missing case of slope $v = 1$ 
was treated subsequently in \cite{BGR18}.  
The general goal of this paper is to give a complete treatment of the case of slope $v = \frac{3}{2}$, 
thereby filling a gap in the literature computing the reduction for all small slopes less than 2. 

%A more general aim of this paper is as follows. 
More specifically, let us say that a weight $k$ is {\it exceptional} for a particular half-integral 
(and possibly integral) slope $v \in \frac{1}{2} \Z$ if 
$$k \equiv 2v+2 \mod (p-1).$$ 
These weights turn out to be the hardest to treat in the sense that the reduction seems to take on
more possibilities rather than just the generic answer for that slope. 
In \cite[Conjecture 1.1]{G19}, a general zig-zag conjecture was made describing the 
reduction $\bar{V}_{k,a_p}$ for all exceptional weights for all half-integral slopes $0 < v \leq \frac{p-1}{2}$. 
The conjecture specializes
to known theorems when $v = \frac{1}{2}$ \cite{BG13} and $v = 1$ \cite{BGR18}.
In this paper, we shall prove the zig-zag conjecture for slope $v = \frac{3}{2}$. In particular, this removes
an assumption made for exceptional weights in \cite{Bhattacharya-Ghate} when $v = \frac{3}{2}$, giving
a complete description of the reduction for slope $v = \frac{3}{2}$.

In order to state the main theorem, let us recall some recent history on the reduction problem for exceptional 
weights. To do this we introduce some standard notation. 
Let $\omega$ and $\omega_2$
be the mod $p$ fundamental characters of levels $1$ and $2$. Let $\ind(\omega_2^c)$ be the (irreducible) mod $p$ 
representation of $G_{\Q_p}$ obtained
by inducing the $c$-th power of $\omega_2$ from the index 2 subgroup $G_{\Q_{p^2}}$ of $G_{\Q_p}$ to $G_{\Q_p}$ (for $p+1 \nmid c$). 
Let $\mu_\lambda$ be the
unramified character of $G_{{\mathbb Q}_p}$ mapping a (geometric) Frobenius at $p$ to $\lambda \in \bar{\mathbb F}_p$. 
Let $r := k-2$ and let $b \in \{1,2, \cdots, p-1\}$ represent the congruence class of $r$ 
mod $(p-1)$. Then $b = 2v$ is a representative for the exceptional congruence class of weights $r$ mod $(p-1)$. In particular,
$b = 1, 2, 3, \cdots$, represent the exceptional congruence classes of weights $r$ mod $(p-1)$ for 
the half-integral slopes $v = \frac{1}{2}, 1, \frac{3}{2}, \cdots$, respectively.  

In \cite{BG09}, Buzzard-Gee showed that 
the reduction $\bar{V}_{k,a_p}$ is always {\it irreducible} for slopes $v$ in $(0,1)$ (and isomorphic to  $\ind (\omega_2^{b+1})$), except possibly in the exceptional case $v = \frac{1}{2}$ and $b = 1$. This case was only treated later in \cite{BG13}, where the 
authors show that when a certain parameter, which we call $\tau$, 
is larger than another parameter, which we call $t$, a reducible possibility (namely $\omega \oplus \omega$ on inertia) occurs instead. 
More precisely, setting 
\begin{eqnarray*}
  \tau & = & v \left( \frac{a_p^2 - r p}{p a_p} \right), \\
    t  & = & v (1-r), 
\end{eqnarray*}
it is shown in \cite[Theorem A]{BG13} that in the exceptional case $v = \frac{1}{2}$ and $b = 1$, there is a {\it dichotomy}:
\begin{eqnarray*}
  \bar{V}_{k,a_p} & \sim & 
  \begin{cases}
    \mathrm{ind}(\omega_2^{b+1}),  & \text{if }  \tau < t  \\
	\mu_{\lambda} \cdot \omega^b \,\oplus\,\mu_{\lambda^{-1}} \cdot \omega, & \text{if }  \tau \geq t,
  \end{cases}  
\end{eqnarray*}
for $r > 1$, where $\lambda$ is a root of the quadratic equation  
\begin{eqnarray*}
  \lambda + \frac{1}{\lambda} & = & \overline{\frac{1}{1-r} \cdot \frac{a_p^2 - r p}{p a_p}}.
\end{eqnarray*} 

In \cite{BGR18}, the reduction $\bar{V}_{k,a_p}$ was completely determined on the boundary $v = 1$ of the annulus $(0,1)$, 
and was shown to be generically {\it reducible} instead. In the difficult 
exceptional case $v=1$ and $b = 2$, the authors established, %(after some previous iterations of the paper on the {\tt arXiv})
for $r > 2$, the following {\it trichotomy}:

%Let us introduce some notation.
%Let $E$ be a finite extension field of $\Q_p$ and let $v$ be the
%valuation of $\bar\Q_p$ normalized so that $v(p) = 1$. Let $a_p \in E$
%with $v(a_p) > 0$ and let $k \geq 2$.  Let $V_{k,a_p}$ be the irreducible
%crystalline representation of $G_{\Q_p}$ with Hodge-Tate weights
%$(0,k-1)$ and slope $v(a_p) > 0$ such that $D_\mathrm{cris}(V_{k,a_p}^*)
%= D_{k,a_p}$, where $D_{k,a_p}  = E e_1 \oplus E e_2$ is the filtered
%%$\varphi$-module as defined in \cite[\S 2.3]{Berger11}.  The
%semisimplification $\bar{V}_{k,a_p}^{ss}$ of the reduction
%$\bar{V}_{k,a_p}$ with respect to a lattice in ${V}_{k,a_p}$ is
%independent of the choice of the lattice.  Let $\omega = \omega_1$ and
%$\omega_2$ denote the fundamental characters of level 1 and 2
%respectively. Let $\mathrm{ind}(\omega_2^{a})$ denote the  representation
%of $G_{\Q_p}$ obtained by inducing the character $\omega_2^a$, for $a \in
%\Z$, from $G_{\Q_{p^2}}$ to  $G_{\Q_p}$; it is irreducible if $p+1 \nmid
%a$.  Finally, let $\mu_x$ be the unramified character of $G_{\Q_p}$
%taking (geometric) Frobenius at $p$ to $x \in \bar\F_p^\times$. 
%
%The following theorem describes the reduction $\bar{V}_{k,a_p}^{ss}$ when
%the slope $v(a_p)$ is equal to $1$, for all primes $p > 3$.
%
%\begin{theorem}
% \label{maintheoslopeone}
%  Let $p>3$, let $k\geq 2p+2$ and let $r = k-2 \equiv b \mod (p-1)$, with $ 2 \leq b\leq p$. Suppose that the slope $v(a_p)=1$.  
%  Then $\bar{V}_{k, a_p}^{ss}$ is as follows:
%\begin{eqnarray*}
%
%
\begin{eqnarray*}
  \bar{V}_{k, a_p} & \sim & 
  \begin{cases}
    \mathrm{ind}(\omega_2^{b+1}),  & \text{if } \tau < t \\
    \mu_{\lambda} \cdot \omega^b\,\oplus\,\mu_{\lambda^{-1}} \cdot \omega, & \text{if } \tau = t \\ 
    \mathrm{ind}(\omega_2^{b+p}), & \text{if } \tau > t,       
\end{cases} 
\end{eqnarray*}
where
\begin{eqnarray*}
  \tau & = & v \left( \frac{a_p^2 - \binom{r}{2} p^2}{p a_p} \right), \\
    t  & = & v (2-r), 
\end{eqnarray*}
and $\lambda$ is a constant given by
\begin{eqnarray*}
    \lambda & = &  \overline{\dfrac{2}{2-r} \cdot \dfrac{a_p^2 -\binom{r}{2} p^2}{pa_p} }.
\end{eqnarray*}

Based on these results for slopes $\frac{1}{2}$ and $1$, and some  computations in \cite{Roz16} for a few
larger half-integral slopes, one might guess that in the general exceptional case $b = 2v$ for half-integral slope $v \leq \frac{p-1}{2}$,
there are $b+1$ possibilities for
$\bar{V}_{k,a_p}$, with various irreducible and reducible cases occurring alternately. A more concrete version of 
this guess (and several subtleties related to it) were outlined by one of the authors in a conjecture called the zig-zag conjecture,
see Conjecture 1.1 of \cite{G19}.
That this conjecture is indeed true for slope $v = \frac{3}{2}$ is the main theorem of this paper:

\begin{theorem} 
  \label{maintheorem}
  If $v = \frac{3}{2}$, then the zig-zag conjecture \cite{G19} is true.  
  More precisely, say $p \geq 5$, the slope $v = \frac{3}{2}$ and $b = 3$, so that the congruence class of $r$ mod $(p-1)$ is exceptional.
  If $r > b$, 
\begin{eqnarray*}
    c & = &  \frac{a_p^2 - (r-2) \binom{r-1}{2} p^3}{p a_p}, 
  \end{eqnarray*}
  and
  \begin{eqnarray*}
  \tau & = & v \left( c \right), \\
      t  & = & v (b-r),
  \end{eqnarray*}
then  the reduction $\bar{V}_{k,a_p}$ satisfies the following {\it tetrachotomy}:
\begin{eqnarray*}
\bar{V}_{k, a_p} & \sim & 
  \begin{cases}
    \mathrm{ind}(\omega_2^{b+1}),  & \text{if } \tau < t \\
    \mu_{\lambda_1} \cdot \omega^b\,\oplus\,\mu_{\lambda_1^{-1}} \cdot \omega, & \text{if } \tau = t \\ 
    \mathrm{ind}(\omega_2^{b+p}), & \text{if } t <  \tau < t+1 \\       
    \mu_{\lambda_2} \cdot \omega^{b-1}\,\oplus\,\mu_{\lambda_2^{-1}} \cdot \omega^2, & \text{if } \tau \geq t+1, 
    \end{cases}
\end{eqnarray*}
where the $\lambda_i$, for $i = 1$, $2$, are constants given by
\begin{eqnarray*}
  \lambda_1 &  = & \overline{\dfrac{b}{b-r} \cdot c } \\
  \lambda_2 + \frac{1}{\lambda_2} & = &  \overline{\dfrac{b-1}{(b-1-r)(b-r)} \cdot \dfrac{c}{p}}.
\end{eqnarray*}
\end{theorem}

Theorem~\ref{maintheorem} was proved in \cite{Br03} for $r = p + 2$. Moreover, it was proved 
in \cite[Theorem 1]{Bhattacharya-Ghate} under a relatively strong assumption (in the present context), 
namely $\tau = v(c) =  \frac{1}{2}$ is as small as possible.
Indeed, condition ($\star$) imposed in \cite[Theorem 1.1]{Bhattacharya-Ghate} % when $v = \frac{3}{2}$ and $b = 3$ 
exactly says that the numerator of $c$ above has minimal possible valuation $3$. While the authors of \cite{Bhattacharya-Ghate} suspected 
that removing condition $(\star)$ would require new ideas, they did not perhaps appreciate how difficult removing this condition
would turn out to be.  An important psychological step was taken in \cite{BGR18}.
Indeed, many of the techniques used in this paper
to prove the tetrachotomy above for $v = \frac{3}{2}$ substantially develop techniques used in \cite{BGR18} to prove
the trichotomy for $v = 1$, although there are many additional complications.
In any case, since $\bar{V}_{k,a_p}$ was determined in \cite{Bhattacharya-Ghate} for all other slopes $v$ in $(1,2)$
(even for $p \geq 3$), we can finally state the following corollary.

\begin{cor}
  If $p \geq 5$, then the reduction $\bar{V}_{k,a_p}$ is known for all slopes $v = v(a_p)$ less than $2$. 
\end{cor}

%\subsection{Local Constancy vs Zig-Zag}

%Let us now explain the caveat  `except for cases coming from small weights and local constancy' in the statement
%of the zig-zag conjecture. Firstly, as Theorem~\ref{maintheorem} and the above historical discussion shows,
%there is no such caveat for the first few half-integral
%slopes $\frac{1}{2}$, $1$ and $\frac{3}{2}$. However, this caveat is required for slopes at least 2. In order to explain this, recall
%the following local constancy result of Berger \cite[Theorem B]{Berger12}.

\subsection{Proof of Theorem~\ref{maintheorem}} 
\label{sectionintroproof}

The proof  of Theorem~\ref{maintheorem} uses the compatibility of the $p$-adic and mod
$p$ Local Langlands Correspondences \cite{Br03}, with respect to the process of reduction \cite{B10}. This compatibility allows one to 
reduce the problem of computing $\bar{V}_{k,a_p}$ to a representation theoretic one, namely, to computing
the reduction of a $\mathrm{GL}_2(\Q_p)$-stable lattice in a 
certain unitary $\mathrm{GL}_2(\Q_p)$-Banach space. The key ingredients in the argument have been recalled several times in earlier 
works, so let us only recall the main steps here. 
%In doing so we will also see how the zig-zag conjecture acquires
%its name. 
%We must then solve this problem using some delicate harmonic analysis on the 
%Bruhat-Tits tree. 

Let $G =  \mathrm{GL}_2(\Q_p)$ and let $B(V_{k,a_p})$ be the unitary $G$-Banach space associated
to $V_{k,a_p}$ by the $p$-adic Local Langlands Correspondence. 
The reduction $\overline{B(V_{k,a_p})}^{ss}$ of a lattice 
in this Banach space coincides with the image of $\bar{V}_{k,a_p}$
under the (semisimple) mod $p$ Local Langlands Correspondence defined in \cite{Br03}. 
Since the mod $p$ correspondence is by definition injective, to compute $\bar{V}_{k,a_p}$ it suffices to compute the reduction 
$\overline{B(V_{k,a_p})}^{ss}$.

Let $K = \mathrm{GL}_2(\Z_p)$, a maximal compact open subgroup of $G =
\mathrm{GL}_2(\Q_p)$, and let $Z= \Q_p^\times$ be the center of $G$.  Let $X =
KZ \backslash G$ be the (vertices of the) Bruhat-Tits tree associated to
$G$.  The module $\mathrm{Sym}^r \bar\Q_p^2$, for $r = k-2$, carries a
natural action of $KZ$, and the projection

$$
KZ \backslash ( G \times   {\mathrm {Sym}}^{k-2} \bar\Q_p^2 ) \rightarrow KZ \backslash G = X
$$
defines a  local system on $X$. The space 
\begin{eqnarray*}  
   {\mathrm{ind}}_{KZ}^{G} \> {\mathrm {Sym}}^{k-2} \bar\Q_p^2
\end{eqnarray*}
consisting of all sections $f : G \rightarrow \mathrm{Sym}^{k-2}
\bar\Q_p^2$ of this local system which are compactly supported mod $KZ$ form a representation space for $G$,
which is equipped with a $G$-equivariant Hecke operator $T$. 
Let $\Pi_{k,a_p}$ be the locally algebraic representation of
$G$ defined by taking the cokernel of $T-a_p$ acting on the above space of sections.
Let $\Theta_{k,a_p}$ be the image of the integral sections
${\mathrm{ind}}_{KZ}^{G} \> {\mathrm {Sym}}^{k-2} \bar\Z_p^2$ in
$\Pi_{k,a_p}$. 
Then $B(V_{k,a_p})$ is the completion $\hat{\Pi}_{k,a_p}$ of
$\Pi_{k,a_p}$ with respect to the lattice $\Theta_{k,a_p}$.
The completion $\hat{\Theta}_{k,a_p}$, and sometimes by
abuse of notation $\Theta_{k,a_p}$ itself, is called the standard lattice
in $B(V_{k,a_p})$.  We have $\overline{B(V_{k,a_p})}^{ss} \cong
\bar{\hat{\Theta}}_{k,a_p}^{ss} \cong \bar\Theta_{k,a_p}^{ss}$.

Thus, to compute $\bar{V}^{}_{k,a_p}$, it suffices to compute the reduction 
$\bar\Theta_{k,a_p}^{ss}$ of $\Theta_{k,a_p}$.
In order to do this, we need a bit more notation. 
Let $V_r$ denote the $(r+1)$-dimensional $\bar{\mathbb{F}}_p$-vector space of 
homogeneous polynomials $P(X,Y)$ in two variables $X$ and $Y$ of degree $r$ over
$\bar{\mathbb{F}}_p$. 
The group $\Gamma=\mathrm{GL}_2(\mathbb{F}_p)$
acts on $V_r$ by the formula 
$\left(\begin{smallmatrix} a & b\\
c & d
\end{smallmatrix}\right)\cdot P(X,Y)=P(aX+cY, bX+dY)$, and $KZ$ acts on $V_r$ via projection
to $\Gamma$, with $\left( \begin{smallmatrix} p & 0 \\ 0 & p \end{smallmatrix} \right) \in
Z$ acting trivially. By definition of the lattice $\Theta_{k,a_p}$, there is a surjection 
${\mathrm{ind}}_{KZ}^{G} \> {\mathrm {Sym}}^{k-2} \bar\Z_p^2 \twoheadrightarrow \Theta_{k,a_p}$,
which induces a surjective map
\begin{eqnarray}
   \label{bartheta}
   \mathrm{ind}_{KZ}^G V_r \twoheadrightarrow\bar{\Theta}_{k,a_p},
\end{eqnarray}
for $r = k-2$.
Thus to compute $\bar\Theta_{k,a_p}$  one needs to understand the kernel of the map \eqref{bartheta}.
%Let $X_r \subset V_r$ denote the
%$\bar{\mathbb{F}}_p[\Gamma]$-span of the monomial $X^r$.
%%
Let $\theta(X,Y)=X^pY-XY^p$.
The action of $\Gamma$ on $\theta$ is via the determinant $D : \Gamma \rightarrow \mathbb{F}_p^\times$.
Define the following $\Gamma$- (hence $KZ$-) submodules of $V_r$. First, let 
%$V_r^* = \{P(X,Y) \in V_r :  \theta  | P \}$ be the submodule of $V_r$ consisting of all polynomials divisible by the $\theta$-polynomial.
%Similarly let 
$V_r^{**} = \{P(X,Y) \in V_r :  \theta^2  | P \}$ be the $\Gamma$-submodule of $V_r$ consisting of all polynomials divisible by $\theta^2$.
%$$
%V_r^{(i)} \sim \begin{cases}
%                                          0, &\text{if } r< i(p+1)\\
%                                          V_{r-i(p+1)}\otimes D^i, &\text{if } r\geq i(p+1),
%                                         \end{cases}
%$$
%for $i \geq 0$. 
Second, let $X_{r-1}$ be the $\Gamma$-submodule of $V_r$ generated by $X^{r-1}Y$. 
These submodules are important in computing the kernel of \eqref{bartheta} because of the following two useful facts
%We set $X_r^*:=X_r\cap V_r^*$ and
%$X_r^{**}=X_r\cap V_r^{**}$.  
%The mod $p$ reduction $\overline{\Theta}_{k,a_p}$ of the lattice
%$\Theta_{k,a_p}$ is a quotient of $\mathrm{ind}_{KZ}^G V_r$, for $r=k-2$. Indeed by
\cite[Remark 4.4]{BG09}: if the slope $v < 2$, then $\mathrm{ind}_{KZ}^G V_r^{**}$ lies in
the kernel of \eqref{bartheta}, and if the slope $v > 1$ and $r \geq 2p+1$,
then $\mathrm{ind}_{KZ}^G X_{r-1}$ lies in
the kernel of \eqref{bartheta}. Thus, if $1 < v < 2$ and $r \geq 2p+1$ (this is not a restriction, since 
as mentioned above, $r = p+2$ was treated in \cite{Br03}),  the surjection \eqref{bartheta}
factors through the 
map $$\mathrm{ind}_{KZ}^G \> Q \twoheadrightarrow\bar{\Theta}_{k,a_p},$$
where
\begin{align*}
Q = \frac{V_r}{X_{r-1} + V_r^{**}},
\end{align*}
for $r = k-2$. The module $Q$ was studied in some detail in \cite{Bhattacharya-Ghate}, and it is known that $Q$ has at most three Jordan-H\"older (JH) factors, which we call $J_1$, $J_2$ and $J_3$ (though these factors were called $J_2$, $J_0$ and $J_1$, respectively, in 
\cite{Bhattacharya-Ghate}).

Now, the mod $p$ Local Langlands Correspondence (mod $p$ LLC) 
says (roughly) that {\it irreducible} Galois representations 
$\bar{V}_{k,a_p}$ correspond to supersingular representations of the form $\frac{\ind_{KZ}^G J}{T}$, for some irreducible 
$\Gamma$-module $J$, whereas {\it reducible} $\bar{V}_{k,a_p}$ correspond (generically) to a sum of two principal series representations of
the form $\frac{\ind_{KZ}^G J}{T-\lambda}$ and $\frac{\ind_{KZ}^G J'}{T-\lambda^{-1}}$, for some (possibly equal) `dual' 
irreducible $\Gamma$-modules $J$, $J'$,
the sum of whose dimensions is $p-1$ mod $(p-1)$, and some $\lambda \in \bar\F_p^\times$.  

Thus,  exactly one, or possibly exactly two, of the above JH factors $J_i$, $i = 1$, $2$, $3$, contribute to $\bar\Theta_{k,a_p}$. 
Now,
the sum of the dimensions of $J_{2}$ and $J_{3}$ is $p+1$, and by the mod $p$ LLC, each of them gives the same 
irreducible Galois representation $\bar{V}_{k,a_p}$ when it occurs as the sole contributing factor 
to a supersingular $\bar\Theta_{k,a_p}$. Also,
%an {\it irreducible} Galois representation $\bar{V}_{k,a_p}$.
the sum of the dimensions of the pair $(J_{1}, J_{2})$ is $p-1$ 
so a potential `duality' occurs and these two JH factors may contribute together to $\bar\Theta_{k,a_p}$ giving a reducible Galois representation 
$\bar{V}_{k,a_p}$. Finally, it turns out that the last JH factor $J_3 = V_{p-2} \otimes D^2$ is potentially `self-dual', noting 
that twice its dimension is still $p-1$ mod $(p-1)$, so that again it may give rise to a reducible Galois representation $\bar{V}_{k,a_p}$.

We are now ready to make some key observations.
We claim that as $\tau$ varies through the rational line, the JH factors that contribute to $\bar\Theta_{k,a_p}$ actually 
occur in the following order: first $J_1$ contributes, then both $J_1$ and $J_2$ contribute together,
then only $J_2$ contributes, then only $J_3$  contributes, then finally $J_3$ contributes together with itself in a self-dual way. 
% (which is possible since the dimension of $J_3$ is $p-1$).  
Moreover, we claim that the jumps when two JH factors contribute as in the two cases described above
occur when $\tau$ takes specific integral values, namely, $\tau = t$ and $\tau = t+1$, respectively.
These claims are remarkable considering that {\it a priori} there is 
no reason to expect that there should be any patterns in the way $\bar\Theta_{k,a_p}$ `selects' JH factors. Indeed this
selection of these JH factors seems to be the beginning of a more general conjectural zig-zag pattern among the JH factors 
(see \cite[Conjecture 1.1]{G19}). In the present case, it explains
why the reduction $\bar{V}_{k,a_p}$ alternates between irreducible and reducible possibilities, with the reducible possibilities occurring 
exactly at the integer $\tau = t$ and for $\tau \geq t+1$.

All of this is best summarized with a picture. Let $F_i$ be the subquotient of $\bar\Theta_{k,a_p}$ occurring
as the image of $\ind_{KZ}^G J_i$, for $i = 1, 2, 3$. Then we prove that all the subquotients $F_i$ vanish in $\bar\Theta_{k,a_p}$, {\it except} for the $F_i$ occurring for $\tau$  in the following regions:

%\vskip 0.5 cm
%{\tiny
%
%\begin{xymatrix}{
%                        &  F_1   &  (F_1,F_2)  \ar@{-}[d]    &    F_2 \text{ or } F_3  &   (F_3,F_4)   \ar@{-}[d]   &  \cdots   &  {\begin{cases} (F_{b}, F_b) \\ (F_{b-1}, F_b) \end{cases}}  \ar@{-}[d]   & {  \begin{cases} (F_b,F_b),  & b=2n-1  \\ F_b, & b=2n  \end{cases} }    & &  \\ 
%    \ar@{<->}[rrrrrrrr]  &  & \ar@{-}[d]  & &   \ar@{-}[d] & &   \ar@{-}[d]  &  & \tau  \\
%        &      & t  &  & t+1 &   \cdots & t+n-1, \> n \geq 1 &  & &  \\ }  
%\end{xymatrix}
% 
%\vskip 0.5 cm
%}

%Proving this expectation is currently out of reach for general half-integral $v$, and in any case was not the goal of this paper, 
%but for $v = \frac{3}{2}$ and $b = 3$, we can show that the subquotients $F_i$, for $i =1$, $2$, $3$, of $\bar\Theta_{k,a_p}$ do occur in the %above order. 
%Indeed, the computations of this paper show that the following (slightly more precise) picture
%holds:
\vskip 0.5 cm

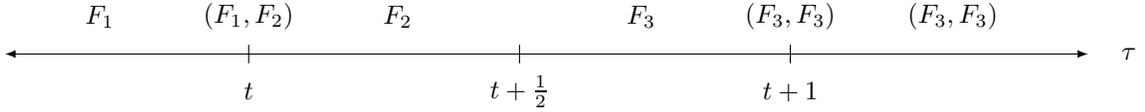
\begin{tikzpicture}[xscale = 1.8, auto=center][extra thick]
\draw [latex-latex] (0,0) -- (8,0);
\foreach \x in {1.8,3.8,5.8}
\draw[shift={(\x,0)},color=black] (0pt,3.5pt) -- (0pt,-3.5pt);
\node at(.7,0.5) {$F_1$};
\node at (1.8,0.5) {$(F_1,F_2)$};
\node at (2.9,0.5){$F_2$};
\node at (4.7,0.5) {$F_3$};
\node at (5.8,0.5) {$(F_3,F_3)$};
\node at (7,0.5){$(F_3,F_3)$};
\node at (1.8,-0.5) {$t$};
\node at (3.8, -0.5){$t+ \frac{1}{2}$};
\node at (5.8,-0.5){$t+1$};
\node at (8.3, 0) {$\tau$};
\end{tikzpicture}

%\vskip 0.5 cm
%{\tiny
%\qquad 
%\begin{xymatrix}{
%                        &  F_1   &  (F_1,F_2)  \ar@{-}[d]    &    F_2  & \ar@{--}[d]   &  F_3  &   (F_3,F_3)   \ar@{-}[d]   & (F_3, F_3)   & &  \\ 
%    \ar@{<->}[rrrrrrrr]  &  & \ar@{-}[d]  & &   \ar@{-}[d] & &   \ar@{-}[d]  &  &  \tau & \\
%       &      & t  &  & t+\frac{1}{2} &     & t+1 &  & &  \\ }  
%\end{xymatrix}
% 
%\vskip 0.5 cm
%}
\noindent More formally, we prove the following key symmetric nine-part proposition:
\begin{prop}
  \label{prop-nine-part}
  Assume $v = \frac{3}{2}$ and $b = 3$. Then the subquotients $F_i$ of $\bar\Theta_{k,a_p}$ satisfy the following 
  \begin{enumerate}
  \item Around $t$:
  \begin{itemize}
    \item $\tau > t \implies F_1 = 0$
    \item $\tau = t \implies F_1 \twoheadleftarrow \dfrac{\ind J_1}{T - \lambda_1^{-1}}$ and 
                                      $F_2 \twoheadleftarrow \dfrac{\ind J_2}{T - \lambda_1}$,
                                       with $\lambda_1  = \overline{\dfrac{b}{b-r} \cdot c }$
                                        %\begin{cases} 
                                        %       F_1 \twoheadleftarrow \dfrac{\ind J_1}{T - \lambda_1^{-1}}  \\
                                        %       F_2 \twoheadleftarrow \dfrac{\ind J_2}{T - \lambda_1}  
                                        %\end{cases}$, 
    \item $\tau < t \implies F_2 = 0$,
 \end{itemize}
 \item Around $t + \frac{1}{2}$:
 \begin{itemize}
    \item $\tau > t + \frac{1}{2} \implies F_2 = 0$
    \item $\tau = t + \frac{1}{2} \implies F_2 \twoheadleftarrow \dfrac{\ind J_2}{T}$ and $F_3 \twoheadleftarrow \dfrac{\ind J_3}{T}$
                                         %\begin{cases} 
                                          %    F_2 \twoheadleftarrow \dfrac{\ind J_2}{T}  \\
                                           %   F_3 \twoheadleftarrow \dfrac{\ind J_3}{T}  
                                          %\end{cases}$ 
    \item $\tau < t + \frac{1}{2} \implies F_3 = 0,$\footnote{In fact, we can only prove this for $\tau \leq t$, but this suffices.}
  \end{itemize}  
  \item Around $t + 1$:
  \begin{itemize}
    \item $\tau > t + 1 \implies  F_3 \twoheadleftarrow \dfrac{\ind J_3}{T^2+1}$
    \item $\tau = t + 1 \implies  F_3 \twoheadleftarrow \dfrac{\ind J_3}{T^2-dt+1}$,
                                               with % $\lambda_2 + \frac{1}{\lambda_2}  = 
                                                      $d =   \overline{\dfrac{b-1}{(b-1-r)(b-r)} \cdot \dfrac{c}{p}}$.   
    \item $\tau < t + 1 \implies F_3  \twoheadleftarrow \dfrac{\ind J_3}{T}$,
  \end{itemize}  
  \end{enumerate} 
  where $\ind = \ind_{KZ}^G$.                                                                 
\end{prop}

The main theorem of this paper, Theorem~\ref{maintheorem}, follows immediately from Proposition~\ref{prop-nine-part} and the fact that
$\bar\Theta_{k,a_p}$ corresponds to $\bar{V}_{k,a_p}$ under the mod $p$ LLC. Checking each of the nine statements in the proposition requires a substantial amount of work\footnote{The fifth statement actually follows easily from the mod $p$ LLC and the ninth statement.} 
involving explicit computations
with the Hecke operator on various polynomial valued functions on $G$. The hardest part of the argument is coming up with the
function in the first place. Once the functions are found, it is an elementary though lengthy process to cross-check that the computations
involving the Hecke operator are correct.
In this ({\tt arXiv}) version of the paper (where space is less of a constraint), we have decided to provide complete details 
for the benefit of the interested reader.

We end with an outline of the paper.  In Section~\ref{sectionbasics} we recall some basics facts.
In Section~\ref{sectioncomb} we prove 
some combinatorial identities and in Section~\ref{sectiontelescope} we prove several
useful telescoping lemmas in order to deal with the action of the Hecke operator at `infinity'. With these tools in hand,
the statements about $F_1$, $F_2$ and $F_3$ in the proposition above are proved in 
Sections~\ref{sectionF1}, \ref{sectionF2}  and \ref{sectionF3}, respectively. This then completes the proof of Proposition~\ref{prop-nine-part}
and hence the proof of Theorem~\ref{maintheorem}.

\section{Basics}
\label{sectionbasics}

In this section, we recall some notation and well-known facts, see \cite{Br03}, \cite{Bhattacharya-Ghate}, \cite{BGR18}.

\subsection{Hecke operator $T$ }
  \label{Hecke}

Let $G = \mathrm{GL}_2(\Q_p)$, $K = \mathrm{GL}_2(\Z_p)$ be the standard
maximal compact subgroup of $G$ and $Z = \Q_p^\times$ be the center of $G$.
Let $R$ be a $\Z_p$-algebra and let $V = \Sym^r R^2\otimes D^s$ be the
usual symmetric power representation of $KZ$ twisted by a power of the
determinant character $D$, modeled on homogeneous polynomials of degree
$r$ in the variables $X$, $Y$ over $R$. We require that $p \in Z$ 
%$\left( \begin{smallmatrix} p & 0 \\ 0 & p \end{smallmatrix} \right) \in Z$
acts trivially.
We will denote $\mathrm{ind}_{KZ}^{G}$ 
to mean compact induction. Thus $\mathrm{ind}_{KZ}^{G} V$ consists of functions
$f : G \rightarrow V$ such that $f(hg) = h \cdot f(g)$, for all $h \in KZ$
and $g \in G$, and $f$ is compactly supported mod $KZ$. Recall that $G$ acts on such
functions by right translation: $(g' \cdot f)(g) = f(gg')$, for $g$, $g' \in G$. 
For $g \in G$, $v \in V$, let
$[g,v] \in \mathrm{ind}_{KZ}^{G} V$ be the function with support in
${KZ}g^{-1}$ given by 
  $$g' \mapsto
     \begin{cases}
         g'g \cdot v,  \ & \text{ if } g' \in {KZ}g^{-1} \\
         0,                  & \text{ otherwise.}
      \end{cases}$$ 
One checks that $g' \cdot [g,v] = [g'g, v]$ and $[gh,v] = [g, h \cdot v]$, for all $g$, $g' \in G$, $h \in KZ$, $v \in V$.      
Any function in $\mathrm{ind}_{KZ}^G V$ is a finite linear combination of functions of the form $[g,v]$, for $g \in G$ and $v\in V$.  
The Hecke operator $T$ is defined by its action on these elementary functions via

\begin{equation}\label{T} T([g,v(X,Y)])=\underset{\lambda\in\F_p}\sum\left[g\left(\begin{smallmatrix} p & [\lambda]\\
                                                 0 & 1
                                                \end{smallmatrix}\right),\:v\left(X, -[\lambda]X+pY\right)\right]+\left[g\left(\begin{smallmatrix} 1 & 0\\
                                                                                                                                                   0 & p
                                                \end{smallmatrix}\right),\:v(pX,Y)\right],\end{equation}
 where $[\lambda]$ denotes the Teichm\"uller representative of $\lambda\in\F_p$.                                               
We may write $T = T^+ + T^-$, where 
\begin{eqnarray}
  T^+([g,v(X,Y)]) & = & \underset{\lambda\in\F_p}\sum\left[g\left(\begin{smallmatrix} p & [\lambda]\\
                                                                                                                                   0 & 1
                                                \end{smallmatrix}\right),\:v\left(X, -[\lambda]X+pY\right)\right],  \label{T^+} \\                                              
  T^-([g,v(X,Y)]) & = &    \left[g\left(\begin{smallmatrix} 1 & 0\\
                                                                                      0 & p
                                                \end{smallmatrix}\right),\:v(pX,Y)\right]. \label{T^-}
\end{eqnarray}

For $m = 0$, set $I_0 = \{0\}$, and 
for $m >0$, let
$I_m = \{ [\lambda_0] + [\lambda_1] p + \cdots + [\lambda_{m-1}]p^{m-1}  \> : \>  \lambda_i \in \F_p \} 
              \subset \Z_p$,
where the square brackets denote Teichm\"uller representatives. For $m \geq 1$, 
there is a truncation map
$[\quad]_{m-1}: I_{m} \rightarrow I_{m-1}$ given by taking the first $m-1$ terms in the $p$-adic expansion above;
for $m = 1$, $[\quad]_{m-1}$ is the $0$-map.
Let $\alpha =  \left( \begin{smallmatrix} 1 & 0 \\ 0  & p \end{smallmatrix} \right)$. 
For $m \geq 0$ and $\lambda \in I_m$, let
\begin{eqnarray*}
  g^0_{m, \lambda} =  \left( \begin{smallmatrix} p^m & \lambda \\ 0 & 1 \end{smallmatrix} \right) & \quad \text{and} \quad 
  g^1_{m, \lambda} = \left( \begin{smallmatrix} 1 & 0  \\ p \lambda  & p^{m+1} \end{smallmatrix} \right),
\end{eqnarray*}
noting that $g^0_{0,0}=\mathrm{Id}$ is the identity matrix and $g_{0,0}^1=\alpha$ in $G$. % and
Recall the decomposition
\begin{eqnarray*}
    G & = & \coprod_{\substack{m\geq 0,\,\lambda \in I_m,\\ i\in\{0,1\}}} {KZ} (g^i_{m, \lambda})^{-1}.
\end{eqnarray*}
Thus a general element in $\mathrm{ind}_{KZ}^G V$ is a finite sum of functions of the form $[g,v]$, 
with $g=g_{m,\lambda}^0$ or $\,g_{m,\lambda}^1$, for some $\lambda\in I_m$ and $v\in V$.
%For a $\Z_p$-algebra $R$, let $v = \sum_{i=0}^r c_i X^{r-i} Y^i \in V = \mathrm{Sym}^r R^2\otimes D^s$.
%One may make the formulas \eqref{T^+} and \eqref{T^-} quite explicit. We have:
%\begin{eqnarray*}
%  T^+([g^0_{n,\mu},v]) & = & \sum_{\lambda \in I_1} \left[ g^0_{n+1, \mu +p^n\lambda},    
%         \sum_{j=0}^r \left( p^j \sum_{i=j}^r c_i \binom{i}{j}(-\lambda)^{i-j} \right) X^{r-j} Y^j \right], \\
%  T^-([g^0_{n,\mu},v]) & = & \left[ g^0_{n-1, [\mu]_{n-1}},    
%         \sum_{j=0}^r \left( \sum_{i=j}^r p^{r-i} c_i {i \choose j} 
%         \left( \frac{\mu - [\mu]_{n-1}}{p^{n-1}} \right)^{i-j} \right) X^{r-j} Y^j \right] \quad (n > 0), \\
%  T^-([g^0_{n,\mu},v]) & = &  [ \alpha,  \sum_{j=0}^r  p^{r-j}  c_j  X^{r-j} Y^j ] \quad  (n=0). 
%\end{eqnarray*}
%These formulas for $T^+$ and $T^-$ will often be used to compute $(T-a_p)f$, for $f \in \mathrm{ind}_{KZ}^G \Sym^r\bar\Q_p^2$.

\subsection{The mod $p$ Local Langlands Correspondence}
\label{subsectionLLC}

For $0 \leq r \leq p-1$, $\lambda \in \bar{\F}_p$ and $\eta : \Q_p^\times
\rightarrow \bar\F_p^\times$ a smooth character, let
\begin{eqnarray*}
  \pi(r, \lambda, \eta) & := & \frac{\mathrm{ind}_{KZ}^{G} \:\Sym^r\bar\F_p^2}{T-\lambda} \otimes (\eta\circ \mathrm{det})
\end{eqnarray*}
be the smooth admissible representation of ${G}$, %where $\mathrm{ind}_{KZ}^G$ is compact induction, and  $T$ is the Hecke operator. 
known to be irreducible unless $(r,\lambda)=(0,\pm 1)$ or $(p-1,\pm 1)$, by the classification of irreducible  representations  of $G$ in characteristic $p$ in \cite{BL94, BL95, Breuil03a}.
Breuil's semisimple mod $p$ Local Langlands Correspondence %(see, e.g., 
\cite[Def. 1.1]{Br03} %)
is given by:
\begin{itemize} 
  \item  $\lambda = 0$: \quad
             $\mathrm{ind}(\omega_2^{r+1}) \otimes \eta \:\:\overset{LL}\longmapsto\:\: \pi(r,0,\eta)$,
  \item $\lambda \neq 0$: \quad
             $\left( \mu_\lambda\cdot \omega^{r+1}  \oplus \mu_{\lambda^{-1}} \right) \otimes \eta 
                   \:\:\overset{LL}\longmapsto\:\:  \pi(r, \lambda, \eta)^{ss} \oplus  \pi([p-3-r], \lambda^{-1}, \eta \omega^{r+1})^{ss}$, 
\end{itemize}
where $\{0,1, \ldots, p-2 \} \ni [p-3-r] \equiv p-3-r \mod (p-1)$.  For a more functorial description, see \cite{Col}.

Recall that there is a locally algebraic representation of $G$ given by %compact induction
\begin{eqnarray*} 
\Pi_{k, a_p} = \frac{ \mathrm{ind}_{{KZ}}^{G}
\Sym^r \bar\Q_p^2 }{T-a_p}, 
\end{eqnarray*} 
where $r=k-2 \geq 0$ and $T$ is the
Hecke operator. Consider the standard lattice in $\Pi_{k,a_p}$ given by
\begin{equation} 
\label{definetheta}  
\Theta=\Theta_{k, a_p} :=
\mathrm{image} \left( \mathrm{ind}_{KZ}^{G} \Sym^r \bar\Z_p^2 \rightarrow
\Pi_{k, a_p} \right) \iso \frac{ \mathrm{ind}_{KZ}^{G} \Sym^r \bar\Z_p^2
}{(T-a_p)(\mathrm{ind}_{KZ}^{G} \Sym^r \bar\Q_p^2) \cap
\mathrm{ind}_{KZ}^{G} \Sym^r \bar\Z_p^2 }.  
\end{equation} 
It is known  \cite{B10}
that the semisimplification of the reduction of this lattice satisfies
$\bar\Theta_{k,a_p}^\mathrm{ss} \iso LL(\bar{V}_{k,a_p}^{ss})$, where $LL$ is
the (semisimple) mod $p$ Local Langlands Correspondence above.
Since the map ${LL}$ is clearly injective, it is enough to know ${LL}(\bar V_{k,a_p}^{ss})$ to determine 
$\bar{V}_{k,a_p}^{ss}$.

\subsection{The structure of the quotient $Q$}

%\begin{prop}\label{es1}
% Let $p\geq 3$, and $r\equiv  b \mod (p-1)$, with $1\leq b \leq p-1$. 
%  \begin{enumerate}
%      \item[(i)] For $r\geq p$, the $\Gamma$-module structure of $V_r/V_r^*$ is given by
%               \begin{eqnarray*}
%                0\rightarrow V_{b}\rightarrow\frac{V_r}{V_r^*}\rightarrow V_{p-b-1}\otimes D^{b}\rightarrow 0,
%                \end{eqnarray*} 
%             and the  sequence splits  if and only if $b=p-1$.
%       \item[(ii)] For $r\geq 2p+1$, the $\Gamma$-module structure of $V_r^*/V_r^{**}$ is given by
%              \begin{eqnarray*}
%                0\rightarrow V_{p-2}\otimes D\rightarrow\frac{V_r^*}{V_r^{**}}\rightarrow V_{1}\rightarrow 0, & & \text{ if } b=1,\\
%                0\rightarrow V_{p-1}\otimes D\rightarrow\frac{V_r^*}{V_r^{**}}\rightarrow V_{0}\otimes D\rightarrow 0, & &\text{ if } b=2,\\
%                0\rightarrow V_{b-2}\otimes D\rightarrow\frac{V_r^*}{V_r^{**}}\rightarrow V_{p-b+1}\otimes D^{b-1}\rightarrow 0, & & \text{ if } 3\leq b\leq p-1,
%               \end{eqnarray*} 
%              and the sequences above split  if and only if $b=2$.
%   \end{enumerate} 
%\end{prop}
%
%
%\begin{proof}
%  See \cite[Prop. 2.1, Prop. 2.2]{Bhattacharya-Ghate}.
%\end{proof}
%\subsection{The structure of $Q$}

Let $V_r =$ Sym$^r\bar{\mathbb{F}}_p^2$ be the usual symmetric power representation of $\Gamma:=$ GL$_2(\mathbb{F}_p)$ (hence $KZ$, with $p\in Z$ acting trivially).

%Recall \begin{align*}
%\Theta_{k,a_p} =  \frac{\text{ind}_{KZ}^G\text{Sym}^r\overline{\mathbb{Z}}_p^2}{(T-a_p)\text{ind}_{KZ}^G\text{Sym}^r\overline{\mathbb{Q}}_p^2\cap %\text{ind}_{KZ}^G\text{Sym}^r\overline{\mathbb{Z}}_p^2},
%\end{align*}  
%where $r = k-2$. 
It follows directly from the definition of $\Theta_{k,a_p}$ in \eqref{definetheta}, that there is a surjection 
ind$_{KZ}^G$Sym$^r\bar{\mathbb{Z}}_p^2 \twoheadrightarrow \Theta_{k,a_p}$, for $r =  k -2$, whence a surjective map 
\begin{align}
\text{ind}_{KZ}^G V_r \twoheadrightarrow \bar{\Theta}_{k,a_p}, \label{surjection from indV_r to Theta}
\end{align}
where $\bar{\Theta}_{k,a_p} = \Theta_{k,a_p} \otimes_{\bar{\mathbb{Z}}_p} \bar{\mathbb{F}}_p$.

Write $X_{k,a_p}$ for the kernel. A model for $V_r$ is the space of all homogeneous polynomials of degree $r$ in two variables $X$ and $Y$ over $\bar{\mathbb{F}}_p$ with the standard action of $\Gamma$, recalled in Section~\ref{sectionintroproof}. 
Let $X_{r-1} \subset V_r$ be the $\Gamma$- (hence $KZ$-) submodule generated by $X^{r-1}Y$. Let $V_r^*$ and $V_r^{**}$ be the submodules of $V_r$ consisting of polynomials divisible by $\theta$ and $\theta^2$ respectively, for $\theta = X^pY - XY^p$. If $r \geq 2p +1$, then Buzzard-Gee have shown [\cite{BG09}, Remark 4.4]:
\begin{itemize}
\item $v(a_p) > 1 \implies $ ind$_{KZ}^GX_{r-1} \subset X_{k,a_p}$,
\item $v(a_p) < 2 \implies $ ind$_{KZ}^GV_r^{**} \subset X_{k,a_p}$.
\end{itemize}
It follows that when $1<v(a_p)<2$, the map \eqref{surjection from indV_r to Theta} induces a surjective map ind$_{KZ}^G Q\twoheadrightarrow \bar{\Theta}_{k,a_p}$, where 
\begin{align*}
Q:= \frac{V_r}{X_{r-1} + V_r^{**}}.
\end{align*}
Note that unlike $V_r$ the length of $Q$ as a $\Gamma$-module is bounded independently of $r$. 
The $\Gamma$-module structure of $Q$ has been derived in \cite[Proposition 6.4]{Bhattacharya-Ghate}. We have:

\begin{prop}\label{Structure of Q}
 Let $p\geq 5$, $r \geq 2p+1$ and $r\equiv 3~\emph{mod}~(p-1)$. % The $\Gamma$-module structure of $Q$ is as follows.
\begin{enumerate}[label = (\arabic{*})]
\item If $r \not\equiv 3$ \emph{mod} $p$, then there is an exact sequence
\begin{align}
0 \longrightarrow V_{p-2}\otimes D^2 \longrightarrow Q \longrightarrow V_{p-4}\otimes D^3\longrightarrow 0, \label{exact sequence t = 0}
\end{align}
and moreover the exact sequence above is $\Gamma$-split.
\item If $r\equiv 3$ \emph{mod} $p$, then
\begin{align}
0\longrightarrow V_r^*/V_r^{**}\longrightarrow Q \longrightarrow V_{p-4}\otimes D^3 \longrightarrow 0, \label{exact sequence t >= 1}
\end{align}
where $ V_r^*/V_r^{**}$ is the non-trivial extension of $V_{p-2}\otimes D^2$ by $V_1\otimes D$.
\end{enumerate}
\end{prop}

Let us now set some important notation. Let us denote the JH factors of $Q$ above as follows:
\begin{eqnarray*}
  J_1 & = & V_{p-4}\otimes D^3, \\
  J_2  & = & V_{1} \otimes D,  \\
  J_3 & = & V_{p-2}\otimes D^2. 
\end{eqnarray*}  
(These JH factors of $Q$ were called $J_2$, $J_0$, $J_1$, respectively, in \cite{Bhattacharya-Ghate} and \cite{BGR18}.) 

We now define some important subquotients of $\bar\Theta_{k,a_p}$. Let 
\begin{itemize}
  \item $F_2$  denote the image of ind$_{KZ}^GJ_2$ inside $\bar{\Theta}_{k,a_p}$, 
  \item $F_{2,3}$ denote the image of ind$_{KZ}^G(V_r^*/V_r^{**})$ inside $\bar{\Theta}_{k,a_p}$,  
  \item $F_3:= F_{2,3}/F_2$, and, 
  \item $F_1 := \bar{\Theta}_{k,a_p}/ F_{2,3}$.   
\end{itemize}
So we have a short exact sequence $$0 \rightarrow (F_2 \rightarrow F_{2,3} \rightarrow F_3)  \rightarrow \bar\Theta_{k,a_p}  \rightarrow F_1 \rightarrow 0. $$

Thus, to study the structure of 
$\bar\Theta_{k,a_p}$ up to semisimplification  it suffices to study the  
surjections  ind$_{KZ}^GJ_i \twoheadrightarrow F_i$
for $i = 1$, $2$, $3$.  This will be done in detail
in Sections~\ref{sectionF1},  \ref{sectionF2}, \ref{sectionF3}, respectively.
Note that even though $\bar{\Theta}_{k,a_p}$ is not necessarily semisimple, $F_1$, $F_2$ and 
$F_3$ are semisimple.

We end with the following useful lemma.

\begin{lemma}
 \label{generator}
Let $p\geq 3$, $r\geq 2p+1$, $ r\equiv b \mod (p-1)$, with $b = 3$.
 \begin{enumerate}

   \item[(i)] 
    The image of 
$X^{r-i}Y^{i}$ in $Q$ maps to 
$0 \in J_1$, %= V_{p-b-1}\otimes D^b$, 
for $0\leq i\leq b-1$, whereas the image of 
$X^{r-b}Y^{b}$ in $Q$ maps to 
$X^{p-b-1}$ in $J_1$.

   \item[(ii)] 
The image of 
$\theta X^{r-p-1}$, $\theta Y^{r-p-1}$ in $Q$ correspond to
$X^{b-2}$, $Y^{b-2}$, respectively, in $J_2$, so both map to $0$ in $J_3$.  %= V_{a-2}\otimes D$.

   \item[(iii)] 
The image of 
$\theta X^{r-p-b+1}Y^{b-2}$ in $Q$ maps to
$X^{p-b+1} \in J_3$. %= V_{p-a+1}\otimes D^{a-1}$.

  \end{enumerate}

%If $a = 2$, then:
% \begin{enumerate}  
%
%   \item[(i)] 
%The image of 
%$\theta X^{r-p-1}$ in $W_0$ maps to
%$X^{p-1} \in J_0$. %= V_{p-1}\otimes D$.
%
%
%   \item[(ii)] 
%For $r > 2p$,
%the image of 
%$\theta X^{r-2p}Y^{p-1}$ in $W_1$ maps to
%$1 \in J_1$. % = V_0\otimes D$.
%
%   \item[(iii)] The image of 
%$X^{r-i}Y^{i}$ in $ W_2$ maps to 
%$0 \in J_2$,  %  = V_{p-3}\otimes D^2$, 
%for $ i=0, 1$, whereas
%the image of 
%$X^{r-2}Y^{2}$ in $ W_2$ maps to 
%$X^{p-3} \in J_2$.
%\end{enumerate}

%If $a = 1$, then:
%\begin{enumerate}  
%
%   \item[(i)] 
%If $p\mid r$,
%the image of 
%$\theta X^{r-p-1}$ in $W_0$ maps to
%$X^{p-2} \in J_0$. %= V_{p-2}\otimes D$.
%
%
%   \item[(ii)] 
%The image of 
%$\theta X^{r-2p+1}Y^{p-2}$ in $W_1$ maps to
%$X \in J_1$. % = V_{1}$.
%
%
%   \item[(iii)] 
%The image of 
%$X^{r-1}Y$ in $W_2$ maps to
%$X^{p-2} \in J_2$. % = V_{p-2}\otimes D$.
%
%
%  \end{enumerate}
\end{lemma}

\begin{proof}
This is a special case of \cite[Lemma 8.5]{Bhattacharya-Ghate} and \cite[Lemma 3.4]{BGR18}, though the JH factors $J_1$, $J_2$, $J_3$ were called $J_2$, $J_0$, $J_1$, respectively in these papers.
The proof is elementary and consists of explicit calculations with the maps given in \cite[(4.2)]{G78} 
and \cite[Lem. 5.3]{Br03}.% as explained in the proof of \cite[Lem. 8.5]{Bhattacharya-Ghate}.
\end{proof}

\section{Combinatorial Identities}
\label{sectioncomb}

In this section, we state several combinatorial identities. 

We first state three lemmas on the congruence properties of binomial coefficients, which will later be useful in working
with the operator $T^-$, and to some extent with the operator $T^+$.

\begin{lemma}
\label{lemma 0.1.} 
Let $p>3.$ If $r\equiv 3~\emph{mod}~(p-1)$ and $t=v(r-3)$, 
then $p^i\binom{r}{i} \equiv 0~\emph{mod}~p^{t+4}$, for $i\geq 4$. 
\end{lemma}

\begin{proof}
We need to show that $v(p^i\binom{r}{i}) \geq t+4$ for all $i \geq 4$. 
%We will prove the lemma for $i=4,5$ and then for general $i\geq 6$. 
Now $$v(p^i\binom{r}{i})=i + v(r(r-1)\cdots(r-i+1))-v(i!).$$ Since $i\geq 4$, we get that $v(p^i\binom{r}{i})\geq i + t -v(i!)$. 

For $i =4$, we have $v(p^4\binom{r}{4})\geq 4+t-v(4!)$. Since $p\geq 5$, we get that $v\left(p^4\binom{r}{4}\right)\geq t+4$.
Similarly for $i=5$,  $v(p^5\binom{r}{5})\geq 5+t -v(5!)$. Since $v(5!)\leq 1$  for $p\geq 5$,  we get that $v\left(p^5\binom{r}{5}\right)\geq t+4$. 
Now assume $i\geq 6$. The difference $i-v(i!)\geq i-\left(\frac{i}{p}+\frac{i}{p^2}+\cdots\right) =
% $i-v(i!)\geq 
i\left(1-\frac{1}{p-1}\right) 
%Now $ i\left(1-\frac{1}{p-1}\right)
\geq 6\left(1-\frac{1}{4}\right)\geq 4$, because $i\geq 6$ and $p\geq 5$. Therefore again $i+t-v(i!)\geq t+4$.
\end{proof}

\begin{lemma}\label{lemma 0.2}
Let $p>3.$ If $r\equiv 2~\emph{mod}~(p-1)$ and $t=v(r-2)$, then $p^i\binom{r}{i} \equiv 0~\emph{mod} ~p^{t+3},$ for $i \geq 3.$
\end{lemma}
\begin{proof}
This is Lemma 2.6 in \cite{BGR18}.
\end{proof}

\begin{lemma}\label{lemma 0.3.}
Let $p \geq 3.$ If $r\equiv 1~\emph{mod}~(p-1)$ and $t=v(s-1)$, then $p^i\binom{r}{i}\equiv 0 ~\emph{mod}~p^{t+2},$ for $i \geq 2.$
\end{lemma}
\begin{proof}
This is Lemma 2.1 in \cite{BG13}.
\end{proof}

Second, we state several propositions
on the congruence properties of sums of products of binomial coefficients. These will be useful in computing with the 
operator $T^+$. The proofs are similar to the proofs of similar results in \cite{BGR18}, though we provide complete details 
of all new results here.

The first proposition will be used in the proof of Proposition~\ref{F_0 when tau >= t}.

\begin{prop}\label{proposition 0.4.}
Let $p>3.$ If $r = 3+ n(p-1)p^t$, with $t=v(r-3)$ and $n>0$, then we have:
\begin{enumerate}[label =(\arabic{*})]
\item $\begin{aligned}[t]
\sum_{\substack{{0<j<r} \\{j\equiv 3~\emph{mod}~(p-1)}}}\binom{r}{j} \equiv & \> \> \frac{3-r}{6(1-p)}\left(6p^2+5p-3\binom{2p+1}{p-1}\right)\\&+ \> \frac{1}{6}\left(\frac{3-r}{1-p}\right)^2\left(-3p^2-3p+3\binom{2p+1}{p-1}\right)\emph{mod}~p^{t+3}.
\end{aligned}$
\item $\begin{aligned}[t]
\sum_{\substack{{0<j<r} \\{j\equiv 3~\emph{mod}~(p-1)}}}j\binom{r}{j} &\equiv \frac{pr(3-r)}{2}~\emph{mod}~p^{t+2}.
\end{aligned}$
\item $\begin{aligned}[t]
\sum_{\substack{{0<j<r} \\{j\equiv 3~\emph{mod}~(p-1)}}}\binom{j}{2}\binom{r}{j} &\equiv 0 ~\emph{mod}~p^{t+1}.
\end{aligned}$
\item $\begin{aligned}[t]
\sum_{\substack{{0<j<r} \\{j\equiv 3~\emph{mod}~(p-1)}}}\binom{j}{3}\binom{r}{3} &\equiv \frac{\binom{r}{3}}{1-p}~\emph{mod}~p^t.
\end{aligned}$
\item $\begin{aligned}[t]
p^i\sum_{\substack{{0<j<r} \\{j\equiv 3~\emph{mod}~(p-1)}}}\binom{j}{i}\binom{r}{j} &\equiv 0 ~\emph{mod}~p^{t+4}, \forall i \geq 4.
\end{aligned}$
\end{enumerate}
\end{prop}
\begin{proof}
Let $0 \leq i < r$. Let $S_{i,r} = \sum\limits_{\substack{{0<j<r} \\{j\equiv 3~\text{mod}~(p-1)}}}\binom{j}{i}\binom{r}{j}$. Let $\sum_{i,r}= (p-1)\sum\limits_{\substack{{j\geq 0} \\{j\equiv 3~\text{mod}~(p-1)}}}\binom{j}{i}\binom{r}{j}$, then $\sum_{i,r} = (p-1)\left(S_{i,r} + \binom{r}{i}\right)$. Let $f_r(x)= (1+x)^r= \sum\limits_{ j\geq 0}\binom{r}{j}x^j$, which we consider as a function from $\mathbb{Z}_p$ to $\mathbb{Z}_p$. Let
\begin{align*}
g_{i,r}(x) &= \frac{x^{i-3}}{i!}f_r^{(i)}(x)\\
&= \sum_{j\geq 0}\binom{j}{i}\binom{r}{j}x^{j-3}. 
\end{align*}
Let $\mu_{p-1}$ be the set of $(p-1)$-st roots of unity. Summing $g_{i,r}(x)$ as $x$ varies over all $(p-1)$-st roots of unity we get
\begin{align*}
\sum_{\xi \in \mu_{p-1}}g_{i,r}(\xi) &= \sum_{\xi \in \mu_{p-1}}\sum_{j \geq 0}\binom{j}{i}\binom{r}{j}\xi^{j-3}\\
&= \sum_{\xi \in \mu_{p-1}}\left(\sum_{j\equiv 3 ~\text{mod}~(p-1)}\binom{j}{i}\binom{r}{j}\xi^{j-3} + \sum_{j\not\equiv 3~\text{mod}~(p-1)}\binom{j}{i}\binom{r}{j}\xi^{j-3}\right).
\end{align*}

We will be using the following easy fact quite frequently:
\begin{align} \label{sum of roots of 1}
\sum\limits_{\lambda\in\mathbb{F}_p^\times}[\lambda]^i =
\begin{cases}
p-1 &\text{if}~p-1~\text{divides}~i,\\
0 &\text{if}~(p-1)\nmid i,
\end{cases}
\end{align}
where $[\lambda]$ is the Teichm{\"u}ller representative of $\lambda\in\mathbb{F}_p$.

Using the above fact we get that $$\sum_{\xi\in\mu_{p-1}}g_{i,r}(\xi) = (p-1)\sum\limits_{\substack{{j\geq 0}\\{j\equiv 3~\text{mod}~(p-1)}}}\binom{j}{i}\binom{r}{j},$$ which is equal to $\sum_{i,r}$. Let $\mu'$ be the set of $(p-1)$-st roots of unity except $-1$. Since  $g_{i,r}(x)= x^{i-3}\binom{r}{i}(1+x)^{r-i}$, 
\begin{align*}
\textstyle\sum_{i,r} &= \sum_{\xi\in\mu_{p-1}}\binom{r}{i}\xi^{i-3}(1+\xi)^{r-i}\\
&= \binom{r}{i}\sum_{\xi\in\mu'}\xi^{i-3}(1+\xi)^{r-i}, ~\text{as}~ r-i > 0.
\end{align*}
Let $\xi\in\mu'$, clearly $1+\xi \in\mathbb{Z}_p$. Since $\xi$ is the unique Teichm{\"u}ler lift of some element of $\mathbb{F}_p^\times$ and the lift of $-1$ mod $p$ is $-1$, we see that $\overline{\xi} \not\equiv -1$ mod $p$. Since $1+\xi \not\in p\mathbb{Z}_p$, we have $1+\xi \in \mathbb{Z}_p^\times$. As $\mathbb{Z}_p^\times \cong \mu_{p-1} \times (1+p\mathbb{Z}_p)$, we can write $(1+\xi)^{p-1} = 1+pz_{\xi}$, for some $z_\xi \in \mathbb{Z}_p$. 

Computing $S_{0,r}$ mod $p^{t+3}$:
\begin{align*}
\textstyle\sum_{0,r} &= \sum_{\xi\in\mu'}\xi^{-3}(1+\xi)^r\\
&= \sum_{\xi\in\mu'}\xi^{-3}(1+\xi)^3(1+\xi)^{n(p-1)p^t}\\
&= \sum_{\xi\in\mu'}(1+\xi^{-1})^3(1+pz_\xi)^{np^t}.
\end{align*}
 Computing $\sum_{0,r}$ mod $p^{t+3}$ we get
 \begin{align}
\textstyle\sum_{0,r} &\equiv \sum_{\xi\in\mu'}(1+\xi^{-1})^3\left(1+ np^{t+1}z_{\xi}+np^{t+2}z_{\xi}^2\left(\frac{np^t-1}{2}\right)\right)~\text{mod}~p^{t+3}. \label{sum0r}
\end{align}

Now $\sum\limits_{\xi\in\mu'}(1+\xi^{-1})^3= \sum\limits_{\xi\in\mu'}(1+3\xi^{-1}+3\xi^{-2} + \xi^{-3})= (p-2) + 3 -3 +1= p-1$.

Specializing equation \eqref{sum0r} at $n=1,t=0$, so $r = p+2$, we get
\begin{align*}
\textstyle\sum_{0,p+2} &\equiv \sum_{\xi\in\mu'}(1+\xi^{-1})^3(1+pz_\xi)~\text{mod}~p^3\\
(p-1)\left(\binom{p+2}{3} + 1\right)&\equiv p-1 + p\sum_{\xi\in\mu'}(1+\xi^{-1})^3z_\xi~\text{mod}~ p^3\\
\frac{p(p^2-1)(p+2)}{6} &\equiv p\sum_{\xi\in\mu'}(1+\xi^{-1})^3z_{\xi}~\text{mod}~p^3.
\end{align*}
Simplifying, we get that
\begin{align}
p\sum\limits_{\xi\in\mu'}(1+\xi^{-1})^3z_{\xi}\equiv -\frac{p(p+2)}{6}~\text{mod}~p^3. \label{sum r,0 (a)}
\end{align}

Again specializing equation \eqref{sum0r} at $n=2,t=0$, so $r =2p+1$, we get
\begin{align*}
\textstyle\sum_{0,2p+1} &\equiv \sum_{\xi\in\mu'}(1+\xi^{-1})^3(1+2pz_{\xi} + p^2z_{\xi}^2)~\text{mod}~p^3\\
(p-1)\left(\binom{2p+1}{3}+\binom{2p+1}{p+2} + 1\right) &\equiv p-1 + 2p\sum_{\xi\in\mu'}(1+\xi^{-1})^3z_{\xi} + p^2\sum_{\xi\in\mu'}(1+\xi^{-1})^3z_\xi^2 ~\text{mod}~ p^3.
\end{align*}
Using equation \eqref{sum r,0 (a)} in the above expression, we get
\begin{align*}
(p-1)\binom{2p+1}{3} +(p-1)\binom{2p+1}{p+2} &\equiv \frac{-p(p+2)}{3} +p^2\sum_{\xi\in\mu'}(1+\xi^{-1})^3z_{\xi}^2~\text{mod}~ p^3.
\end{align*}
Simplifying above we get that
\begin{align} 
p^2\sum\limits_{\xi\in\mu'}(1+\xi^{-1})^3z_{\xi}^2 \equiv p + (p-1)\binom{2p+1}{p+2}~\text{mod}~p^3. \label{sum r,0 (b)}
\end{align}

Now using equations \eqref{sum r,0 (a)} and \eqref{sum r,0 (b)} in equation \eqref{sum0r}, we get
\begin{align*}
%(p-1)(1+S_{0,r}) &\equiv \sum_{\xi\in\mu'}\left((1+\xi^{-1})^3+np^t(1+\xi^{-1})^3pz_{\xi}\right.\\
%&+ \left. \frac{np^t(np^t-1)}{2}p^2(1+\xi^{-1})^3z_{\xi}^2\right)~\text{mod}~p^{t+3}.\\
%\text{ We get}~
  (p-1)(1+S_{0,r}) \equiv & \> p-1 +np^t\left(\frac{-p(p+2)}{6}\right)\\
  & + \frac{np^t(np^t-1)}{2}\left(p +(p-1)\binom{2p+1}{p+2}\right) ~\text{mod}~p^{t+3}. 
\end{align*}
So
\begin{align*}
(p-1)S_{0,r} &\equiv \frac{-np^{t+1}(p+2)}{6} + \left(\frac{n^2p^{2t}-np^t}{2}\right)\left(p+ (p-1)\binom{2p+1}{p+2}\right)~\text{mod}~p^{t+3}.
\end{align*}
Replacing $np^t$ by $\frac{r-3}{p-1}$, we get
\begin{align*}
(p-1)S_{0,r} \equiv & \> -\frac{1}{6}\left(\frac{r-3}{p-1}\right)\left(p^2+ 5p+ 3(p-1)\binom{2p+1}{p+2}\right)\\
& +\frac{1}{2}\left(\frac{r-3}{p-1}\right)^2\left(p+ (p-1)\binom{2p+1}{p+2}\right)~\text{mod}~p^{t+3}.
\end{align*}
Finally, dividing by $(p-1)$, we get
\begin{align*}
S_{0,r} &\equiv \frac{1}{6}\left(\frac{3-r}{1-p}\right)\left(5p + 6p^2 -3\binom{2p+1}{p+2}\right)\\
& +\frac{1}{6}\left(\frac{3-r}{1-p}\right)^2\left(-3p-3p^2 + 3\binom{2p+1}{p+2}\right)~\text{mod}~p^{t+3},
\end{align*}
proving part \textit{(1)}.

Computing $S_{1,r}$ mod $p^{t+2}$:
\begin{align*}
\textstyle\sum_{1,r} &= \binom{r}{1}\sum_{\xi\in\mu'}\xi^{-2}(1+\xi)^{r-1}\\
&= r\sum_{\xi\in\mu'}\xi^{-2}(1+\xi)^2(1+\xi)^{n(p-1)p^t}\\
&= r\sum_{\xi\in\mu'}(1+\xi^{-1})^2(1+pz_\xi)^{np^t}.
\end{align*}
Computing $\sum_{1,r}$ mod $p^{t+2}$, we get
\begin{align} \textstyle\sum_{1,r} &\equiv r\sum\limits_{\xi\in\mu'}(1+\xi^{-1})^2(1+np^{t+1}z_{\xi})~\text{mod}~p^{t+2}. \label{sum1r}
\end{align}

Now $\sum\limits_{\xi\in\mu'}(1+\xi^{-1})^2= \sum\limits_{\xi\in\mu'}(1+2\xi^{-1}+ \xi^{-2}) =(p-2) +2 -1= p-1$.

Specializing equation \eqref{sum1r} at $n=1$, $t= 0$, so $r =p+2$, we get
 \begin{align*}
\textstyle\sum_{1,p+2} &\equiv (p+2)\left(\sum_{\xi\in\mu'}(1+\xi^{-1})^2+p\sum_{\xi\in\mu'}(1+\xi^{-1})^2z_{\xi}\right)~\text{mod}~p^2\\
 (p-1)\left(3\binom{p+2}{3} + p+2\right) &\equiv (p+2)\left(p-1 + p\sum_{\xi\in\mu'}(1+\xi^{-1})^2z_{\xi}\right)~\text{mod}~p^2.
 \end{align*}
Simplifying, we get
\begin{align}
 p\sum\limits_{\xi\in\mu'}(1+\xi^{-1})^2z_{\xi}\equiv \frac{-p}{2}~\text{mod}~p^2. \label{sum r,1 (a)}
\end{align}
Now using \eqref{sum r,1 (a)} in equation \eqref{sum1r}, we get
\begin{align*}
(p-1)(S_{1,r}+r) &\equiv r(p-1)-\frac{nrp^{t+1}}{2}~\text{mod}~p^{t+2}, 
\end{align*}
which is the same as $S_{1,r} \equiv \frac{pr(3-r)}{2}~\text{mod}~p^{t+2}$,
proving part \textit{(2)}.

Computing $S_{2,r}$ mod $p^{t+1}$:
\begin{align*}
\textstyle\sum_{2,r} &= \binom{r}{2}\sum_{\xi\in\mu'}\xi^{-1}(1+\xi)^{r-2}\\
&= \binom{r}{2}\sum_{\xi\in\mu'}(1+\xi^{-1})(1+\xi)^{n(p-1)p^t}\\
&= \binom{r}{2}\sum_{\xi\in\mu'}(1+\xi^{-1})(1+pz_\xi)^{np^t}.
\end{align*}
 Computing $\sum_{2,r}$ mod $p^{t+1}$, we get
 \begin{align} \textstyle\sum_{2,r} &\equiv \binom{r}{2}\sum_{\xi\in\mu'}(1+\xi^{-1})~\text{mod}~p^{t+1}. \label{sum2r}
 \end{align}
 
 Now $\sum\limits_{\xi\in\mu'}(1+\xi^{-1})= p-1$, so using this in equation \eqref{sum2r}, we get
\begin{align*}
(p-1)\left(S_{2,r} + \binom{r}{2}\right) & \equiv \binom{r}{2}(p-1)~\text{mod}~p^{t+1}, 
\end{align*}
which is the same as $S_{2,r}\equiv 0 ~\text{mod}~p^{t+1}$, proving part \textit{(3)}.

Computing $S_{3,r}$ mod $p^{t}$:

Since $n>0$, $r>3$, hence $r-i >0$, so
\begin{align*}
\textstyle\sum_{3,r} &= \binom{r}{3}\sum_{\xi\in\mu'}(1+\xi)^{r-3}\\
&= \binom{r}{3}\sum_{\xi\in\mu'}(1+ pz_{\xi})^{np^t}.
\end{align*}
Computing $\sum_{3,r}$ mod $p^t$, we get that $\sum_{3,r}\equiv \binom{r}{3}(p-2)$ mod $p^t$. Therefore we get 
\begin{align*}
(p-1)\left(S_{3,r} + \binom{r}{3}\right) &\equiv \binom{r}{3}(p-2)~\text{mod}~p^t,
\end{align*}
which is the same as $S_{3,r} \equiv \frac{\binom{r}{3}}{1-p}~\text{mod}~p^t$,
proving part \textit{(4)}.

Computing $p^iS_{i,r}$ mod $p^{t+4}$, for $i\geq 4$:

Note that $\binom{r}{i}$ divides $\binom{j}{i}\binom{r}{j}$, since $\binom{j}{i}\binom{r}{j} = \binom{r}{i}\binom{r-i}{j-i}$. By Lemma \ref{lemma 0.1.} we see that if $i\geq 4$, then
\begin{align*}
p^i\binom{r}{i} &\equiv 0 ~\text{mod}~p^{t+4},\\
\therefore \sum_{\substack{{0<j<r}\\{j\equiv 3~\text{mod}~(p-1)}}} p^i\binom{j}{i}\binom{r}{j}&\equiv 0~\text{mod}~p^{t+4}.
\end{align*}
Hence part \textit{(5)} holds.
\end{proof}

%\vspace{5mm}

The next two propositions are used in the proof of Proposition~\ref{F_1 when tau <= t}.

\begin{prop}\label{propositon 0.5}
Let $p>3$ and write $s= 1+n(p-1)p^t,$ with $t\geq 0$ and $n > 0.$ Then we have:
\begin{enumerate}[label=(\arabic{*})]
\item $\begin{aligned}[t]
\sum_{j\equiv 1~\emph{mod}~(p-1)}\binom{s}{j} &\equiv 1+ np^{t+1}~\emph{mod}~p^{t+2}.
\end{aligned}$
\item $\begin{aligned}[t]
\sum_{j\equiv 1 ~\emph{mod}~(p-1)}j\binom{s}{j} &\equiv \frac{s(p-2)}{p-1}-snp^{t+1}~\emph{mod}~p^{t+2}.
\end{aligned}$
\item$\begin{aligned}[t]
\text{For all}~i\geq 2, p^i\sum_{j\equiv 1~\emph{mod}~(p-1)}\binom{j}{i}\binom{s}{j} &\equiv 0~\emph{mod}~p^{t+2}.
\end{aligned}$
\end{enumerate}
\end{prop}
\begin{proof}
This is Proposition 2.9 in \cite{BGR18}.
\end{proof}

%\vspace{5mm}

\begin{prop}\label{proposition 0.6}
Let $p>3$. If $r =3+n(p-1)p^t$, with $t=v(r-3)$ and $n>0$, then we have:
\begin{enumerate}[label=(\arabic{*})]
\item $\begin{aligned}[t]
\sum_{\substack{{1<j\leq r-2}\\{j \equiv 1~\emph{mod}~(p-1)}}}\binom{r}{j} &\equiv 3-r ~\emph{mod}~p^{t+1}.
\end{aligned}$
\item $\begin{aligned}[t]
\sum_{\substack{{1<j\leq r-2}\\{j \equiv 1~\emph{mod}~(p-1)}}}j\binom{r}{j} &\equiv 0 ~\emph{mod}~p^{t+1}.
\end{aligned} $
\item $\begin{aligned}[t]
\sum_{\substack{{1<j\leq r-2}\\{j \equiv 1~\emph{mod}~(p-1)}}}\binom{j}{2}\binom{r}{j} &\equiv 0~\emph{mod}~p^t.
\end{aligned}$
\item $\begin{aligned}[t]
\sum_{\substack{{1<j\leq r-2}\\{j \equiv 1~\emph{mod}~(p-1)}}}\binom{j}{3}\binom{r}{j} &\equiv \frac{\binom{r}{3}}{1-p}~\emph{mod}~p^t.
\end{aligned}$
\item $\begin{aligned}[t]
p^i\sum_{\substack{{1<j\leq r-2}\\{j \equiv 1~\emph{mod}~(p-1)}}}\binom{j}{i}\binom{r}{j} &\equiv 0 ~\emph{mod}~p^{t+4}, \forall i \geq 4.
\end{aligned}$
\end{enumerate} 
\end{prop}
\begin{proof}
Let $0\leq i<r$. Let $S_{i,r}= \sum\limits_{\substack{{1<j\leq r-2}\\{j\equiv 1~\text{mod}~(p-1)}}}\binom{j}{i}\binom{r}{j}$. Let $\sum_{i,r}= (p-1)\sum\limits_{\substack{{j\geq 0}\\{j\equiv 1~\text{mod}~(p-1)}}}\binom{j}{i}\binom{r}{j}$, so we get
\begin{align*}
\textstyle\sum_{i,r} = (p-1)\left(r\binom{1}{i} + S_{i,r}\right).
\end{align*}
 Let $f_r(x)=(1+x)^r=\sum\limits_{j\geq 0}\binom{r}{j}x^j$, which we consider as a function from $\mathbb{Z}_p$ to $\mathbb{Z}_p$. Let \begin{align*}
 g_{i,r}(x) &= \frac{x^{i-1}}{i!}f_{r}^{(i)}(x)\\
 &= \sum_{j\geq 0}\binom{j}{i}\binom{r}{j}x^{j-1}.
\end{align*}  
 Let $\mu_{p-1}$ be the set of $(p-1)$-st roots of unity. Summing $g_{i,r}(x)$ as $x$ varies over the $(p-1)$-st roots of unity we get
\begin{align*}
\sum_{\xi\in\mu_{p-1}}g_{i,r}(\xi) &= \sum_{\xi\in\mu_{p-1}}\sum_{j\geq 0}\binom{j}{i}\binom{r}{j}\xi^{j-1}\\
&= \sum_{\xi\in\mu_{p-1}}\left(\sum_{j\equiv 1~\text{mod}~(p-1)}\binom{j}{i}\binom{r}{j}\xi^{j-1} + \sum_{j\not\equiv 1~\text{mod}~(p-1)}\binom{j}{i}\binom{r}{j}\xi^{j-1}\right).
\end{align*}
Using \eqref{sum of roots of 1}, we get
\begin{align*}
\sum_{\xi\in\mu_{p-1}}g_{i,r}(\xi) = (p-1)\sum_{j\equiv 1~\text{mod}~(p-1)}\binom{j}{i}\binom{r}{j},
\end{align*}
which is equal to $\sum_{i,r}$. Let $\mu'$ be the set of $(p-1)$-st roots of unity except $-1$. Since $g_{i,r}(x) =\binom{r}{i}x^{i-1}(1+x)^{r-i}$, we have
\begin{align}
\textstyle\sum_{i,r} &= \sum_{\xi\in\mu_{p-1}}\binom{r}{i}\xi^{i-1}(1+\xi)^{r-i}\nonumber \\
&= \binom{r}{i}\sum_{\xi\in\mu'}\xi^{i-1}(1+\xi)^{r-i}, ~\text{as}~ r-i >0. \label{sum r, j=1}
\end{align} 
Also for $\xi\in\mu'$, we can write $(1+\xi)^{p-1}=1+pz_\xi$, for some $z_\xi\in \mathbb{Z}_p$.

Computing $S_{0,r}$ mod $p^{t}$:
\begin{align}
\textstyle\sum_{0,r} &= \sum_{\xi\in\mu'}\xi^{-1}(1+\xi)^r \nonumber \\
&= \sum_{\xi\in\mu'}\xi^{-1}(1+\xi)^3(1+\xi)^{r-3}\nonumber \\
&= \sum_{\xi\in\mu'}\xi^{-1}(1+\xi)^3(1+pz_\xi)^{np^t}. \label{sum r,i=0,j=1}
\end{align}
Now $\sum\limits_{\xi\in\mu'}\xi^{-1}(1+\xi)^3=\sum\limits_{\xi\in\mu'}(\xi^{-1}+3 +3\xi +\xi^2)= 1 +3(p-2) + 3 -1= 3(p-1)$, so using this in equation \eqref{sum r,i=0,j=1}, we get
\begin{align*}
(p-1)(r+S_{0,r}) &\equiv 3(p-1)~\text{mod}~p^{t+1},
\end{align*} 
which is the same as $S_{0,r}\equiv 3-r ~\text{mod}~p^{t+1}$, proving part \textit{(1)}.

Computing $S_{1,r}$ mod $p^{t+1}$:
\begin{align}
\textstyle\sum_{1,r} &= \binom{r}{1}\sum_{\xi\in\mu'}(1+\xi)^{r-1} \nonumber\\
&= r\sum_{\xi\in\mu'}(1+\xi)^2(1+pz_\xi)^{np^t} \nonumber \\
&\equiv r\sum_{\xi\in\mu'}(1+\xi)^2~\text{mod}~p^{t+1}. \label{sum r,i=1,j=1}
\end{align}
Now $\sum\limits_{\xi\in\mu'}(1+\xi)^2 = \sum\limits_{\xi\in\mu'}(1+2\xi + \xi^2)=p-2 + 2 -1= p-1$, so using this in equation \eqref{sum r,i=1,j=1}, we get
\begin{align*}
(p-1)(r+S_{1,r}) &\equiv r(p-1)~\text{mod}~p^{t+1},\\
S_{1,r} &\equiv 0~\text{mod}~p^{t+1},
\end{align*} 
proving part \textit{(2)}.

Computing $S_{2,r}$ mod $p^t$:
\begin{align}
\textstyle\sum_{2,r} &= \binom{r}{2}\sum_{\xi\in\mu'}\xi(1+\xi)^{r-2} \nonumber\\
&= \binom{r}{2}\sum_{\xi\in\mu'}\xi(1+\xi)(1+\xi)^{r-3} \nonumber\\
&= \binom{r}{2}\sum_{\xi\in\mu'}(\xi+\xi^2)(1+pz_\xi)^{np^t}. \label{sum r,i=2,j=1}
\end{align}
Now $\sum\limits_{\xi\in\mu'}(\xi+\xi^2)= 1-1=0$, so using this in equation \eqref{sum r,i=2,j=1}, we get
\begin{align*}
(p-1)S_{2,r} &\equiv 0~\text{mod}~p^t\\
S_{2,r} &\equiv 0~\text{mod}~p^t,
\end{align*}
proving part \textit{(3)}.

Computing $S_{3,r}$ mod $p^t$:

Since $n>0$, $r>3$, hence $r-i>0$, so
\begin{align}
\textstyle\sum_{3,r} &= \binom{r}{3}\sum_{\xi\in\mu'}\xi^2(1+\xi)^{r-3} \nonumber \\
&= \binom{r}{3}\sum_{\xi\in\mu'}\xi^2(1+pz_\xi)^{np^t}. \label{sum r,i=3,j=1}
\end{align} 
Now $\sum\limits_{\xi\in\mu'}\xi^2 = -1$. Using this in equation \eqref{sum r,i=3,j=1}, we get
\begin{align*}
\textstyle\sum_{3,r} &\equiv \binom{r}{3}\sum_{\xi\in\mu'}\xi^2~\text{mod}~p^t\\
(p-1)S_{3,r} &\equiv -\binom{r}{3}~\text{mod}~p^{t}\\
S_{3,r} &\equiv \frac{\binom{r}{3}}{1-p}~\text{mod}~p^t,
\end{align*}
proving part \textit{(4)}.

Computing $p^iS_{i,r}$ mod $p^{t+4}$, for $i\geq 4$:

Note that $\binom{r}{i}$ divides $\binom{j}{i}\binom{r}{j},$ since $\binom{j}{i}\binom{r}{j}= \binom{r}{i}\binom{r-i}{j-i}.$ By Lemma \ref{lemma 0.1.}, we get that for $i \geq 4$,
\begin{align*}
p^i\binom{r}{i}&\equiv 0~\text{mod}~p^{t+4},\\
\therefore\sum_{\substack{{1<j\leq r-2}\\{j\equiv 1~\text{mod}~(p-1)}}}p^i\binom{j}{i}\binom{r}{j} &\equiv 0~\text{mod}~p^{t+4},
\end{align*}
proving part \textit{(5)}.
\end{proof}

%\vspace{5mm}

The next proposition is used in the proofs of Proposition~\ref{F_1 when tau > t + 0.5}, Proposition~\ref{F_2 when tau < t + 1} and 
Proposition~\ref{F_2 when tau = > t + 1}.

\begin{prop}\label{Proposition 0.11} Let $p>3$ and write $r= 2+ n(p-1)p^t$, with $t= v(r-2)$ and $n >0$. Then:
\begin{enumerate}[label = (\arabic{*})]
\item $\begin{aligned}[t]
\sum\limits_{\substack{{0<j<r}\\{j\equiv 2~\emph{mod}~(p-1)}}}\binom{r}{j}  \equiv \frac{p(2-r)}{2} + \frac{3p^2(2-r)}{2} -\frac{p^2(2-r)^2}{2}~\text{mod}~p^{t+3}.
\end{aligned}$
\item $\begin{aligned}[t]
\sum\limits_{\substack{{0<j<r}\\{j\equiv 2~\emph{mod}~(p-1)}}}j\binom{r}{j} \equiv \frac{pr(2-r)}{1-p}~\emph{mod}~p^{t+2}.
\end{aligned}$
\item $\begin{aligned}[t]
\sum\limits_{\substack{{0<j<r}\\{j\equiv 2~\emph{mod}~(p-1)}}}\binom{j}{2}\binom{r}{j}\equiv \frac{\binom{r}{2}}{1-p}~\emph{mod}~p^{t+1}.
\end{aligned}$
\item $\begin{aligned}[t]
p^i\sum\limits_{\substack{{0<j<r}\\{j\equiv 2~\emph{mod}~(p-1)}}}\binom{j}{i}\binom{r}{j} \equiv 0~\emph{mod}~p^{t+3}, ~\text{for all}~ i\geq 3.
\end{aligned}$
\end{enumerate}
\end{prop}
\begin{proof}
We prove part \textit{(1)} of the proposition. Parts \textit{(2)}, \textit{(3)} and \textit{(4)} are in  \cite{BGR18}, Proposition~2.8.

Let $S_r = \sum\limits_{\substack{{2\leq j < r}\\{j\equiv 2 ~\text{mod}~(p-1)}}}\binom{r}{j}$. Let $\sum_r = (p-1)\sum\limits_{\substack{{2\leq j \leq r}\\{j\equiv 2 ~\text{mod}~(p-1)}}}\binom{r}{j}$, then $\sum_r = (p-1)\left(S_r + 1\right)$. Let 
\begin{align}
g_r(x) = x^{-2}(1+x)^r = \sum_{j\geq 0}\binom{r}{j}x^{j-2} \label{g_r(x)}.
\end{align}
Let $\mu_{p-1}$ be the set of $(p-1)$-st roots of unity. Summing $g_r(x)$ as $x$ varies over all $(p-1)$-st roots of unity we get
\begin{align*}
\sum_{\xi\in\mu_{p-1}}g_r(\xi) &= \sum_{\xi\in\mu_{p-1}}\sum_{j\geq 0}\binom{r}{j}\xi^{j-2}\\
&= \sum_{\xi\in\mu_{p-1}}\left(\sum_{j\equiv 2~\text{mod}~(p-1)}\binom{r}{j}\xi ^{j-2} + \sum_{j\not\equiv 2~\text{mod}~(p-1)}\binom{r}{j}\xi ^{j-2}  \right)\\
&= (p-1)\sum_{\substack{{2\leq j \leq r}\\{j \equiv 2~\text{mod}~(p-1)}}}\binom{r}{j},
\end{align*}
by \eqref{sum of roots of 1}. So $\sum\limits_{\xi\in\mu_{p-1}}g_r(x) = \sum_r$. Let $\mu'$ be the set of $(p-1)$-st roots of unity except $-1$. By \eqref{g_r(x)}, we have
\begin{align}
\textstyle\sum_r &= \sum_{\xi\in\mu'}\xi^{-2}(1 + \xi)^r. \label{sum_r} 
\end{align}
For $\xi\in\mu'$, we can write $(1+ \xi)^{p-1} = 1 + pz_{\xi}$, for some $z_{\xi}\in\mathbb{Z}_p$. From \eqref{sum_r}, we get
\begin{align*}
\textstyle\sum_r &= \sum_{\xi\in\mu'}\xi^{-2}(1 + \xi)^{2 + n(p-1)p^t} \\
& = \sum_{\xi\in\mu'}(1 + \xi^{-1})^2(1 + pz_{\xi})^{np^t}.
\end{align*}
Computing $\sum_r$ mod $p^{t+3}$, we get
\begin{align}
\textstyle\sum_r &\equiv  \sum_{\xi\in\mu'}(1 + \xi^{-1})^2\left(1 + np^{t+1}z_\xi + np^{t+2}\left(\frac{np^t-1}{2}\right)z^2_{\xi} \right)~\text{mod}~p^{t+3}. \label{Sum_r}
\end{align}
Now $\sum_{\xi\in\mu'}(1 + \xi^{-1})^2 = \sum_{\xi\in\mu'}(1 + 2\xi^{-1} + \xi^{-2}) = p-2 + 2 - 1 = p-1 $.

Specializing equation \eqref{Sum_r} at $n=1$, $t=0$, so $r = p + 1$, we get 
\begin{align*}
\textstyle\sum_{p+1} &\equiv \sum_{\xi\in\mu'}(1 + \xi^{-1})^2(1 + pz_\xi)~\text{mod}~ p^3\\
(p-1)\left(\binom{p+1}{2} + \binom{p+1}{p+1}\right) &\equiv p-1 + p\sum_{\xi\in\mu'}(1 + \xi^{-1})^2z_{\xi} ~\text{mod}~p^3.
\end{align*}
Simplifying we get 
\begin{align}
p\sum_{\xi\in\mu'}(1 + \xi^{-1})^2z_{\xi} \equiv \frac{-p}{2}~\text{mod}~p^3. \label{Sum_(p+1)}
\end{align}
Again specializing \eqref{Sum_r} at $n=2$, $t = 0$, so $r = 2p$, we get
\begin{align*}
\textstyle\sum_{2p} &\equiv \sum_{\xi\in\mu'}(1 + \xi^{-1})^{2}(1 + 2pz_{\xi} + p^2z^2_{\xi})~\text{mod}~p^3\\
(p-1)\left(\binom{2p}{2} + \binom{2p}{p+1} + \binom{2p}{2p}\right) &\equiv p-1 + 2p\sum_{\xi\in\mu'}(1 + \xi^{-1})^2z_{\xi} + p^2 \sum_{\xi\in\mu'}(1 + \xi^{-1})^2z^2_{\xi}~\text{mod}~p^3. 
\end{align*}
By using \eqref{Sum_(p+1)} in the above expression and simplifying further we get 
\begin{align}
p^2\sum_{\xi\in\mu'}(1 + \xi^{-1})^2z^2_{\xi} &\equiv 2p - 3p^2 + (p-1)\binom{2p}{p+1}~\text{mod}~p^3. \label{Sum_(2p)} 
\end{align}
Now using equations \eqref{Sum_(p+1)} and \eqref{Sum_(2p)} in \eqref{Sum_r}, we get 
\begin{align*}
(p-1)(S_r + 1) &\equiv p-1 + np^t\left(\frac{-p}{2}\right) + np^t\left(\frac{np^t-1}{2}\right)\left(2p - 3p^2 + (p-1)\binom{2p}{p+1}\right)~\text{mod}~p^{t+3}.
\end{align*}
So
\begin{align*}
S_r &\equiv \frac{-np^{t+1}}{2(p-1)} + \frac{np^t(np^t - 1)}{2(p-1)}\left(2p - 3p^2 + (p-1)\binom{2p}{p+1}\right)~\text{mod}~p^{t+3}. 
\end{align*}
As $\frac{1}{p^2}(2p - 3p^2 + (p-1)\binom{2p}{p+1}) \equiv 1$ mod $p$, we get
\begin{align*}
S_r  &\equiv \frac{np^{t+1}}{2(1-p)} -\frac{np^{t+2}(np^t-1)}{2(1-p)}~\text{mod}~p^{t+3}\\
&\equiv \frac{np^{t+1}}{2} + \frac{np^{t+2}}{2} + \frac{np^{t+2}}{2} -\frac{n^2p^{2t+2}}{2} ~\text{mod}~p^{t+3}\\
&= \frac{np^{t+1}}{2} + np^{t+2} -\frac{n^2p^{2t+2}}{2}~\text{mod}~p^{t+3}.
\end{align*}
Replacing $np^{t}$ by $\frac{2-r}{1-p}$, we get
\begin{align*}
S_r &\equiv \frac{(2-r)p}{2(1-p)} + \frac{(2-r)p^2}{1-p} - \frac{p^2}{2}\left(\frac{2-r}{1-p}\right)^2~\text{mod}~p^{t+3}\\
&\equiv \frac{(2-r)p}{2} + \frac{(2-r)p^2}{2} + (2-r)p^2  -\frac{p^2(2-r)^2}{2}~\text{mod}~p^{t+3}\\
&= \frac{(2-r)p}{2} + \frac{3(2-r)}{2} -\frac{p^2(2-r)^2}{2}~\text{mod}~p^{t+3},
\end{align*}
proving part \textit{(1)} of the proposition.

Note in particular that $S_r \equiv \frac{(2-r)p}{2}$ mod $p^{t+2}$, which matches with part \textit{(1)} of Proposition~2.8 in \cite{BGR18}.
\end{proof}
%\vspace{5mm}

The next proposition is used in the proofs of Proposition~\ref{F_2 when tau =< t}, Proposition~\ref{F_2 when tau < t + 1} and
Proposition~\ref{F_2 when tau = > t + 1}.

\begin{prop}{\label{Proposition 0.10}}
Let $p> 3$. If $r = 3+ n(p-1)p^t$, with $t=v(r-3)$ and $n>0$, then we have:
\begin{enumerate}[label = (\arabic{*})]
\item $\begin{aligned}[t] 
\sum\limits_{\substack{{2<j\leq r-1}\\{j\equiv 2~\emph{mod}~(p-1)}}}\binom{r}{j} \equiv 3-\binom{r}{2} + \frac{5np^{t+1}}{2} ~\emph{mod}~p^{t+2}.
\end{aligned}$
\item $\begin{aligned}[t]
\sum\limits_{\substack{{2<j\leq r-1}\\{j\equiv 2~\emph{mod}~(p-1)}}}j\binom{r}{j} \equiv r(3-r)~\emph{mod}~p^{t+1}.
\end{aligned}$
\item $\begin{aligned}[t]
\sum\limits_{\substack{{2<j\leq r-1}\\{j\equiv 2~\emph{mod}~(p-1)}}}\binom{j}{2}\binom{r}{j}\equiv 0~\emph{mod}~p^{t+1}.
\end{aligned}$
\item $\begin{aligned}[t]
\sum\limits_{\substack{{2<j\leq r-1}\\{j\equiv 2~\emph{mod}~(p-1)}}}\binom{j}{3}\binom{r}{j} \equiv \frac{\binom{r}{3}}{p-1}~\emph{mod}~p^{t+1}.
\end{aligned}$
\item $\begin{aligned}[t]
p^i\sum\limits_{\substack{{2<j\leq r-1}\\{j\equiv 2~\emph{mod}~(p-1)}}}\binom{j}{i}\binom{r}{j} \equiv 0 ~\emph{mod}~p^{t+4}, \forall i \geq 4.
\end{aligned}$
\end{enumerate}
\end{prop}
\begin{proof}
Let $0 \leq i < r$. Let $S_{i,r}= \sum\limits_{\substack{{2<j\leq r-1}\\{j\equiv 2~\text{mod}~(p-1)}}}\binom{j}{i}\binom{r}{j}$. Let $\sum_{i,r} = (p-1)\sum\limits_{\substack{{j \geq 0}\\{j\equiv 2~\text{mod}~(p-1)}}}\binom{j}{i}\binom{r}{j}$, then $\sum_{i,r} = (p-1)\left(\binom{2}{i}\binom{r}{2} + S_{i,r}  \right)$. Let $f_r(x)= (1+x)^r = \sum\limits_{j \geq 0}\binom{r}{j}x^j$, which we consider as a function from $\mathbb{Z}_p$ to $\mathbb{Z}_p$. Let 
\begin{align*}
g_{i,r}(x) &= \frac{x^{i-2}}{i!}f_r^{(i)}(x)\\
& = \sum\limits_{j\geq 0}\binom{j}{i}\binom{r}{j}x^{j-2}.
\end{align*}
Let $\mu_{p-1}$ be the $(p-1)$-st roots of unity. Summing $g_{i,r}(x)$ as $x$ varies over all $(p-1)$-st roots of unity we get
\begin{align*}
\sum_{\xi\in\mu_{p-1}} g_{i,r}(\xi) &= \sum_{\xi\in\mu_{p-1}}\sum_{j\geq 0}\binom{j}{i}\binom{r}{j}\xi^{j-2}\\
&= \sum_{\xi\in\mu_{p-1}}\left(\sum_{j\equiv 2~\text{mod}~(p-1)}\binom{j}{i}\binom{r}{j}\xi^{j-2}+ \sum_{j\not\equiv 2~\text{mod}~(p-1)}\binom{j}{i}\binom{r}{j}\xi^{j-2}\right).
\end{align*}
Using \eqref{sum of roots of 1}, we get 
\begin{align*}
\sum_{\xi\in\mu_{p-1}} g_{i,r}(\xi) = (p-1)\sum_{\substack{{j\equiv 2~\text{mod}~(p-1)}}}\binom{j}{i}\binom{r}{j},
\end{align*}
which is equal to $\sum_{i,r}$. Let $\mu'$ be the set of $(p-1)$-st roots of unity except $-1$. Since $g_{i,r}(x)= x^{i-2}\binom{r}{i}(1+x)^{r-i}$, we have
\begin{align*}
\textstyle\sum_{i,r} &= \sum_{\xi\in\mu_{p-1}} \binom{r}{i}\xi^{i-2}(1+\xi)^{r-i}\\
&= \binom{r}{i}\sum_{\xi\in\mu'}\xi^{i-2}(1+\xi)^{r-i}, ~\text{as}~ r-i >0.
\end{align*}
Also for $\xi\in\mu'$, write $(1+\xi)^{p-1}= 1+ pz_\xi$, for some $z_\xi \in\mathbb{Z}_p$.

Computing $S_{0,r}$ mod $p^{t+2}$:
\begin{align*}
\textstyle\sum_{0,r} &= \sum_{\xi\in\mu'}\xi^{-2}(1+\xi)^r\\
&= \sum_{\xi\in\mu'}\xi^{-2}(1+ \xi)^3(1+\xi)^{n(p-1)p^t}\\
&= \sum_{\xi\in\mu'}\xi^{-2}(1+\xi )^3(1+pz_{\xi})^{np^t}.
\end{align*}
Computing $\sum_{0,r}$ mod $p^{t+2}$ we get 
\begin{align}
\textstyle\sum_{0,r} &\equiv \sum_{\xi\in\mu'}\xi^{-2}(1+\xi)^3(1+ np^{t+1}z_\xi) ~\text{mod}~p^{t+2}. \label{i=0, j=2, sum_{0,r}}
\end{align}
Now $\sum\limits_{\xi\in\mu'}\xi^{-2}(1+\xi)^3 = \sum\limits_{\xi\in\mu'}(\xi^{-2} + 3\xi^{-1} + 3 + \xi) = -1 + 3 + 3(p-2) + 1= 3(p-1)$.

Specializing  equation \eqref{i=0, j=2, sum_{0,r}} at $t=0$, $n=1$, so $r = p+2$, we get  
\begin{align*}
\textstyle\sum_{0,p+2} &\equiv \sum_{\xi\in\mu'}\xi^{-2}(1+\xi)^3(1+ pz_{\xi})~\text{mod}~p^{2}\\
(p-1)\left(\binom{p+2}{2}+ \binom{p+2}{p+1}\right) &\equiv 3(p-1) + p\sum_{\xi\in\mu'}\xi^{-2}(1+\xi)^3z_\xi ~\text{mod}~p^2.
\end{align*}
Simplifying we get
\begin{align}
 p\sum_{\xi\in\mu'}\xi^{-2}(1+\xi)^3z_{\xi} \equiv -\frac{5p}{2}~\text{mod}~p^2. \label{i=0,j=2,r=p+2, sum_{0,r}}
\end{align}
Using equation \eqref{i=0,j=2,r=p+2, sum_{0,r}} in equation \eqref{i=0, j=2, sum_{0,r}}, we get
\begin{align*}
(p-1)\left(\binom{r}{2} + S_{0,r}\right) &\equiv 3(p-1) + np^t\left(\frac{-5p}{2}\right)~\text{mod}~p^{t+2}\\
\binom{r}{2} + S_{0,r} &\equiv 3-\frac{5np^{t+1}}{2(p-1)}~\text{mod}~p^{t+2}\\
S_{0,r} &\equiv 3-\binom{r}{2} + \frac{5np^{t+1}}{2}~\text{mod}~p^{t+2},
\end{align*}
proving part \textit{(1)}.

Computing $S_{1,r}$ mod $p^{t+1}$:
\begin{align*}
\textstyle\sum_{1,r} &= \binom{r}{1}\sum_{\xi\in\mu'}\xi^{-1}(1+\xi)^{r-1}\\
&= r\sum_{\xi\in\mu'}\xi^{-1}(1+ \xi)^2(1+pz_\xi)^{np^t}.
\end{align*}
Computing $\sum_{1,r}$ mod $p^{t+1}$ we get
\begin{align}
\textstyle\sum_{1,r} &\equiv r\sum_{\xi\in\mu'}\xi^{-1}(1+\xi)^2 ~\text{mod}~p^{t+1}. \label{i=1,j=2, sum_{1,j}}
\end{align}
Now $\sum\limits_{\xi\in\mu'}\xi^{-1}(1+\xi)^2 = \sum\limits_{\xi\in\mu'}(\xi^{-1} + 2 + \xi)= 1+2(p-2) + 1 = 1= 2(p-1)$.
Using this in equation \eqref{i=1,j=2, sum_{1,j}}, we get
 \begin{align*}
 (p-1)\left(2\binom{r}{2} + S_{1,r}\right) &\equiv 2r(p-1)~\text{mod}~p^{t+1}.
\end{align*}
So
\begin{align*}
 S_{1,r}& \equiv r(3-r)~\text{mod}~p^{t+1},
\end{align*}
proving part \textit{(2)}.

Computing $S_{2,r}$ mod $p^{t+1}$: 
\begin{align*}
\textstyle\sum_{2,r} &= \binom{r}{2}\sum_{\xi\in\mu'}(1+\xi)^{r-2}\\
&= \binom{r}{2}\sum_{\xi\in\mu'}(1+\xi)(1+\xi)^{n(p-1)p^t}\\
&= \binom{r}{2}\sum_{\xi\in\mu'}(1+\xi)(1+pz_\xi)^{np^t}.
\end{align*}
Computing $\sum_{2,r}$ mod $p^{t+1}$ we get
\begin{align}
\textstyle\sum_{2,r} &\equiv \binom{r}{2}\sum_{\xi\in\mu'}(1+\xi)~\text{mod}~p^{t+1}. \label{i=2,j=2,sum_{2,j}}
\end{align}
Now $\sum\limits_{\xi\in\mu'}(1+\xi) = (p-2) +1 = p-1$.
Using this in equation \eqref{i=2,j=2,sum_{2,j}}, we get
\begin{align*}
\textstyle\sum_{2,r} & \equiv \binom{r}{2}(p-1)~\text{mod}~p^{t+1}\\
(p-1)\left(\binom{r}{2} + S_{2,r}\right) & \equiv \binom{r}{2}(p-1)~\text{mod}~p^{t+1},
\end{align*}
which is the same as $S_{2,r} \equiv 0~\text{mod}~p^{t+1}$, proving part \textit{(3)}.

Computing $S_{3,r}$ mod $p^{t+1}$:

Since $n>0$, $r>3$, hence $r-i >0$, so 
\begin{align*}
\sum_{3,r} &= \binom{r}{3}\sum_{\xi\in\mu'}\xi(1+\xi)^{r-3}\\
&= \binom{r}{3} \sum_{\xi\in\mu'}\xi(1+ pz_\xi)^{np^t}.
\end{align*}
Computing $\sum_{3,r}$ mod $p^{t+1}$ we get
\begin{align}
\textstyle\sum_{3,r} &\equiv \binom{r}{3}\sum_{\xi\in\mu'}\xi~\text{mod}~p^{t+1}. \label{i=3,j=2,sum_{3,r}}
\end{align}
Now $\sum\limits_{\xi\in\mu'}\xi =1$. Using this in equation \eqref{i=3,j=2,sum_{3,r}}, we get
\begin{align*}
\textstyle\sum_{3,r} &\equiv \binom{r}{3}~\text{mod}~p^{t+1}\\
(p-1)S_{3,r} & \equiv \binom{r}{3}~\text{mod}~p^{t+1},
\end{align*}
which is the same as $S_{3,r} \equiv \frac{\binom{r}{3}}{p-1}$, proving part \textit{(4)}.

Computing $p^iS_{i,r}$ mod $p^{t+4}$ for $i\geq 4$:  

As before $\binom{r}{i}$ divides $\binom{j}{i}\binom{r}{j}$. From Lemma \ref{lemma 0.1.} we get that for $i \geq 4$
\begin{align*}
p^i\binom{r}{i}&\equiv 0 ~\text{mod}~p^{t+4},\\
\therefore p^i\sum\limits_{\substack{{2<j\leq r-1}\\{j\equiv 2~\text{mod}~(p-1)}}}\binom{j}{i}\binom{r}{j}& \equiv 0~\text{mod}~p^{t+4},
\end{align*}
proving part \textit{(5)}.
\end{proof}
%\vspace{5mm}

The final proposition is a slight variant of the previous proposition. It is also used in the proof of Proposition~\ref{F_2 when tau =< t}.

\begin{prop}\label{beta}Let $p>3$. If $r=3+n(p-1)p^t$, with $t = v(r-3)$ and $n>0$, then one can choose integers $\beta_j$ for all $j\equiv 2~\emph{mod}~(p-1),$ with $2\leq j < r-1$, satisfying:
\begin{enumerate}[label=(\arabic{*})]
\item $\begin{aligned}[t]
\beta_j \equiv \binom{r}{j}~\emph{mod}~p^t ~\text{for all j's above.}\\
\end{aligned}$
\item $\begin{aligned}[t]
\sum\limits_{\substack{{2\leq j<r-1}\\{j\equiv 2~\emph{mod}~(p-1)}}}\binom{j}{i}\beta_j \equiv 0 ~\emph{mod}~p^{t+2-i} ~\text{for} ~i =0,1~\text{and}~2.
\end{aligned}$
\item $\begin{aligned}[t]
p^3\sum\limits_{\substack{{2\leq j < r-1}\\{j\equiv 2~\emph{mod}~(p-1)}}}\binom{j}{3}\beta_j \equiv \frac{p^3\binom{r}{3}}{p-1}~\emph{mod}~p^{t+3}.
\end{aligned}$
\item $\begin{aligned}[t] \text{For}~i\geq 4, ~\text{we have}~ p^i\sum\limits_{\substack{{j\equiv 2~\emph{mod}~(p-1)}\\{2\leq j < r-1}}}\binom{j}{i}\beta_j \equiv 0~\emph{mod}~p^{t+3}.
\end{aligned}$
\end{enumerate}
\end{prop}

\begin{proof}
The proof is similar to that of Lemma 7.2 in \cite{Bhattacharya-Ghate}.
Define \begin{align}
\beta_j & = \binom{r}{j}~\text{for}~j\equiv 2~\text{mod}~(p-1),\, 2<j<r -1,\, j\neq 2p, \label{def beta_j}\\
\beta_2&= -\sum\limits_{\substack{{2< j < r-1}\\{j\equiv 2~\text{mod}~(p-1)}}}b'j\binom{r}{j}~\text{where}~2b'\equiv 1~\text{mod}~p^{t+1}, \label{def beta_2}\\
\beta_{2p}&= - \sum\limits_{\substack{{2< j < r-1}\\{j\equiv 2~\text{mod}~(p-1)}\\{j\neq 2p}}}\binom{r}{j}-\beta_2. \label{def beta_2p}
\end{align}
 Clearly $\beta_j \equiv \binom{r}{j}~\text{mod}~p^t$, when $2<j<r-1$, $ j\equiv 2~\text{mod}~(p-1)$ and $j\neq 2p$.

 Now for $j = 2 $ we have 
\begin{align}
2\beta_2 & = -\sum\limits_{\substack{{2< j < r-1}\\{j\equiv 2~\text{mod}~(p-1)}}}2b'j\binom{r}{j} \nonumber \\
&\equiv -\sum\limits_{\substack{{2< j < r-1}\\{j\equiv 2~\text{mod}~(p-1)}}}j\binom{r}{j}~\text{mod}~p^t, \label{beta_2}
\end{align} 
because $2b'\equiv 1~ \text{mod}~p^{t+1}$. Using Proposition \ref{Proposition 0.10}, part \textit{(2)} in equation \eqref{beta_2}, we get $2\beta_2 \equiv r(r-1)~\text{mod}~p^t$, which is the same as $\beta_2 \equiv \binom{r}{2}~\text{mod}~p^t$.

For $j =2p$ we have to show that
 \begin{align*}
\beta_{2p} = -\sum\limits_{\substack{{2< j < r-1}\\{j\equiv 2~\text{mod}~(p-1)}\\{j\neq 2p}}}\binom{r}{j}-\beta_2 \equiv \binom{r}{2p}~\text{mod}~p^t,
\end{align*}
which is the same as showing 
\begin{align}
-\sum_{\substack{{2<j<r-1}\\{j\equiv 2~\text{mod}~(p-1)}}}\binom{r}{j} \equiv \binom{r}{2}~\text{mod}~p^t, \label{beta_{2p}}
\end{align} because $\beta_2 \equiv\binom{r}{2}$ mod $p^t$.
By Proposition \ref{Proposition 0.10}, part \textit{(1)} mod $p^t$, the left hand side of equation \eqref{beta_{2p}} is 
\begin{align*}
 -3 + \binom{r}{2} + \binom{r}{r-1}  =r-3 + \binom{r}{2} \equiv \binom{r}{2}~\text{mod}~p^t,
\end{align*} 
which is the right hand side of \eqref{beta_{2p}}. Therefore $\beta_{2p}\equiv \binom{r}{2p}~\text{mod}~p^t$. This proves part \textit{(1)} of the proposition.

For $i=0$,
 \begin{align*}
\sum_{\substack{{2\leq j<r-1}\\{j\equiv 2~\text{mod}~(p-1)}}}\beta_j = \beta_2 + \beta_{2p} + \sum_{\substack{{2< j < r-1}\\{j\equiv 2~\text{mod}~(p-1)}\\{j\neq 2p}}}\binom{r}{j} = 0.
\end{align*}
Therefore $\sum\limits_{\substack{{2\leq j< r-1}\\{j\equiv 2~\text{mod}~(p-1)}}}\beta_j \equiv 0~\text{mod}~p^{t+2}$.

For $i =1$,
 \begin{align*}
\sum\limits_{\substack{{2\leq j< r-1}\\{j\equiv 2~\text{mod}~(p-1)}}}j\beta_j &= 2\beta_2 +2p\beta_{2p}+ \sum\limits_{\substack{{2< j< r-1}\\{j\equiv 2~\text{mod}~(p-1)}\\{j\neq 2p}}}j\binom{r}{j}\\
&\equiv -\sum\limits_{\substack{{2<j<r-1}\\{j\equiv 2~\text{mod}~(p-1)}}}j\binom{r}{j} + 2p\beta_{2p} + \sum\limits_{\substack{{2< j< r-1}\\{j\equiv 2~\text{mod}~(p-1)}\\{j\neq 2p}}}j\binom{r}{j} ~\text{mod}~p^{t+1}\\
&= -2p\binom{r}{2p} + 2p\beta_{2p}  \equiv 0~\text{mod}~p^{t+1}.
\end{align*}
For $i=2$,
\begin{align*}
\sum_{\substack{{2\leq j<r-1}\\{j\equiv 2~\text{mod}~(p-1)}}}\binom{j}{2}\beta_j &\equiv \sum_{\substack{{2\leq j<r-1}\\{j\equiv 2~\text{mod}~(p-1)}}}\binom{j}{2}\binom{r}{j} ~\text{mod}~p^{t}.
\end{align*}
Using Proposition \ref{Proposition 0.10}, part \textit{(3)}, we get 
\begin{align*}
\sum_{\substack{{2\leq j<r-1}\\{j\equiv 2~\text{mod}~(p-1)}}}\binom{j}{2}\beta_j &\equiv \binom{r}{2}-r\binom{r-1}{2}~\text{mod}~p^t\\
&\equiv \binom{r}{2}(3-r) \equiv ~0 ~\text{mod}~p^t,
\end{align*}
proving part \textit{(2)} of the proposition.

Since $\beta_j\equiv \binom{r}{j}$ mod $p^t$, we have
\begin{align}
p^3\sum_{\substack{{2\leq j<r-1}\\{j\equiv 2~\text{mod}~(p-1)}}}\binom{j}{3}\beta_j\equiv \sum_{\substack{{2\leq j<r-1}\\{j\equiv 2~\text{mod}~(p-1)}}}\binom{j}{3}\binom{r}{j}~\text{mod}~p^{t+3}. \label{i=3,sum beta_j}
\end{align}
Using Proposition \ref{Proposition 0.10}, part \textit{(4)} in equation \eqref{i=3,sum beta_j}, we get 
\begin{align*}
p^3\sum_{\substack{{2\leq j<r-1}\\{j\equiv 2~\text{mod}~(p-1)}}}\binom{j}{3}\binom{r}{j}\equiv p^3\left(\frac{\binom{r}{3}}{p-1}-r\binom{r-1}{3}\right)~\text{mod}~p^{t+4}.
\end{align*}
Using Lemma \ref{lemma 0.2}, we get that
\begin{align*}
p^3\sum_{\substack{{2\leq j<r-1}\\{j\equiv 2~\text{mod}~(p-1)}}}\binom{j}{3}\beta_j \equiv \frac{p^3\binom{r}{3}}{p-1}~\text{mod}~p^{t+3},
\end{align*}
proving part \textit{(3)}.

By part \textit{(1)} of the proposition we know that $\beta_j \equiv \binom{r}{j}$ mod $p^t$,
hence we have the following 
\begin{align*}
 p^i\sum_{\substack{{2\leq j < r-1}\\{j\equiv 2~\text{mod}~(p-1)}}}\binom{j}{i}\beta_j\equiv p^i\sum_{\substack{{2\leq j < r-1}\\{j\equiv 2~\text{mod}~(p-1)}}}\binom{j}{i}\binom{r}{j}~\text{mod}~p^{t+i}.
\end{align*}
For $i\geq 4$, by Proposition \ref{Proposition 0.10}, part \textit{(5)}, we get that 
\begin{align}
p^i\sum_{\substack{{2\leq j < r-1}\\{j\equiv 2~\text{mod}~(p-1)}}}\binom{j}{i}\binom{r}{j} \equiv -p^ir\binom{r-1}{i}~\text{mod}~p^{t+4}. \label{i>=4, sum beta_j}
\end{align}
Using Lemma \ref{lemma 0.2} in equation \eqref{i>=4, sum beta_j}, we finally get $p^i\sum\limits_{\substack{{2\leq j < r-1}\\{j\equiv 2~\text{mod}~(p-1)}}}\binom{j}{i}\beta_j\equiv 0~\text{mod}~p^{t+3}$ for all $i\geq 4$, proving part \textit{(5)}.
\end{proof}

\section{Telescoping Lemmas}
\label{sectiontelescope}

In this section we study the action of the Hecke operator on some simple functions  supported mod $KZ$ on a line segment in the Bruhat-Tits tree 
whose length is linear in $t = v(3-r)$.  Similar functions were used once in \cite{BG13} and twice in \cite{BGR18}. 
In this paper the importance of these functions will become clear as they
will be used extensively to study the structure of the subquotients $F_1$, $F_2$ and $F_3$ of $\bar\Theta_{k,a_p}$ at `infinity'. 

Define
\begin{eqnarray*}
  c & = & \frac{a_p^2-(r-2)\binom{r-1}{2}p^3}{pa_p},  \\
  \tilde{c} & = & \frac{a_p^2- \binom{r}{3}p^3}{pa_p},
\end{eqnarray*}
and set $\tau = v(c)$ and $\tilde{\tau} = v(\tilde{c})$. These quantities will play a crucial role in everything that follows.
%\vspace{5mm}
The following lemma explicates the relationship between $\tau$ and $\tilde{\tau}$.

\begin{lemma}{\label{Lemma 5.1}}
Let $p>3$, $r=3+n(p-1)p^t$ with $t =v(r-3)$, and let the slope of $a_p$ be $v(a_p)=\frac{3}{2}$. We have:
%Let $\tilde{c}=\frac{a_p^2-\binom{r}{3}p^3}{pa_p}$, $c =\frac{a_p^2-(r-2)\binom{r-1}{2}p^3}{pa_p}$, $\tilde{\tau} = v(\tilde{c})$ and $\tau = v(c)$. 
\begin{enumerate}[label =(\roman{*})]
\item If $\tilde{\tau} < t+ \frac{1}{2}$, then $\tilde{\tau} = \tau$.
\item If $\tilde{\tau} \geq t+ \frac{1}{2}$, then $\tau \geq t+ \frac{1}{2}$.
\item If $\tau < t+ \frac{1}{2}$, then $\tau = \tilde{\tau}$.
\item If $\tau \geq t+ \frac{1}{2}$, then $\tilde{\tau} \geq t+ \frac{1}{2}$.
\end{enumerate}
\end{lemma}
\begin{proof} Write 
\begin{align*}
c=\frac{a_p^2-(r-2)\binom{r-1}{2}p^3}{pa_p} = \frac{a_p^2 - \binom{r}{3}p^3}{pa_p} + \frac{\binom{r}{3}p^3 - (r-2)\binom{r-1}{2}p^3}{pa_p} = \tilde{c}  + \frac{2(3-r)\binom{r-1}{2}p^2}{3a_p} .
\end{align*}
We use the non-archimedean properties of the valuation $v$. We have $\tau \geq \text{min} (\tilde{\tau}, v(\frac{2(r-3)\binom{r-1}{2}p^2}{a_p} ))$. Moreover, since $v(\frac{(3-r)\binom{r-1}{2}p^2}{a_p}) \geq t + \frac{1}{2}$, we get that if $\tilde{\tau} < t+ \frac{1}{2}$, then $\tau = \tilde{\tau}$ and if $\tilde{\tau} \geq t+ \frac{1}{2}$, then $\tau \geq t+ \frac{1}{2}$, proving part $\textit{(i)}$ and $\textit{(ii)}$. The proof of parts $\textit{(iii)}$ and $\textit{(iv)}$ are similar. %to the proof of parts $\textit{(i)}$ and $\textit{(ii)}$.
\end{proof}

We now state and prove several lemmas involving the action of the Hecke operators on simple functions supported mod $KZ$
on a line segment on the tree. All the lemmas are characterized by a simple common property: due to a telescoping
cancellation property, the functions obtained after applying the Hecke operator have a very simple form with support mod $KZ$ on
vertices very close to the origin of the tree.
  
We make some remarks about notation. Let $f$, $g \in \ind_{KZ}^G \> \Sym^r\mathbb{\bar{Q}}_p^2$. By $f = g + O(p^s)$, for some $s\geq 0$, we shall mean that the difference $f-  g \in \ind_{KZ}^G \> \Sym^r\mathbb{\bar{Z}}_p^2$ is integral and is divisible by $p^s$. Also, 
in the proofs we will often compute in `radius $n$'. This means that we are computing the part of a function which is supported mod $KZ$
at distance $n$ from the origin of the tree.
This `$n$'  is not to be confused with the `$n$' occurring in the statements of the lemmas below.

The first lemma will be used in the proof of Proposition \ref{F_0 when tau >= t} and \ref{F_1 when tau <= t}.
%\vspace{5mm}
\begin{lemma}\label{chi}
Let $p>3$, $r \geq 2p+1$, $r=3 + n(p-1)p^t$, with $t=v(r-3)$, $v(a_p)= \frac{3}{2}$, $\tilde{c} =\frac{a_p^2-\binom{r}{3}p^3}{pa_p}$, $v(\tilde{c}) =\tilde{\tau}$ and let $t_0 =\mathrm{min}\lbrace t,\tilde{\tau}\rbrace$. If 
\begin{align*}
\chi = \sum_{n=0}^{t}a_p^n[g^0_{n,0}, Y^r - X^{r-3}Y^3],
\end{align*}
then\begin{align*}
(T-a_p)\chi= [\alpha, Y^r] + a_p[1, X^{r-3}Y^3] +p^{t+1}h_{\chi} + O(p^{t_0+2}),
\end{align*}
where $h_\chi$ is an integral linear combination of the terms of the form $[g, X^r]$ and $[g,X^{r-1}Y],$ for $g\in G.$
\end{lemma}
\begin{proof}
Write $\chi = \sum\limits_{n=0}^{t}\chi_n$, where $\chi_n := a_p^n[g^0_{n,0}, Y^r -X^{r-3}Y^3]$, for $n\geq 0$. By the formula for the Hecke operator, in radius $-1$ we have
\begin{align}
T^-\chi_0 = [\alpha, Y^r-(pX)^{r-3}Y^3]= [\alpha, Y^r] + O(p^{t_0+3}), \label{chi radius -1}
\end{align} because $p>3$, so $r-3 \geq (p-1)p^t\geq 4p^t > t+3\geq t_0 + 3$. 

In radius $0$ we have $-a_p\chi_0 = -a_p[1, Y^r-X^{r-3}Y^3]$,
\begin{align*}
T^-\chi_1&= a_p[g^0_{1,0}\alpha, Y^r- (pX)^{r-3}Y^3] = a_p[1, Y^r] + O(p^{t_0+3}). 
\end{align*}
So finally in radius $0$ we have
\begin{align} -a_p\chi_0 +T^-\chi_1 = a_p[1,X^{r-3}Y^3] + O(p^{t_0+3}). \label{chi radius 0}
\end{align}
Now for radius $n$ with $1\leq n\leq t-1,$ we compute $T^-\chi_{n+1}-a_p\chi_n+ T^+\chi_{n-1}$:
\begin{align*}
T^-\chi_{n+1} &= a_p^{n+1}[g^0_{n+1,0}\alpha, Y^r -(pX)^{r-3}Y^3] = a_p^{n+1}[g^0_{n,0}, Y^r] + O(p^{t_0+3}). 
\end{align*}
Also we have $-a_p\chi_n = -a_p^{n+1}[g^0_{n,0}, Y^r-X^{r-3}Y^3]$ and 
\begin{align*}
T^+\chi_{n-1} &= a_p^{n-1}\Bigg(\sum_{\lambda\in\mathbb{F}_p}[g^{0}_{n-1,0}g^0_{1,[\lambda]}, (-[\lambda]X+pY)^r- X^{r-3}(-[\lambda]X+pY)^3]\Bigg)\\
&= a_p^{n-1}\Bigg(\sum_{\lambda\in\mathbb{F}_p^{\times}}\Bigg[g^0_{n-1,0}g^0_{1,[\lambda]}, (-[\lambda]X)^r +pr(-[\lambda]X)^{r-1}Y+p^2\binom{r}{2}(-[\lambda]X)^{r-2}Y^2\\
&~~~  + p^3\binom{r}{3}(-[\lambda]X)^{r-3}Y^3+ \sum_{j\geq 4}p^j\binom{r}{j}(-[\lambda]X)^{r-j}Y^j- X^{r-3}(-[\lambda]^3X^3+3p[\lambda]^2X^2Y \\
&~~~-3p^2[\lambda]XY^2+p^3Y^3)\Bigg] + [g^0_{n-1,0}g^0_{1,0}, p^rY^r-p^3X^{r-3}Y^3]\Bigg). 
\end{align*}
By Lemma \ref{lemma 0.1.}, we get 
\begin{align*}
T^+\chi_{n-1} &= a_p^{n-1}\Bigg(\sum_{\lambda\in\mathbb{F}_p^{\times}}\Bigg[g^0_{n-1,0}g^0_{1,[\lambda]}, (r-3)p[\lambda]^2X^{r-1}Y+\left(3-\binom{r}{2}\right)p^2[\lambda]X^{r-2}Y^2\\
&~~~+ \left(\binom{r}{3}-1\right)p^3X^{r-3}Y^3\Bigg] - [g^0_{n,0}, p^3X^{r-3}Y^3]\Bigg) +O(p^{t_0+4}).
\end{align*}
Note that $v(3-\binom{r}{2})\geq t$ and $v(\binom{r}{3}-1) \geq t$. So we get
\begin{align*}
T^+\chi_{n-1} = a_p^{n-1}(r-3)p\sum_{\lambda\in\mathbb{F}_p^\times}[\lambda]^2[g^0_{n,p^{n-1}[\lambda]}, X^{r-1}Y] - a_p^{n-1}p^3[g^0_{n,0}, X^{r-3}Y^3]  + O(p^{t_0+2}).
\end{align*}
Therefore in radius $n$, for $1\leq n \leq t-1$, we have
 \begin{align*}
 T^-\chi_{n+1}-a_p\chi_n + T^+\chi_{n-1} = a_p^{n+1}[g^0_{n,0}, X^{r-3}Y^3]-a_p^{n-1}p^3[g^0_{n,0}, X^{r-3}Y^3] + p^{t+1}h_n +O(p^{t_0 +2}),
\end{align*}
where $h_n$ is an integral linear combination of terms of the form $[g, X^{r-1}Y]$ for $g\in G$. As $\tilde{c}=\frac{a_p^2-\binom{r}{3}p^3}{pa_p}$, so $a_p^2-p^3 = \tilde{c}pa_p + p^3\left(\binom{r}{3}-1\right).$ Hence $v(a_p^2-p^3)$= min$\lbrace v(\tilde{c})+\frac{5}{2}, t+3\rbrace > t_0+ 2$. We finally get
\begin{align}
T^-\chi_{n+1} -a_p\chi_n+ T^+\chi_{n-1}= p^{t+1}h_n + O(p^{t_0+ 2}). \label{chi radius 1 to t-1}
\end{align}
Similarly for $n=t\geq 1$, we compute $-a_p\chi_t+ T^+\chi_{t-1}$:

$-a_p\chi_t=-a_p^{t+1}[g^0_{t,0}, Y^r-X^{r-3}Y^3]$. Also
\begin{align*}
T^+\chi_{t-1} &= a_p^{t-1}\Bigg(\sum_{\lambda\in\mathbb{F}_p^{\times}}[g^0_{t-1,0}g^0_{1,[\lambda]}, (-[\lambda]X+pY)^r-X^{r-3}(-[\lambda]X+pY)^3]\\
&~~~+[g^0_{t-1,0}g^0_{1,0}, (pY)^r-X^{r-3}(pY)^3]\Bigg)\\
&= a_p^{t-1}\Bigg(\sum_{\lambda\in\mathbb{F}_p^\times}\left[g^0_{t-1,0}g^0_{1,[\lambda]},(-[\lambda]X)^r+pr(-[\lambda]X)^{r-1}Y + p^2\binom{r}{2}(-[\lambda]X)^{r-2}Y^2\right.\\
&~~~+p^3\binom{r}{3}(-[\lambda]X)^{r-3}Y^3 + \sum_{j\geq 4}p^j\binom{r}{j}(-[\lambda]X)^{r-j}Y^j-X^{r-3}(-[\lambda]^3X^3+3p[\lambda]^2X^2Y\\
&~~~-3p^2[\lambda]XY^2 + p^3Y^3)\bigg] +[g^0_{t,0}, -p^3X^{r-3}Y^3]\Bigg) + O(p^{t_0+4}),
\end{align*}
 because from Lemma \ref{lemma 0.1.}, $p^j\binom{r}{j} \equiv 0~\text{mod}~p^{t+4}, \forall j \geq 4$. As $v\left(3-\binom{r}{2}\right)\geq t$ and $v\left(\binom{r}{3}-1\right)\geq t$, we get\begin{align*}
T^+\chi_{t-1}= a_p^{t-1}(r-3)p\sum_{\lambda\in\mathbb{F}_p^\times}[\lambda]^2[g_{t, p^{t-1}[\lambda]}, X^{r-1}Y] - a_p^{t-1}p^3[g^0_{t,0}, X^{r-3}Y^3] + O(p^{t_0+2}).
\end{align*}
As before, note that $v(a_p^2- p^3) > t_0+2$, so we get
\begin{align}
-a_p\chi_t+ T^+\chi_{t-1}= p^{t+1}h_t +O(p^{t_0+2}),  \label{chi radius t}
\end{align}
where $h_t$ is an integral linear combination of the terms of the form $[g, X^r]$ and $[g, X^{r-1}Y]$, for $g\in G$.

Finally in radius $t+1$, we compute
\begin{align}
T^+\chi_t &= a_p^t\left(\sum_{\lambda\in\mathbb{F}_p^\times}[g^0_{t,1}g^0_{1,[\lambda]}, (-[\lambda]X+pY)^r-X^{r-3}(-[\lambda]X+pY)^3]+[g^0_{t+1,0},(pY)^r-p^3X^{r-3}Y^3]\right) \nonumber \\
&= p^{t+1}h_{t+1} + O(p^{t_0+2}), \label{chi radius t+1}
\end{align}
where $h_{t+1}$ is a linear combination of the terms of the form $[g, X^{r-1}Y]$ for $g\in G.$\\

Using equations \eqref{chi radius -1}, \eqref{chi radius 0}, \eqref{chi radius 1 to t-1}, \eqref{chi radius t}, \eqref{chi radius t+1}, we get
\begin{align*}
(T-a_p)\chi &=T^-\chi_0-a_p\chi_0+T^-\chi_1+\sum_{n=1}^{t-1}(T^-\chi_{n+1}-a_p\chi_n+T^+\chi_{n-1})-a_p\chi_t +T^+\chi_{t-1}+T^+\chi_t\\
&= [\alpha, Y^r]+a_p[1,X^{r-3}Y^3]+ p^{t+1}h_\chi + O(p^{t_0+2}),
\end{align*}
where $h_\chi$ is an integral linear combination of the terms of the form $[g, X^r]$ and $[g, X^{r-1}Y]$, for $g\in G.$
\end{proof}

The following lemma will be used in the proof of Proposition \ref{F_1 when tau > t + 0.5}.

\begin{lemma}\label{lemma 4.2} Let $p>3$, $r\geq 2p+1$, $r=3+n(p-1)p^t$, with $t=v(r-3)$, $v(a_p)= \frac{3}{2}$ and $c=\frac{a_p^2-(r-2)\binom{r-1}{2}p^3}{pa_p}$. Let $\tau =v(c)$ and assume $\tau > t+ \frac{1}{2}$. If
\begin{align*}
\xi = \frac{-1}{p^2(3-r)}\sum_{n = 0}^{2t+1}\left(\frac{a_p}{p}\right)^{n+1}[g^0_{n,0}, X^{r-2}Y^2 + (r-3)X^pY^{r-p}- (r-2)XY^{r-1}],
\end{align*} then
\begin{align*}
(T-a_p)\xi = \frac{a_p(r-2)}{p^2(3-r)}[\alpha, XY^{r-1}] + \frac{a_p^2}{p^3(3-r)}[1, X^{r-2}Y^2+ (r-3)X^pY^{r-p}] + O(\sqrt{p}).
\end{align*}
\end{lemma}
\begin{proof}
 Let $\xi_n = \frac{-1}{p^2(3-r)}\left(\frac{a_p}{p}\right)^{n+1}[g^0_{n,0}, X^{r-2}Y^2+(r-3)X^pY^{r-p}-(r-2)XY^{r-1}]$. So we can write $\xi = \sum\limits_{n=0}^{2t+1}\xi_n$. We will compute $(T-a_p)\xi$ in several steps.

Computation in radius $-1$:
\begin{align}
T^-\xi_0 &= \frac{-1}{p^2(3-r)}\left(\frac{a_p}{p}\right)[\alpha, (pX)^{r-2}Y^2+ (r-3)(pX)^{p}Y^{r-p}-(r-2)pXY^{r-1}] \nonumber \\
&= \frac{a_p(r-2)}{p^2(3-r)}[\alpha, XY^{r-1}] + O(p^2), \label{xi radius -1}
\end{align}
because $p>3$, so $r-2 \geq (p-1)p^t +1 > t+4$.

Computation in radius $0$:
\begin{align*}
-a_p\xi_0 &= \frac{a_p^2}{p^3(3-r)}[1, X^{r-2}Y^2 + (r-3)X^pY^{r-p} - (r-2)XY^{r-1}],\\
T^-\xi_1 &= \frac{-1}{p^2(3-r)}\left(\frac{a_p}{p}\right)^2[1, (pX)^{r-2}Y^2 + (r-3)(pX)^pY^{r-p} - (r-2)pXY^{r-1}]\\
&= \frac{a_p^2(r-2)}{p^3(3-r)}[1, XY^{r-1}] + O(p^2). 
\end{align*}
Therefore in radius $0$ we get
\begin{align}
-a_p\xi_0 + T^-\xi_1 = \frac{a_p^2}{p^3(3-r)}[1, X^{r-2}Y^2 + (r-3)X^pY^{r-p}] + O(p^2). \label{xi radius 0}
\end{align}

Computation in radius $n$, when $1\leq n \leq 2t$:
\begin{align*}
-a_p\xi_n &= \frac{a_p}{p^2(3-r)}\left(\frac{a_p}{p}\right)^{n+1}[g^0_{n,0}, X^{r-2}Y^2 + (r-3)X^pY^{r-p} -(r-2)XY^{r-1}],\\
T^-\xi_{n+1} &= \frac{-1}{p^2(3-r)}\left(\frac{a_p}{p}\right)^{n+2}[g^0_{n+1,0}\alpha, (pX)^{r-2}Y^2 + (r-3)(pX)^{p}Y^{r-p} - (r-2)pXY^{r-1}]\\
&=\frac{r-2}{p(3-r)}\left(\frac{a_p}{p}\right)^{n+2}[g^0_{n,0}, XY^{r-1}] + O(p^2).
\end{align*} 
So $-a_p\xi_n + T^-\xi_{n+1}= \frac{a_p}{p^2(3-r)}\left(\frac{a_p}{p}\right)^{n+1}[g^0_{n,0}, X^{r-2}Y^2] +  O(\sqrt{p})$.

We will calculate $T^+\xi_{n-1}$, for $1\leq n \leq 2t + 2$:
\begin{align*}
T^+\xi_{n-1} &= \frac{-1}{p^2(3-r)}\left(\frac{a_p}{p}\right)^n\sum_{\lambda \in\mathbb{F}_p}[g^0_{n-1,0}g^0_{1,[\lambda]}, X^{r-2}(-[\lambda]X+pY)^2 + (r-3)X^p(-[\lambda]X + pY)^{r-p}\\
&~~~ - (r-2)X(-[\lambda]X+pY)^{r-1}]\\
&= \frac{-1}{p^2(3-r)}\left(\frac{a_p}{p}\right)^n\bigg(\sum_{\lambda\in\mathbb{F}_p^\times}[g^0_{n,p^{n-1}[\lambda]}, X^{r-2}(-[\lambda]X+pY)^2 + (r-3)X^p(-[\lambda]X + pY)^{r-p}\\
&~~~ - (r-2)X(-[\lambda]X+pY)^{r-1}] + [g^0_{n,0}, p^2X^{r-2}Y^2]\bigg) + O(p^3).
\end{align*}
The coefficient of $[g^0_{n,p^{n-1}[\lambda]}, X^r]$ in $T^+\xi_{n-1}$ for $\lambda\neq 0$ is $0$.
The coefficient of $[g^0_{n,p^{n-1}[\lambda]}, X^{r-1}Y]$ in $T^+\xi_{n-1}$ for $\lambda\neq 0$ is
\begin{align*}
\frac{[\lambda]}{p^2(3-r)}\left(\frac{a_p}{p}\right)^n(2p+p(r-3)(r-p)-p(r-2)(r-1))= [\lambda]\left(\frac{a_p}{p}\right)^n = O(\sqrt{p}).
\end{align*}
The coefficient of $[g^0_{n,p^{n-1}[\lambda]}, X^{r-2}Y^2]$ in $T^+\xi_{n-1}$ for $\lambda\neq 0$ is 
\begin{align*}
\frac{-1}{p^2(3-r)}\left(\frac{a_p}{p}\right)^n\left(p^2 + p^2(r-3)\binom{r-p}{2} - p^2(r-2)\binom{r-1}{2}\right) = O(\sqrt{p}).
\end{align*}
Finally using Lemma \ref{lemma 0.2} for the remaining terms in $T^+\xi_{n-1}$, we get 
\begin{align}
T^+\xi_{n-1}= \frac{-1}{(3-r)}\left(\frac{a_p}{p}\right)^n[g^0_{n,0}, X^{r-2}Y^2] + O(\sqrt{p}). \label{xi T plus n}
\end{align}
Therefore in radius $n$ for $1\leq n \leq 2t$, we have
\begin{align*}
T^-\xi_{n+1}-a_p\xi_n + T^+\xi_{n-1} = \frac{1}{(3-r)}\left(\frac{a_p}{p}\right)^n\left(\frac{a_p^2-p^3}{p^3}\right)[g^0_{n,0}, X^{r-2}Y^2] + O(\sqrt{p}).
\end{align*}
Now $a_p^2-p^3 = pa_pc + (r-2)\binom{r-1}{2}p^3- p^3$, so we get 
\begin{align}
v(a_p^2-p^3)\geq~\text{min}\left(\tau + \frac{5}{2}, t + 3\right). \label{v(a_p^2-p^3)}
\end{align}
 As $\tau > t+\frac{1}{2}$, we get $v(a_p^2-p^3) = t+3$. Therefore
\begin{align}
T^-\xi_{n+1}-a_p\xi_n + T^+\xi_{n-1} = O(\sqrt{p}). \label{xi radius n}
\end{align}

Computation in radius $2t+1$:
\begin{align*}
-a_p\xi_{2t+1} = \frac{a_p}{p^2(3-r)}\left(\frac{a_p}{p}\right)^{2t+2}[g^0_{2t+1,0}, X^{r-2}Y^2 + (r-3)X^pY^{r-p} - (r-2)XY^{r-1}] = O(\sqrt{p}).
\end{align*}
Using equation \eqref{xi T plus n}, we get 
\begin{align*}
T^+\xi_{2t} = \frac{-1}{(3-r)}\left(\frac{a_p}{p}\right)^{2t+1}[g^0_{2t+1,0}, X^{r-2}Y^2] + O(\sqrt{p}) = O(\sqrt{p}).
\end{align*}
Therefore in radius $2t+1$, we have
\begin{align}
-a_p\xi_{2t+1} + T^+\xi_{2t} = O(\sqrt{p}). \label{xi radius 2t+1}
\end{align}

Computation in radius $2t+2$:

By equation \eqref{xi T plus n}, we get
\begin{align}
T^+\xi_{2t+1} = O(p). \label{xi radius 2t+2}
\end{align}

Finally, by equations \eqref{xi radius -1}, \eqref{xi radius 0}, \eqref{xi radius n}, \eqref{xi radius 2t+1}, \eqref{xi radius 2t+2}, we get
\begin{align*}
(T-a_p)\xi = \frac{a_p(r-2)}{p^2(3-r)}[\alpha, XY^{r-1}] + \frac{a_p^2}{p^3(3-r)}[1, X^{r-2}Y^2+ (r-3)X^pY^{r-p}] + O(\sqrt{p}).
\end{align*}
\end{proof}

The next two lemmas will be used in the proof of Proposition \ref{F_2 when tau =< t}.
%\vspace{5mm}

\begin{lemma}{\label{chi prime}}
Let $p>3$, $ r \geq 2p+1$, $r = 3+n(p-1)p^t$, with $t=v(r-3)$, $v(a_p)=\frac{3}{2}$, $c=\frac{a_p^2-(r-2)\binom{r-1}{2}p^3}{pa_p}$. Let $\tau =v(c)$ and assume $\tau \leq t$.  If \begin{align*}
\chi' = \sum_{n = 0}^{t}a_p^n[g^0_{n,0}, [\lambda]^{-1}(Y^r- X^{r-3}Y^3)],~\text{for some}~\lambda\neq 0,
\end{align*}  then 
\begin{align*}
(T-a_p)\chi' = [\alpha, [\lambda]^{-1}Y^r] + a_p[1,[\lambda]^{-1}X^{r-3}Y^3] + p^{t+1}h_{\chi'} + O(p^{\tau+2}),
\end{align*}
where $h_{\chi'}$ is an integral linear combination of terms of the form $[g, X^r]$ and $[g, X^{r-1}Y]$, for $g\in G.$
\end{lemma}
\begin{proof}
Since $\tau \leq t$, by Lemma \ref{Lemma 5.1}, we have $\tau = \tilde{\tau}$. Let $\lambda\in \mathbb{F}_p^\times$, by Lemma \ref{chi} we have the following:
\begin{align*}
(T-a_p)\chi' = [\alpha, [\lambda]^{-1}Y^r] + a_p[1, [\lambda]^{-1}X^{r-3}Y^3] + p^{t+1}h_{\chi'} + O(p^{\tau +2}),
\end{align*} 
where $h_{\chi'}$ is an integral linear combination of terms of the form $[g, X^r]$ and $[g, X^{r-1}Y]$, for some $g\in G$.
\end{proof}
%\vspace{5mm}

\begin{lemma}{\label{phi}}
 Let $p>3$, $r \geq 2p+1$, $r=3+n(p-1)p^t$, with $t=v(r-3)$, $v(a_p)=\frac{3}{2}$ and $c=\frac{a_p^2-(r-2)\binom{r-1}{2}p^3}{pa_p}$. Let $\tau =v(c)$ and assume $\tau \leq t$. If
 \begin{align*}
\phi = \sum_{n = 0}^{2t+1}\left(\frac{p^2}{a_p}\right)^n[g^0_{n,0}, X^{r-2}Y^2 - XY^{r-1}],
\end{align*} then
\begin{align*}
(T-a_p)\phi = -[\alpha, pXY^{r-1}] -a_p[1, X^{r-2}Y^2] + p^{\tau+1}h_{\phi} + O(p^{t + 2}),
\end{align*} where $h_{\phi}$ is an integral linear combination of the terms of the form $[g, X^{r-1}Y]$ and $[g, XY^{r-1}]$, for some $g\in G$.
\end{lemma}
\begin{proof}
Let $\phi_n = \left(\frac{p^2}{a_p}\right)^n[g^0_{n,0}, X^{r-2}Y^2- XY^{r-1}]$, so we can write $\phi= \sum\limits_{n=0}^{2t+1}\phi_n$. We will compute $(T-a_p)\phi$ in several steps.

Computation in radius $-1$:
\begin{align}
T^-\phi_0 = [\alpha, (pX)^{r-2}Y^2 - pXY^{r-1}] = [\alpha, -pXY^{r-1}] + O(p^{t+4}), \label{phi radius -1}
\end{align}
because $p>3$, so $r-2 \geq (p-1)p^t + 1 > t + 4$.

Computation in radius $0$:
\begin{align*}
-a_p\phi_0  &= -a_p[1, X^{r-2}Y^2-XY^{r-1}], \\
T^-\phi_1 &= \frac{p^2}{a_p}[1, (pX)^{r-2}Y^2 -pXY^{r-1}] = \frac{-p^3}{a_p}[1, XY^{r-1}] + O(p^{t+4}).
\end{align*}
So in radius $0$ we have 
\begin{align*}
-a_p\phi_0 + T^-\phi_1 = -a_p[1, X^{r-2}Y^2] + \left(\frac{a_p^2-p^3}{a_p}\right)[1, XY^{r-1}] + O(p^{t+4}).
\end{align*}
As $\tau \leq t$, by equation \eqref{v(a_p^2-p^3)}, we have $v(a_p^2 - p^3)= \tau + \frac{5}{2}$. 

Therefore
\begin{align}
-a_p\phi_0 + T^-\phi_1 = -a_p[1, X^{r-2}Y^2] + p^{\tau + 1}[1, XY^{r-1}] + O(p^{t+4}). \label{phi radius 0}
\end{align}

Computation in radius $n$ with $1\leq n \leq 2t$:
\begin{align*}
 -a_p\phi_n &= -a_p\left(\frac{p^2}{a_p}\right)^n[g^0_{n,0}, X^{r-2}Y^2- XY^{r-1}], \\
T^-\phi_{n+1}&= \left(\frac{p^2}{a_p}\right)^{n+1}[g^0_{n,0}, (pX)^{r-2}Y^2 - pXY^{r-1}] = \left(\frac{p^2}{a_p}\right)^{n+1}[g^0_{n,0}, -pXY^{r-1}] + O(p^{t+4}).
\end{align*}
So we get
\begin{align*}
-a_p\phi_n + T^-\phi_{n+1} &= -a_p\left(\frac{p^2}{a_p}\right)^n[g^0_{n,0}, X^{r-2}Y^2] + \left(\frac{p^2}{a_p}\right)^n\left(\frac{a_p^2-p^3}{a_p}\right)[g^0_{n,0}, XY^{r-1}] + O(p^{t+4}) \\
&= -a_p\left(\frac{p^2}{a_p}\right)^n[g^0_{n,0}, X^{r-2}Y^2] + p^{\tau + \frac{3}{2}}[g^0_{n,0}, XY^{r-1}] + O(p^{t+4}).
\end{align*}
We will compute $T^+\phi_{n-1}$ for $1\leq n \leq 2t+2$:
\begin{align*}
T^+\phi_{n-1} &= \left(\frac{p^2}{a_p}\right)^{n-1}\sum_{\lambda\in\mathbb{F}_p}[g^0_{n-1,0}g^0_{1,[\lambda]}, X^{r-2}(-[\lambda]X+ pY)^2 - X(-[\lambda]X+pY)^{r-1}]\\
&= \left(\frac{p^2}{a_p}\right)^{n-1}\left(\sum_{\lambda\neq 0}[g^0_{n, p^{n-1}[\lambda]}, X^{r-2}(-[\lambda]X+ pY)^2 - X(-[\lambda]X+pY)^{r-1}] + [g^0_{n,0}, p^2X^{r-2}Y^2]\right)\\
&~~~ + O(p^{t+5}).
\end{align*} 
Coefficient of $[g^0_{n,p^{n-1}[\lambda]}, X^r]$ in $T^+\phi_{n-1}$ for $\lambda\neq 0$ is $0$.

\noindent Coefficient of $[g^0_{n,p^{n-1}[\lambda]}, X^{r-1}Y]$ in $T^+\phi_{n-1}$ for $\lambda\neq 0$ is $[\lambda]p\left(\frac{p^2}{a_p}\right)^{n-1}(r-3) = O(p^{t+1})$. 

\noindent Coefficient of $[g^0_{n,p^{n-1}[\lambda]}, X^{r-2}Y^2]$ in $T^+\phi_{n-1}$ for $\lambda \neq 0$ is $p^2\left(\frac{p^2}{a_p}\right)^{n-1}\left(1-\binom{r-1}{2}\right) = O(p^{t+2})$.
 
Using Lemma \ref{lemma 0.2}, we have
\begin{align}
T^+\phi_{n-1} = p^2\left(\frac{p^2}{a_p}\right)^{n-1}[g^0_{n,0}, X^{r-2}Y^2] + p^{t+1}h'_n + O(p^{t+2}), \label{phi T plus radius n}
\end{align}
where $h'_n$ is an integral linear combination of the terms of the form $[g, X^{r-1}Y]$, for some $g\in G$.

Therefore in radius $n$, for $1\leq n \leq 2t$, we have 
\begin{align}
T^-\phi_{n+1} -a_p\phi_n + T^+\phi_{n-1} = p^{\tau + \frac{3}{2}}[g^0_{n,0}, XY^{r-1}] + p^{t+1}h'_n + O(p^{t+2}). \label{phi radius n}
\end{align}

Computation in radius $2t+1$:
\begin{align*}
-a_p\phi_{2t+1} &= -a_p\left(\frac{p^2}{a_p}\right)^{2t+1}[g^0_{2t+1,0}, X^{r-2}Y^2 - XY^{r-1}] = O(p^{t+2}).
\end{align*}
>From equation \eqref{phi T plus radius n} we get that $T^+\phi_{2t} = p^{t+1}h'_{2t+1} + O(p^{t+2})$, where $h'_{2t+1}$ is an integral linear combination of the terms of the form $[g, X^{r-1}Y]$, for some $g\in G$. Therefore in radius $2t+1$ we get
\begin{align}
-a_p\phi_{2t+1} + T^+\phi_{2t} =p^{t+1}h'_{2t+1} + O(p^{t+2}). \label{phi radius 2t+1}
\end{align}

Computation in radius $2t+2$:

By equation \eqref{phi T plus radius n}, we get 
\begin{align}
T^+\phi_{2t+1} = p^{t+1}h'_{2t+2} + O(p^{t+ \frac{5}{2}}), \label{phi radius 2t+2}
\end{align}
where $h'_{2t+2}$ is an integral linear combination of the terms of the form $[g, X^{r-1}Y]$, for some $g\in G$.

Finally by equations \eqref{phi radius -1}, \eqref{phi radius 0}, \eqref{phi radius n}, \eqref{phi radius 2t+1}, \eqref{phi radius 2t+2}, we get
\begin{align*}
(T-a_p)\phi = -[\alpha, pXY^{r-1}] -a_p[1, X^{r-2}Y^2] + p^{\tau+1}h_\phi + O(p^{t+2}),
\end{align*}
where $h_\phi$ is an integral linear combination of the terms of the form $[g, X^{r-1}Y]$ and $[g, XY^{r-1}]$, for some $g\in G$.
\end{proof}

The next two lemmas will be used in the proof of Proposition \ref{F_2 when tau = > t + 1}.
%\vspace{5mm}

\begin{lemma}{\label{lemma 0.18}} Let $p> 3$, $r \geq 2p+1$, $r = 3 + n(p-1)p^t$, with $t = v(r-3)$, $v(a_p) =\frac{3}{2}$ and $c = \frac{a_p^2-(r-2)\binom{r-1}{2}p^3}{pa_p}$. Let $\tau = v(c)$ and assume $\tau > t + \frac{1}{2}$. If
 \begin{align*}
\xi' = \frac{1}{p^2}\sum_{n =1}^{2t+2}\left(\frac{a_p}{p}\right)^{n}[g^0_{n,0}, X^{r-2}Y^2 + (r-3)X^pY^{r-p} - (r-2)XY^{r-1}],
\end{align*} then
\begin{align*}
(T-a_p)\xi' =  \frac{a_p}{p^2}[1, (2-r)XY^{r-1}] -\frac{a_p^2}{p^3}[g^0_{1,0}, X^{r-2}Y^2 + (r-3)X^pY^{r-p}] + O(p^{t+\frac{1}{2}}).
\end{align*}\end{lemma}
\begin{proof}
We can write \begin{align*}
\xi' &= \frac{a_p}{p^3}\sum_{n =1}^{2t+2}\left(\frac{a_p}{p}\right)^{n-1}[g^0_{1,0}g^0_{n-1,0}, X^{r-2}Y^2 + (r-3)X^pY^{r-p} - (r-2)XY^{r-1}]\\
&= \frac{a_p}{p^3}\sum_{n = 0}^{2t+1}\left(\frac{a_p}{p}\right)^n[g^0_{1,0}g^0_{n,0}, X^{r-2}Y^2 + (r-3)X^pY^{r-p} - (r-2)XY^{r-1}].
\end{align*}
>From the definition of $\xi$ in Lemma \ref{lemma 4.2}, we can write 
\begin{align*}
\xi = \frac{-a_p}{p^3(3-r)}\sum_{n = 0}^{2t+1}\left(\frac{a_p}{p}\right)^n[g^0_{n,0}, X^{r-2}Y^2 + (r-3)X^pY^{r-p} - (r-2)XY^{r-1}].
\end{align*} So we can rewrite $\xi'$ in terms of $\xi$ as  
\begin{align*}
\xi' &= \frac{a_p}{p^3}\left(\frac{p^3(r-3)\xi_{g^0_{1,0}}}{a_p}\right)= (r-3)\xi_{g^0_{1,0}},
\end{align*} 
where $\xi_{g^0_{1,0}} = g^0_{1,0}\xi$.

So $(T-a_p)\xi' = (r-3)(T-a_p)\xi_{g^0_{1,0}} =(r-3) g^0_{1,0}(T-a_p)\xi$. By Lemma \ref{lemma 4.2} we get 
\begin{align*}
(T-a_p)\xi' &= (r-3)\left(\frac{a_p(r-2)}{p^2(3-r)}\right)[g^0_{1,0}\alpha, XY^{r-1}] + \frac{a_p^2(r-3)}{p^3(3-r)}[1, X^{r-2}Y^2 + (r-3)X^pY^{r-p}] + O(p^{t + \frac{1}{2}})\\
&= \frac{a_p}{p^2}[1,(2-r)XY^{r-1}]- \frac{a_p^2}{p^3}[1, X^{r-2}Y^2 + (r-3)X^pY^{r-p}] + O(p^{t+\frac{1}{2}}). \qedhere
\end{align*}
\end{proof}

\begin{lemma}{\label{Lemma 0.17.}}Let $p>3$, $r\geq 2p+1$, $r = 3 + n(p-1)p^t$, with $t = v(r-3)$, $v(a_p)=\frac{3}{2}$ and $c =\frac{a_p^2-(r-2)\binom{r-1}{2}p^3}{pa_p}$. Let $\tau = v(c)$ and assume $\tau \geq t+ 1$. If
 \begin{align*}
\psi = \frac{1}{p^2}\sum_{n=1}^{t+1}a_p^n[g^0_{n,[\mu]},[\mu]^{-1}(Y^r-X^{r-3}Y^3)], ~\text{for some}~\mu\in\mathbb{F}_p^\times,
\end{align*} then\begin{align*}
(T-a_p)\psi = \frac{a_p}{p^2}[1,[\mu]^{-1}([\mu]X+Y)^r]+\frac{a_p^2}{p^2}[g^0_{1,[\mu]}, [\mu]^{-1}X^{r-3}Y^3] + O(p^{t+\frac{1}{2}}).
\end{align*}
\end{lemma}
\begin{proof}
Let $\mu \in \mathbb{F}_p^\times$ and let $\psi_n = \frac{a_p^n}{p^2}[g^0_{n,[\mu]}, [\mu]^{-1}(Y^r - X^{r-3}Y^3)]$, for $n\geq 1$. So we can write $\psi = \sum\limits_{n=1}^{t+1}\psi_n$.

Computation in radius $0$:
\begin{align}
T^-\psi_1 = \frac{a_p}{p^2}[g^0_{1,[\mu]}\alpha, [\mu]^{-1}(Y^r -(pX)^{r-3}Y^3]
= \frac{a_p}{p^2}[1, [\mu]^{-1}([\mu]X+Y)^r] + O(p^{t+ \frac{5}{2}}), \label{psi radius 0}
\end{align} 
because $p>3$, so $r-3 \geq (p-1)p^t > t+3$.

Computation in radius $1$: 
\begin{align*}
-a_p\psi_1 &= -\frac{a_p^2}{p^2}[g^0_{1,[\mu]}, [\mu]^{-1}(Y^r - X^{r-3}Y^3)],\\
T^-\psi_2 &= \frac{a_p^2}{p^2}[g^0_{2,[\mu]}\alpha , [\mu]^{-1}(Y^r - (pX)^{r-3}Y^3)] \\
&= \frac{a_p^2}{p^2}[g^0_{1,[\mu]}, [\mu]^{-1}Y^r] + O(p^{t+4}).
\end{align*} 
Therefore in radius $1$ we get 
\begin{align}
-a_p\psi_1 + T^-\psi_2 = \frac{a_p^2}{p^2}[g^0_{1,[\mu]}, [\mu]^{-1}X^{r-3}Y^3] + O(p^{t+4}). \label{psi radius 1}
\end{align} 

Computation in radius $n$, with $2\leq n \leq t$:
\begin{align*}
- a_p\psi_n &= -\frac{a_p^{n+1}}{p^2}[g^0_{n,[\mu]}, [\mu]^{-1}(Y^r - X^{r-3}Y^3)],\\
T^-\psi_{n+1} & = \frac{a_p^{n+1}}{p^2}[g^0_{n+1,[\mu]}\alpha, [\mu]^{-1}(Y^r - (pX)^{r-3}Y^3)]\\
&= \frac{a_p^{n+1}}{p^2}[g^0_{n,[\mu]}, [\mu]^{-1}Y^r] + O(p^{t+\frac{11}{2}}).
\end{align*}
So $-a_p\psi_n + T^-\psi_{n+1} = \frac{a_p^{n+1}}{p^2}[g^0_{n,[\mu]}, [\mu]^{-1}X^{r-3}Y^3] + O(p^{t+\frac{11}{2}})$. 

We will compute $T^+ \psi_{n-1}$, with $2\leq n \leq t+2$.
\begin{align*}
T^+\psi_{n-1} &= \frac{a_p^{n-1}}{p^2}\sum_{\lambda\in\mathbb{F}_p}[g^0_{n-1, [\mu]}g^0_{1,[\lambda]}, [\mu]^{-1}((-[\lambda]X+pY)^r-X^{r-3}(-[\lambda]X+pY)^3)]\\
&= \frac{a_p^{n-1}}{p^2}\bigg(\sum_{{\lambda}\neq 0}[g^0_{n, [\mu]+ p^{n-1}[\lambda]}, [\mu]^{-1}((-[\lambda]X + pY)^r - X^{r-3}(-[\lambda]X + pY)^3) \\
&~~~+ [g^0_{n, [\mu]}, - [\mu]^{-1}p^3X^{r-3}Y^3]\bigg) + O(p^{t+ \frac{11}{2}}).
\end{align*}
Coefficient of $[g^0_{n, [\mu] + p^{n-1}[\lambda]}, X^r]$ in $T^+\psi_{n-1}$ for $\lambda\neq 0$ is $0$.

\noindent Coefficient of $[g^0_{n, [\mu] + p^{n-1}[\lambda]}, X^{r-1}Y]$ in $T^+\psi_{n-1}$ for $\lambda\neq 0$ is $[\mu]^{-1}[\lambda]^2\left(\frac{a_p^{n-1}}{p}\right)(r-3) = O(p^{t+ \frac{1}{2}})$.

\noindent Coefficient of $[g^0_{n, [\mu] + p^{n-1}[\lambda]}, X^{r-2}Y^2]$ in $T^+\psi_{n-1}$ for $\lambda\neq 0$ is $-[\mu]^{-1}[\lambda]a_p^{n-1}\left(\binom{r}{2} - 3\right) = O(p^{t+ \frac{3}{2}})$.

\noindent Coefficient of $[g^0_{n, [\mu] + p^{n-1}[\lambda]}, X^{r-3}Y^3]$ in $T^+\psi_{n-1}$ for $\lambda\neq 0$ is $[\mu]^{-1}pa_p^{n-1}\left(\binom{r}{3} - 1\right) = O(p^{t+ \frac{5}{2}})$.

Using Lemma \ref{lemma 0.1.} for the remaining terms in $T^+\psi_{n-1}$, we finally get 
\begin{align}
T^+\psi_{n-1} = -pa_p^{n-1}[g^0_{n,[\mu]}, [\mu]^{-1}X^{r-3}Y^3] + O(p^{t+ \frac{1}{2}}). \label{psi T plus computation}
\end{align}

So we have
\begin{align*}
T^-\psi_{n+1} -a_p\psi_n + T^+\psi_{n-1} = a_p^{n-1}\left(\frac{a_p^2 - p^3}{p^2}\right)[g^0_{n,[\mu]}, [\mu]^{-1}X^{r-3}Y^3] + O(p^{t+ \frac{1}{2}}).
\end{align*}
As $\tau \geq t+1$, by equation \eqref{v(a_p^2-p^3)}, we have $v(a_p^2-p^3) = t +3$.
Therefore in radius $n$, with $2\leq n \leq t$, we get 
\begin{align}
T^-\psi_{n+1} -a_p\psi_n + T^+\psi_{n-1} = O(p^{t+\frac{1}{2}}). \label{psi radius n}
\end{align}

Computation in radius $t+1$:
\begin{align*}
-a_p\psi_{t+1} &= -\frac{a_p^{t+2}}{p^2}[g^0_{t+1,[\mu]}, [\mu]^{-1}(Y^r-X^{r-3}Y^3)] = O(p^{t+1}).
\end{align*}
>From equation \eqref{psi T plus computation} we get that $T^+\psi_{t} = O(p^{t+\frac{1}{2}})$. Therefore in radius $t+1$ we have 
\begin{align}
-a_p\psi_{t+1} + T^+\psi_{t} = O(p^{t+\frac{1}{2}}). \label{psi radius t+1}
\end{align}

Computation in radius $t+2$:

>From equation \eqref{psi T plus computation} we get 
\begin{align}
T^+\psi_{t+1} = -pa_p^{t+1}[g^0_{t+2,[\mu]}, [\mu]^{-1}X^{r-3}Y^3] + O(p^{t+\frac{1}{2}})= O(p^{t+\frac{1}{2}}). \label{psi radius t+2}
\end{align}

Finally, by equations \eqref{psi radius 0}, \eqref{psi radius 1}, \eqref{psi radius n}, \eqref{psi radius t+1}, \eqref{psi radius t+2}, we get 
\begin{align*}
(T-a_p)\psi = \frac{a_p}{p^2}[1, [\mu]^{-1}([\mu]X + Y)^r] + \frac{a_p^2}{p^2}[g^0_{1,[\mu]},[\mu]^{-1}X^{r-3}Y^3] + O(p^{t+ \frac{1}{2}}).
\end{align*} 
\end{proof}

The last two lemmas of this section will be used in the proof of Proposition \ref{F_2 when tau < t + 1}.
%\vspace{5mm}

\begin{lemma}{\label{xi''}}
Let $p>3$, $r\geq 2p+1$, $r= 3+ n(p-1)p^t$, with $t = v(r-3)$, $v(a_p)=\frac{3}{2}$ and $c =\frac{a_p^2- (r-2)\binom{r-1}{2}p^3}{pa_p}$. Let $\tau = v(c)$ and assume $\tau < t+1$. If $$\xi''= \frac{1}{pc}\sum\limits_{n =1}^{2t+2}\left(\frac{a_p}{p}\right)^n[g^0_{n,0}, X^{r-2}Y^2+(r-3)X^pY^{r-p}-(r-2)XY^{r-1}],$$ then \begin{align*}
(T-a_p)\xi'' = \frac{a_p}{pc}[1,(2-r)XY^{r-1}]-\frac{a_p^2}{p^2c}[g^0_{1,0}, X^{r-2}Y^2] +O(p^\epsilon), 
\end{align*}
where $\epsilon =$ \emph{min}$\lbrace t+1-\tau, \frac{1}{2}\rbrace$.
\end{lemma}
\begin{proof}
We can write $\xi'' = \frac{p}{c}\xi'$, where $\xi'$ is as in Lemma \ref{lemma 0.18}. Therefore by Lemma \ref{lemma 0.18}, we get
\begin{align*}
(T-a_p)\xi'' &= \frac{p}{c}(T-a_p)\xi'\\
&= \frac{a_p}{pc}[1, (2-r)XY^{r-1}] - \frac{a_p^2}{p^2c}[1, X^{r-2}Y^2] + O(p^\epsilon),
\end{align*} 
because $\tau < t+1$, so $\frac{p}{c}O(p^{t+\frac{1}{2}})= O(p^\epsilon)$, where $\epsilon =$ min$\lbrace t+1 -\tau, \frac{1}{2} \rbrace$.
\end{proof}

%\vspace{5mm}

\begin{lemma}{\label{psi'}}
Let $p>3$, $r\geq 2p+1$, $r = 3 + n(p-1)p^t$, with $t=v(r-3)$, $v(a_p)=\frac{3}{2}$ and $c=\frac{a_p^2-(r-2)\binom{r-1}{2}p^3}{pa_p}$. Let $\tau = v(c)$ and assume $\tau < t+1$. If $$\psi'= \frac{1}{pc}\sum\limits_{n =1}^{t+1}a_p^n[g^0_{n,[\mu]},[\mu]^{-1}(Y^r-X^{r-3}Y^3)],~\text{for some}~\mu\in\mathbb{F}_p^\times,$$  then
\begin{align*}
(T-a_p)\psi' = \frac{a_p}{pc}[1,[\mu]^{-1}([\mu]X+Y)^r]+\frac{a_p^2}{pc}[g^0_{1,[\mu]},[\mu]^{-1}X^{r-3}Y^3] + O(\sqrt{p}).
\end{align*}
\end{lemma}
\begin{proof}
We can write $\psi' = \frac{p}{c}\psi$, where $\psi$ is as in Lemma \ref{Lemma 0.17.}. Therefore, by Lemma \ref{Lemma 0.17.}, we get
\begin{align*}
(T-a_p)\psi' &= \frac{p}{c}(T-a_p)\psi\\
&= \frac{a_p}{pc}[1,[\mu]^{-1}([\mu]X + Y)^r] + \frac{a_p^2}{pc}[g^0_{1,[\mu]}, [\mu]^{-1}X^{r-3}Y^3] + O(\sqrt{p}),
\end{align*}
because $\tau < t+1$, so $\frac{p}{c}O(p^{t+\frac{1}{2}})= O(\sqrt{p})$.
\end{proof}

\section{Behaviour of $F_1$}
\label{sectionF1}

We now begin to compute the structure of the subquotients of $\bar\Theta_{k,a_p}$.
In this section we compute the structure of the quotient $F_1$.  Recall that $\ind_{KZ}^G J_1 \twoheadrightarrow F_1$, where
$J_1 = V_{p-4} \otimes D^3$. 

In the next three sections we use the following notation:
\begin{itemize}
\item For a function $f\in$ ind$_{KZ}^G$Sym$^r\bar{\mathbb{Q}}_p^2$, we say that ``$f$ dies mod $p$''  if $f\in O(p^\epsilon)$, for some $\epsilon >0$, so that the reduction $\bar{f} \in$ ind$_{KZ}^G$Sym$^r\bar{\mathbb{F}}_p^2$ is identically zero.
\item For $g\in G$ and $\chi$, $\xi$, $\phi$, etc. functions in ind$_{KZ}^G$Sym$^r\bar{\mathbb{Q}}_p^2$ as in Section~\ref{sectiontelescope}, we let $\chi_g = g\cdot \chi$, $\xi_g = g \cdot \xi$, $\phi_g = g \cdot \phi$, etc. in ind$_{KZ}^G$Sym$^r\bar{\mathbb{Q}}_p^2$.\\
\end{itemize}

\begin{prop}{\label{F_0 when tau >= t}}
Let $p>3$, $r\geq 2p+1$, with $r\equiv 3~\emph{mod}~(p-1)$ and assume $v(a_p)= \frac{3}{2}$. Let $t= v(r-3)$, $c =\frac{a_p^2-(r-2)\binom{r-1}{2}p^3}{pa_p}$ and $\tau = v(c) $.
\begin{enumerate}[label =(\roman{*})]
\item If $\tau > t,$ then $F_1 =0.$ 
\item If $\tau =t,$ then $F_1$ is a quotient of $\pi(p-4,\lambda^{-1},\omega^3),$ where $\lambda =\overline{\frac{3}{3-r}\left(\frac{a_p^2-(r-2)\binom{r-1}{2}p^3}{pa_p}\right)}\in\mathbb{F}_p^\times.$
\end{enumerate}
\end{prop}
\begin{proof} The proof is a variant of the proof of Proposition 6.4 in \cite{BGR18}. 

Recall $\tilde{c} = \frac{a_p^2-\binom{r}{3}p^3}{pa_p}$ and $\tilde{\tau} = v(\tilde{c})$.
We claim that it is sufficient to prove the proposition with $\tau$ replaced by $\tilde{\tau}$ in \textit{(i)} and \textit{(ii)}.
Indeed, by Lemma~\ref{Lemma 5.1}, we see that if $\tilde{\tau} = t$, then $\tilde{\tau} < t+ \frac{1}{2}$, so $\tau = \tilde{\tau}$. Similarly, if $\tilde{\tau} > t$, then $\tau =\tilde{\tau}$ (when $t<\tilde{\tau} < t + \frac{1}{2}$) or $\tau \geq t+ \frac{1}{2}$ (when $\tilde{\tau} \geq t +0.5$), so in either case we have $\tau > t$.  
%Hence we prove the proposition for $\tilde{\tau}$ instead of $\tau.$

Let $S_{i} = \sum\limits_{\substack{{0<j<r}\\{j\equiv 3~\text{mod}~(p-1)}}}\binom{j}{i}\binom{r}{j}$,  for $i=0$, $1$, $2$. Note that $S_i = S_{i,r}$ in Proposition~ \ref{proposition 0.4.}.

Define the functions
\begin{align*}
f_0 &= \frac{p-1}{pa_p(3-r)}\left[1, \sum_{\substack{{0<j<r}\\{j\equiv 3~\text{mod}~(p-1)}}}\binom{r}{j}X^{r-j}Y^j- (AX^{r-3}Y^3+BX^{r-p-2}Y^{p+2})\right],\\
f_\infty &= \frac{1}{p(3-r)}\left(\sum_{\lambda\in\mathbb{F}_p^\times}\chi_{g^0_{1,[\lambda]}} +(1-p)\chi_{g^0_{1,0}}\right),
\end{align*}
where $A= \frac{(p+2)S_0-S_1}{p-1}$, $B = \frac{S_1-3S_0}{p-1}$ and $\chi$ is as in Lemma \ref{chi}.
Let $f=f_0+f_\infty. $ 

We will compute $(T-a_p)f$ in several steps.

Computation in radius $-1$:
\begin{align*}
T^-f_0 &= \frac{p-1}{pa_p(3-r)}\left[\alpha , \sum_{\substack{{0<j<r}\\{j\equiv 3~\text{mod}~(p-1)}}}\binom{r}{j}(pX)^{r-j}Y^j - (A(pX)^{r-3}Y^3+B(pX)^{r-p-2}Y^{p+2})\right]\\
           &= \frac{p-1}{pa_p(3-r)}\left[\alpha, \sum_{\substack{{0<j<r}\\{j\equiv 3~\text{mod}~(p-1)}}}p^{r-j}\binom{r}{r-j}X^{r-j}Y^j\right] + O(p),
\end{align*}
because $v(A)$, $v(B) \geq t+1$ and $r-3>r-p-2 \geq p-1>v(pa_p)= \frac{5}{2}$ as $p>3$ and $r\geq 2p+1$.
Using Lemma \ref{lemma 0.1.}, we get 
\begin{align}
T^-f_0 = O(p), \label{F_0 T minus f_0}
\end{align}
because $r-j \geq p-1 \geq 4$, for all $j$ as above. Now
\begin{align}
-a_pf_0 = \frac{1-p}{p(3-r)}\left[1, \sum_{\substack{{0<j<r}\\{j\equiv 3~\text{mod}~(p-1)}}}\binom{r}{j}X^{r-j}Y^j - (AX^{r-3}Y^3 + BX^{r-p-2}Y^{p+2})\right].\label{F_0 -a_pf_0}
\end{align}
Also
\begin{align*}
  T^+f_0 &= \frac{p-1}{pa_p(3-r)}\Bigg(\sum_{\lambda\in\mathbb{F}_p}\Bigg[g^0_{1,[\lambda]}, 
 \sum_{\substack{{0<j<r}\\{j\equiv 3~\text{mod}~(p-1)}}}\binom{r}{j}X^{r-j}(-[\lambda]X+pY)^j-(AX^{r-3}  
 (-[\lambda]X+pY)^3\\
&~~~+ BX^{r-p-2}([-\lambda]X+pY)^{p+2})\Bigg]\Bigg)\\
& = \frac{p-1}{pa_p(3-r)}\Bigg(\sum_{\lambda\in\mathbb{F}_p^\times}\Bigg[g^0_{1,[\lambda]}, \sum_{\substack{{0<j<r}\\{j\equiv 3~\text{mod}~(p-1)}}}\binom{r}{j}X^{r-j}\sum_{i=0}^{j}\binom{j}{i}(-[\lambda]X)^{j-i}(pY)^i \\
&~~~ - \bigg(AX^{r-3}\sum_{i=0}^3\binom{3}{i}(-[\lambda]X)^{3-i}(pY)^i + BX^{r-p-2}\sum_{i=0}^{p+2}\binom{p+2}{i}(-[\lambda]X)^{p+2-i}(pY)^i\bigg)\Bigg] \\
&~~~ + \Bigg[g^0_{1,0}, \sum_{\substack{{0<j<r}\\{j\equiv 3~\text{mod}~(p-1)}}}p^j\binom{r}{j}X^{r-j}Y^j - (Ap^3X^{r-3}Y^3 + Bp^{p+2}X^{r-p-2}Y^{p+2})\Bigg]\Bigg)\\
&=\frac{p-1}{pa_p(3-r)}\Bigg(\sum_{\lambda\in\mathbb{F}_p^\times}\Bigg[g^0_{1,[\lambda]}, \sum_{i=0}^rp^i(-[\lambda])^{3-i}\sum_{\substack{{0<j<r}\\{j\equiv 3~\text{mod}~(p-1)}}}\left(\binom{r}{j}\binom{j}{i}-A\binom{3}{i}-B\binom{p+2}{i}\right)X^{r-i}Y^i\Bigg]\\
&~~~+\Bigg[g^0_{1,0}, \sum_{\substack{{0<j<r}\\{j\equiv 3~\text{mod}~(p-1)}}}p^j\binom{r}{j}X^{r-j}Y^j-(Ap^3X^{r-3}Y^3+Bp^{p+2}X^{r-p-2}Y^{p+2})\Bigg]\Bigg).
\end{align*}
For $\lambda\neq 0$, the coefficients of $[g^0_{1,[\lambda]}, X^r]$, $[g^0_{1,[\lambda]}, X^{r-1}Y]$, $[g^0_{1,[\lambda]}, X^{r-2}Y^2]$ in $T^+f_0$ are given by 
\begin{align*}
&-[\lambda]^3\frac{p-1}{pa_p(3-r)}\left(\sum\limits_{{\substack{{0<j<r}\\{j\equiv 3~\text{mod}~(p-1)}}}}\binom{r}{j}- (A+B)\right),\\
& [\lambda]^2\frac{p-1}{a_p(3-r)}\left(\sum\limits_{{\substack{{0<j<r}\\{j\equiv 3~\text{mod}~(p-1)}}}}j\binom{r}{j}-(3A+(p+2)B)\right),\\ &-[\lambda]\frac{p(p-1)}{a_p(3-r)}\left(\sum\limits_{{\substack{{0<j<r}\\{j\equiv 3~\text{mod}~(p-1)}}}}\binom{j}{2}\binom{r}{j}-(3A+\binom{p+2}{2}B)\right),
\end{align*}
 respectively. From the way we defined $A$ and $B$, we get $A+B=S_0$, $3A+(p+2)B=S_1$ and $3A+\binom{p+2}{2}B \equiv 0~\text{mod}~p^{t+1}$ as $v(A)\geq t+1$, $v(B)\geq t+1$, by parts \textit{(1)} and \textit{(2)} of Proposition \ref{proposition 0.4.}. Hence using Proposition \ref{proposition 0.4.} part \textit{(3)} we get that the coefficients of $[g^0_{1,[\lambda]}, X^r]$, $[g^0_{1,[\lambda]}, X^{r-1}Y]$ and $[g^0_{1,[\lambda]}, X^{r-2}Y^2]$ in $T^+f_0$ die mod $p$.
 
For $\lambda\neq 0$, $i\geq 4$, the coefficient of $[g^0_{1,[\lambda]}, X^{r-i}Y^i]$ in $T^+f_0$ is given by\begin{align*}
(-[\lambda])^{3-i}\frac{p^i(p-1)}{pa_p(3-r)}\left(\sum_{\substack{{0<j<r}\\{j\equiv 3~\text{mod}~(p-1)}}}\binom{j}{i}\binom{r}{j}- B\binom{p+2}{i}\right),
\end{align*} which dies mod $p$ by Proposition \ref{proposition 0.4.} part \textit{(5)} and the fact that $v(pa_p(3-r)) = t+\frac{5}{2} $ and $v(B)\geq t+1$.

For $j > 3$, the coefficient of $[g^0_{1,0}, X^{r-j}Y^j]$ dies mod $p$ again by using Lemma \ref{lemma 0.1.}.

So finally using part \textit{(4)} of Proposition \ref{proposition 0.4.}, we get
\begin{align}
T^+f_0&= \frac{p-1}{pa_p(3-r)}\left(\sum_{\lambda\in\mathbb{F}_p^\times}\left[g^0_{1,[\lambda]},\frac{p^3\binom{r}{3}}{1-p}X^{r-3}Y^3\right]+\left[g^0_{1,0},p^3\binom{r}{3}X^{r-3}Y^3\right]\right) +O(p). \label{F_0 T plus f_0}
\end{align}
Now we do a small computation which will be used in computing $(T-a_p)f_\infty$ just below:
\begin{align}
\frac{1}{p(3-r)}&\left(\sum_{\lambda\in\mathbb{F}_p^\times}[g^0_{1,[\lambda]}\alpha, Y^r]+(1-p)[g^0_{1,0}\alpha, Y^r]\right) \nonumber \\ 
&= \frac{1}{p(3-r)}\left(\sum_{\lambda\in\mathbb{F}_p^\times}\left[p\begin{pmatrix}
1 & [\lambda]\\
0 & 1
\end{pmatrix}, Y^r\right]+(1-p)[1,Y^r]\right) \nonumber \\ 
& = \frac{1}{p(3-r)}\left(\sum_{\lambda\in\mathbb{F}_p^\times}[1, ([\lambda]X + Y)^r + (1-p)[1, Y^r]\right) \nonumber \\
&= \frac{p-1}{p(3-r)}\left[1,\sum_{\substack{{0<j<r}\\{j\equiv 3~\text{mod}~(p-1)}}}\binom{r}{j}X^{r-j}Y^j\right], \label{F_0 T minus f_1}
\end{align}
using the fact \eqref{sum of roots of 1}. %that $\sum\limits_{\lambda\in\mathbb{F}_p^\times}[\lambda]^k = 0$ if $k\nmid p-1$ or else is $p-1$. 

Now from the $f_\infty$ part of $f$, we get
\begin{align*}
(T-a_p)f_\infty &= \frac{1}{p(3-r)}\left(\sum_{\lambda\in\mathbb{F}_p^\times}[g^0_{1,[\lambda]}\alpha, Y^r]+a_p[g^0_{1,[\lambda]}, X^{r-3}Y^3]\right.\\
&~~~+(1-p)([g^0_{1,0}\alpha, Y^r] +a_p[g^0_{1,0}, X^{r-3}Y^3])\bigg)+ h + O(p),
\end{align*}
by Lemma \ref{chi}, where $h$ is an integral linear combination of functions of the form $[g, X^r]$ and $[g, X^{r-1}Y]$, for some $g\in G$. By equation \eqref{F_0 T minus f_1}, we get
 \begin{align}
(T-a_p)f_\infty &= \frac{p-1}{p(3-r)}\left[1,\sum_{\substack{{0<j<r}\\{j\equiv 3~\text{mod}~(p-1)}}}\binom{r}{j}X^{r-j}Y^j\right] + \frac{a_p}{p(3-r)}\sum_{\lambda\in\mathbb{F}_p^\times}[g^0_{1,[\lambda]}, X^{r-3}Y^3] \nonumber \\
&~~~+ \frac{a_p(1-p)}{p(3-r)}[g^0_{1,0}, X^{r-3}Y^3] + h +O(p). \label{(T-a_p)f_infty}
\end{align}
Using equations \eqref{F_0 T minus f_0}, \eqref{F_0 -a_pf_0}, \eqref{F_0 T plus f_0} and \eqref{(T-a_p)f_infty}, we get
\begin{align*}
(T-a_p)(f_0+f_\infty) &= \left(\frac{a_p}{p(3-r)}+\frac{-p^2\binom{r}{3}}{(3-r)a_p}\right)\sum_{\lambda\in\mathbb{F}_p^\times}[g^0_{1,[\lambda]}, X^{r-3}Y^3]\\
&~~~+ \left(\frac{a_p(1-p)}{p(3-r)}+\frac{(p-1)p^2\binom{r}{3}}{a_p(3-r)}\right)[g^0_{1,0}, X^{r-3}Y^3]\\
&~~~+ \frac{p-1}{p(3-r)}[1,(AX^{r-3}Y^3 + BX^{r-p-2}Y^{p+2}] + h + O(p).
\end{align*}
By part (i) of Lemma~\ref{generator}, %Lemma 3.4 of \cite{BGR18}, 
we get that $X^r$, $X^{r-1}Y$ map to $0$ inside $J_1 = V_{p-4} \otimes D^3$ and $X^{r-3}Y^3$ maps to $X^{p-4}\in J_1$. Since $\theta = X^pY- XY^p$ divides $X^{r-3}Y^3 - X^{r-p-2}Y^{p+2}$, we see $X^{r-p-2}Y^{p+2}$ also maps to $X^{p-4}$ in $J_1$. Thus the image of $(T-a_p)f$ in ind$_{KZ}^GJ_1$ is 
\begin{align*}
\overline{(T-a_p)f}&= \overline{\frac{a_p^2-\binom{r}{3}p^3}{pa_p(3-r)}}\sum_{\lambda\in\mathbb{F}_p}[g^0_{1,[\lambda]}, X^{p-4}] + \overline{\frac{(A+B)(p-1)}{p(3-r)}}[1,X^{p-4}].
\end{align*}
By Proposition \ref{proposition 0.4.}, part \textit{(1)}, 
\begin{align*}
\frac{(A+B)(p-1)}{p(3-r)} &= \frac{S_0(p-1)}{p(3-r)}\\
&\equiv -\frac{1}{6}\left[6p+5 -\frac{3}{p}\binom{2p+1}{p-1}\right] +\frac{3-r}{6(p-1)}\left[-3p -3 +\frac{3}{p}\binom{2p+1}{p-1}\right]~\text{mod}~p\\  
&\equiv -\left(\frac{5}{6}-\frac{1}{2}\right)\equiv -\frac{1}{3}~\text{mod}~p,
\end{align*}
 since $\frac{1}{p}\binom{2p+1}{p-1}\equiv 1~\text{mod}~p$. Therefore the image of $(T-a_p)f$ in ind$_{KZ}^GJ_1$ is 
\begin{align}
\frac{1}{3}(\tilde{\lambda} T-1)[1,X^{p-4}], \label{F_0 (T-a_p)f with tilde(lambda)}
\end{align}
where $\tilde{\lambda} = \overline{\frac{3}{3-r}\left(\frac{a_p^2-\binom{r}{3}p^3}{pa_p}\right)}$.
Recall $\lambda =\overline{\frac{3}{3-r}\left(\frac{a_p^2-(r-2)\binom{r-1}{2}p^3}{pa_p}\right)}$. Since
$$\frac{3}{3-r}\left(\frac{a_p^2 - (r-2)\binom{r-1}{2}p^3}{pa_p}\right) = \frac{3}{3-r}\left(\frac{a_p^2-\binom{r}{3}p^3}{pa_p}\right) + \frac{3}{3-r}\left(\frac{\binom{r}{3}p^3 - (r-2)\binom{r-1}{2}p^3}{pa_p}\right),$$ 
and as $v\left(\frac{3}{3-r}\left(\frac{\binom{r}{3}p^3 - (r-2)\binom{r-1}{2}p^3}{pa_p}\right)\right) \geq \frac{1}{2}$, we get that 
\begin{align}
\lambda = \tilde{\lambda}. \label{lambda = tilde(lambda)}
\end{align}
 
If $\tilde{\tau} > t$, then $\tilde{\lambda}= 0$, so we have that $F_1 = 0$ because $\overline{(T-a_p)f} = \frac{-1}{3}[1, X^{p-4}]$ is zero in $\bar{\Theta}_{k,a_p}$ and  $[1,X^{p-4}]$ generates ind$_{KZ}^GJ_1$. 

If $\tilde{\tau} = t$, then by \eqref{F_0 (T-a_p)f with tilde(lambda)}, $F_1$ is a quotient of $\pi(p-4,\tilde{\lambda}^{-1},\omega^3)$ which is the same as saying that $F_1$ is a quotient of $\pi(p-4,\lambda^{-1},\omega^3)$, by \eqref{lambda = tilde(lambda)}.
\end{proof}

\section{Behaviour of $F_2$}
\label{sectionF2}

Next we study the structure of the sub $F_2$ of $\bar\Theta_{k,a_p}$. Recall  $\ind_{KZ}^G J_2 \twoheadrightarrow F_2$, where
$J_2 = V_1 \otimes D$.

\begin{prop}{\label{F_1 when tau <= t}}
Let $ p>3, r\geq 2p+1$, $r = 3 + n(p-1)p^t$ with $t =v(r-3)$ and $v(a_p) = \frac{3}{2}$. Let $c=\frac{a_p^2-(r-2)\binom{r-1}{2}p^3}{pa_p}$ and $\tau = v(c)$.
\begin{enumerate}[label=(\roman{*})]
\item If $\tau < t$, then $F_2 =0.$
\item If $\tau = t$, then $F_2$ is a quotient of $\pi(1,\lambda,\omega)$, where $\lambda = \overline{\frac{3}{3-r}\left(\frac{a_p^2 - (r-2)\binom{r-1}{2}p^3}{pa_p}\right)}\in\mathbb{F}_p^\times.$
\end{enumerate}
\end{prop}
\begin{proof} The proof is a variant of the proof of Proposition 6.7 in \cite{BGR18}.

As in the proof of Proposition \ref{F_0 when tau >= t}, we claim that it is sufficient to prove the above proposition with $\tau$ replaced by $\tilde{\tau}$ in $(i)$ and $(ii)$, where $\tilde{\tau} = v(\tilde{c})$ and $\tilde{c} = \frac{a_p^2-\binom{r}{3}p^3}{pa_p}$.  Indeed, by Lemma \ref{Lemma 5.1}, we see that if $\tilde{\tau} < t$, then $\tau = \tilde{\tau}$, so $\tau < t$. Similarly if $\tilde{\tau} = t$, then $\tau = \tilde{\tau}$, so $\tau = t$.

We first define `building block' functions as follows.
  
Let $A=[1, X^{r-1}Y]$, $B=\left[1,\sum\limits_{\substack{{1<j<r-2}\\{j\equiv 1~\text{mod}~(p-1)}}}\binom{r-2}{j}X^{r-j}Y^j\right]$ and $C=[1, X^{r-p}Y^p]$.

Also for $g\in G$, let 
\begin{align*}
\Phi_g=\left[g, \sum_{\substack{{1<j\leq r-2}\\{j\equiv 1~\text{mod}~(p-1)}}}\binom{r}{j}X^{r-j}Y^j-S_0X^{r-p}Y^p\right] ~\text{and}~\Psi_g=\sum_{\mu\in\mathbb{F}_p^\times}[\mu]^{-2}\chi_{gg^0_{1,[\mu]}},
\end{align*} 
where $S_0 = \sum\limits_{\substack{{1<j\leq r-2}\\{j\equiv 1~\text{mod}~(p-1)}}}\binom{r}{j} $ and $\chi$ is as in Lemma \ref{chi}.

Now consider the function $f=f_0+f_1+f_\infty,$ with 
\begin{align*}
f_0 &= \frac{1-p}{p\tilde{c}}\left(A+\frac{\binom{r}{2}p^3}{3a_p^2}B\right) + \frac{p-1}{a_p}C,\\
f_1&= \frac{-1}{3a_p\tilde{c}}\left(\sum_{\lambda\in\mathbb{F}_p^\times}\Phi_{g^0_{1,[\lambda]}} +(1-p)\Phi_{g^0_{1,0}}\right),\\
f_\infty &= \frac{1}{3\tilde{c}}\left(\frac{1}{1-p}\sum_{\lambda\in\mathbb{F}_p^\times}\Psi_{g^0_{1,[\lambda]}} + \Psi_{g^0_{1,0}}\right).
\end{align*}
We now compute $(T-a_p)f$ in several steps.\\
In radius $-1$ we have 
\begin{align*}
T^-A &= [\alpha, (pX)^{r-1}Y] = O(p^{t+5}),\\
T^-B&= \left[\alpha,\sum\limits_{\substack{{1<j<r-2}\\{j\equiv 1~\text{mod}~(p-1)}}}\binom{r-2}{j}(pX)^{r-j}Y^j \right] = O(p^{t+4}),
\end{align*}
since $p^{r-j}\binom{r-2}{j} = p^2\left(p^{r-j-2}\binom{r-2}{r-2-j}\right)$, and $ r-j-2 \geq p-1>3$ for all $j$ as above, so by Lemma \ref{lemma 0.3.}, we get $p^{r-j-2}\binom{r-2}{r-j-2} \equiv 0~\text{mod}~p^{t+2}$. Also
\begin{align*}
&\mkern-18mu \mkern-18mu \mkern-18mu \mkern-18mu \mkern-18mu \mkern-18mu \mkern-18mu \mkern-18mu \mkern-18mu  T^-C= [\alpha,(pX)^{r-p}Y^p]= O(p^{r-p}).
\end{align*}
As $v(\tilde{c})\leq t$, we get 
\begin{align}
T^-f_0 = O(p^3), \label{F_1, tau <= t, T minus f_0}
\end{align}
 hence it dies mod $p$.

In radius $0$ we compute $-a_pf_0 + T^-f_1$. Now
\begin{align*}
T^-\Phi_g &= \left[g\alpha,\sum_{\substack{{1<j\leq r-2}\\{j\equiv 1~\text{mod}~(p-1)}}}\binom{r}{j}(pX)^{r-j}Y^j- S_0(pX)^{r-p}Y^p \right].
\end{align*}
For all $j$ with $1<j<r-2$ and $j\equiv 1~\text{mod}~(p-1)$, we have $r-j >p>4$, so by Lemma~ \ref{lemma 0.1.} we get that $p^{r-j}\binom{r}{j}\equiv 0$ mod $p^{t+4}$. Also by Proposition \ref{proposition 0.6}, part \textit{(1)}, we have $v(S_0)\geq t$. So we  get
\begin{align}
T^-\Phi_g &= \left[g\alpha, \binom{r}{r-2}p^2X^2Y^{r-2}\right] + O(p^{t+4}). \label{F_1 T minus phi_g}
\end{align}
For $\lambda\neq 0$, by substituting $g^0_{1,[\lambda]}$ for $g$ in \eqref{F_1 T minus phi_g}, we get
\begin{align}
T^-\Phi_{g^0_{1,[\lambda]}} &= \left[g^0_{1,[\lambda]}\alpha, \binom{r}{r-2}p^2X^2Y^{r-2}\right] + O(p^{t+4}) \nonumber \\
&= \left[1, \binom{r}{2}p^2X^2([\lambda]X+Y)^{r-2}\right] + O(p^{t+4}) \nonumber\\
&= \left[1, \binom{r}{2}p^2\sum_{i=0}^{r-2}\binom{r-2}{i}[\lambda]^{r-2-i}X^{r-i}Y^i\right] + O(p^{t+4}). \label{F_1 T minus phi_g_{0, lambda}}
\end{align}
Substituting $g^0_{1,0}$ for $g$ in \eqref{F_1 T minus phi_g}, we get 
\begin{align}
T^-\Phi_{g^0_{1,0}} = \left[1, p^2\binom{r}{2}X^{r-2}Y^2\right] + O(p^{t+4}). \label{F_1 T minus phi_g_{0,1}}
\end{align}
By equation \eqref{F_1 T minus phi_g_{0, lambda}} and \eqref{F_1 T minus phi_g_{0,1}}, we get
\begin{align}
 T^-f_1 &= \frac{-p^2\binom{r}{2}}{3a_p\tilde{c}}\left(\sum_{\lambda\in\mathbb{F}_p^\times}\left[1, \sum_{i=0}^{r-2}\binom{r-2}{i}[\lambda]^{r-2-i}X^{r-i}Y^i\right]+(1-p)[1,X^2Y^{r-2}]\right) + O(p^{\frac{5}{2}}) \nonumber \\
&= \frac{p^2(1-p)\binom{r}{2}}{3a_p\tilde{c}}\left[1, \sum_{\substack{{1\leq j< r-2}\\{j\equiv 1~\text{mod}~(p-1)}}}\binom{r-2}{j}X^{r-j}Y^j\right] + O(p^{\frac{5}{2}}). \label{F_1 T minus f_1}
\end{align}
We also have \begin{align}
-a_pf_0&= \frac{-a_p(1-p)}{p\tilde{c}}[1,X^{r-1}Y]-\frac{p^2(1-p)\binom{r}{2}}{3a_p\tilde{c}}\left[1,\sum_{\substack{{1< j< r-2}\\{j\equiv 1~\text{mod}~(p-1)}}}\binom{r-2}{j}X^{r-j}Y^j\right] \nonumber \\
&~~~- (p-1)[1,X^{r-p}Y^p]. \label{F_1 -a_pf_0}
\end{align}
So finally by equations \eqref{F_1 T minus f_1} and \eqref{F_1 T minus f_1}, in radius $0$ we get
\begin{align}
-a_pf_0+T^-f_1 = -[1, X^{r-1}Y-X^{r-p}Y^p] + O(p), \label{F_1, tau <= t, radius 0}
\end{align}
because the coefficient of $[1, X^{r-1}Y]$ term from $-a_pf_0$ and $T^-f_1$ is 
\begin{align*}
-\frac{a_p(1-p)}{p\tilde{c}} +\frac{p^2(1-p)\binom{r}{2}(r-2)}{3a_p\tilde{c}}
 = \frac{p-1}{\tilde{c}}\left(\frac{a_p^2-\binom{r}{3}p^3}{pa_p}\right) \equiv -1 ~\text{mod}~p. 
\end{align*}

In radius $1$, we compute $T^+f_0-a_pf_1+h_{\infty,1}$, where $h_{\infty,1}$ denotes the part of $(T-a_p)f_\infty$ that lives in radius $1$. We have
\begin{align}
T^+A &= \sum_{\lambda\in\mathbb{F}_p}[g^0_{1,[\lambda]}, X^{r-1}(-[\lambda]X+pY)] \label{F_1 T plus A}.
\end{align}
Also
\begin{align*}
T^+B &=\sum_{\lambda\in\mathbb{F}_p}\left[g^0_{1,[\lambda]}, \sum_{\substack{{1<j<r-2}\\{j\equiv 1 ~\text{mod}~(p-1)}}}\binom{r-2}{j}X^{r-j}(-[\lambda]X+pY)^j\right]\\
&= \sum_{\lambda\in\mathbb{F}_p^\times}\left[g^0_{1,[\lambda]}, \sum_{\substack{{1<j<r-2}\\{j\equiv 1 ~\text{mod}~(p-1)}}}\binom{r-2}{j}X^{r-j}\sum_{i=0}^j\binom{j}{i}(-[\lambda]X)^{j-i}(pY)^i\right] \\
&~~~+ \left[g^0_{1,0}, \sum_{\substack{{1<j<r-2}\\{j\equiv 1 ~\text{mod}~(p-1)}}}p^j\binom{r-2}{j}X^{r-j}Y^j\right]\\
&= \sum_{\lambda\in\mathbb{F}_p^\times}\left[g^0_{1,[\lambda]}, \sum_{i=0}^rp^i(-[\lambda])^{1-i}\sum_{\substack{{1<j<r-2}\\{j\equiv 1 ~\text{mod}~(p-1)}}}\binom{j}{i}\binom{r-2}{j}X^{r-j}Y^j\right] + O(p^{t+2}),
\end{align*}
because by Lemma \ref{lemma 0.3.}, the coefficient of $[g^0_{1,0}, X^{r-j}Y^j]$ is $p^j\binom{r-2}{j} \equiv 0$ mod $p^{t+2}$, for all $1<j<r-2$ and $j\equiv 1$ mod $(p-1)$.

By Proposition \ref{propositon 0.5} parts \textit{(1)}, \textit{(2)}, \textit{(3)} and Lemma \ref{lemma 0.3.}, the coefficients of $[g^0_{1,[\lambda]}, X^r]$, $[g^0_{1,[\lambda]}, X^{r-1}Y]$, $[g^0_{1,[\lambda]}, X^{r-i}Y^i]$ for $\lambda\neq 0$ and $i\geq 2$ in $T^+B$ are 
 \begin{align*}
 -[\lambda]\sum_{\substack{{1<j<r-2}\\{j\equiv 1~\text{mod}~(p-1)}}}\binom{r-2}{j}&\equiv -[\lambda](-(r-2) -1 +1 +np^{t+1})\\
 &\equiv -[\lambda](2-r+np^{t+1})~\text{mod}~p^{t+2},\\
  p\sum_{\substack{{1<j<r-2}\\{j\equiv 1 ~\text{mod}~(p-1)}}}j\binom{r-2}{j} &\equiv p\left(-2(r-2) + \frac{(r-2)(p-2)}{p-1}\right)\\
 & \equiv \frac{-p^2(r-2)}{p-1}~\text{mod}~p^{t+2},\\
  (-[\lambda])^{1-i}p^i\sum_{\substack{{1<j<r-2}\\{j\equiv 1 ~\text{mod}~(p-1)}}}\binom{j}{i}\binom{r-2}{j} &\equiv 0 ~\text{mod}~p^{t+2},
\end{align*}
respectively. Therefore we can write 
\begin{align}
T^+B &= \sum_{\lambda\in\mathbb{F}_p^\times}\left[g^0_{1,[\lambda]},-[\lambda](2-r+np^{t+1})X^r+\frac{p^2(r-2)}{1-p}X^{r-1}Y\right] + O(p^{t+2}). \label{F_1 T plus B}
\end{align}
Finally
\begin{align}
T^+C &= \sum_{\lambda\in\mathbb{F}_p}[g^0_{1,[\lambda]}, X^{r-p}(-[\lambda]X+pY)^p]
= \sum_{\lambda\in\mathbb{F}_p}[g^0_{1,[\lambda]}, -[\lambda]X^r] + O(p^2).\label{F_1 T plus C}
\end{align}
So from equations \eqref{F_1 T plus A}, \eqref{F_1 T plus B} and \eqref{F_1 T plus C}, we get
\begin{align*}
T^+f_0 &= \frac{1-p}{p\tilde{c}}\sum_{\lambda\in\mathbb{F}_p}[g^0_{1,[\lambda]}, X^{r-1}(-[\lambda]X + pY)]\\
&~~~ + \left(\frac{1-p}{p\tilde{c}}\right)\left(\frac{\binom{r}{2}p^3}{3a_p^2}\right)\sum_{\lambda\in\mathbb{F}_p^\times}\left[g^0_{1,[\lambda]}, -[\lambda](2-r + np^{t+1})X^r + \frac{p^2(r-2)}{1-p}X^{r-1}Y\right]\\
&~~~+ \frac{p-1}{a_p}\sum_{\lambda\in\mathbb{F}_p^\times}[g^0_{1,[\lambda]}, -[\lambda]X^r] + O(\sqrt{p}). 
	\end{align*}
Now, the coefficient of $[g^0_{1,0}, X^r]$ in $T^+f_0$ is $0$.

\noindent Coefficient of $[g^0_{1,[\lambda]}, X^r]$ for $\lambda\neq 0$ in $T^+f_0$ is 
\begin{align*}
&-[\lambda]\left(\frac{1-p}{p\tilde{c}}\right) - [\lambda]\left(\frac{1-p}{p\tilde{c}}\right)\left(\frac{\binom{r}{2}p^3}{3a_p^2}\right)(2-r+np^{t+1}) -[\lambda]\left(\frac{p-1}{a_p}\right)\\
&~~~= (1-p)[\lambda]\left(\left(\frac{-a_p^2+\binom{r}{3}p^3}{pa_p^2\tilde{c}}\right)+ \frac{1}{a_p}\right) - \frac{n[\lambda](1-p)\binom{r}{2}p^{t+3}}{3a_p^2\tilde{c}}\\
&~~~ = - \frac{n[\lambda](1-p)\binom{r}{2}p^{t+3}}{3a_p^2\tilde{c}},
\end{align*}
which is integral because $v(\tilde{c}) \leq t$.

\noindent Coefficient of $[g^0_{1,0}, X^{r-1}Y]$ in $T^+f_0$ is $\left(\frac{1-p}{p\tilde{c}}\right)p = \frac{1-p}{\tilde{c}}$.

\noindent Coefficient of $[g^0_{1,[\lambda]}, X^{r-1}Y]$ for $\lambda\neq 0 $ in $T^+f_0$ is 
\begin{align*}
& p\left(\frac{1-p}{p\tilde{c}}\right) + \left(\frac{1-p}{p\tilde{c}}\right)\left(\frac{\binom{r}{2}p^3}{3a_p^2}\right)\left(\frac{p^2(r-2)}{1-p}\right) \\
& = \frac{1-p}{\tilde{c}} +\frac{p}{\tilde{c}}\left(\frac{\binom{r}{3}p^3}{a_p^2}\right) 
= \frac{1}{\tilde{c}} + \frac{p}{\tilde{c}}\left(\frac{pa_p(-\tilde{c})}{a_p^2}\right)  = \frac{1}{\tilde{c}} -\frac{p^2}{a_p}.
\end{align*}
So we get \begin{align*}
T^+f_0 = \left(\frac{1-p}{\tilde{c}}\right)[g^0_{1,0}, X^{r-1}Y] + \frac{1}{\tilde{c}}\sum_{\lambda\in\mathbb{F}_p^\times}[g^0_{1,[\lambda]}, X^{r-1}Y] + h + O(\sqrt{p}), 
\end{align*}
where $h$ is an integral linear combination of terms of the form $[g, X^r]$, for some $g\in G$.
We can rewrite this as 
\begin{align}
T^+f_0 &= \frac{1-p}{3\tilde{c}}\left[g^0_{1,0}, \binom{r}{1}X^{r-1}Y\right] + \frac{1}{3\tilde{c}}\sum_{\lambda\in\mathbb{F}_p^\times}\left[g^0_{1,[\lambda]},\binom{r}{1}X^{r-1}Y\right] + \frac{3-r}{3\tilde{c}}\sum_{\lambda\in\mathbb{F}_p}[g^0_{1,[\lambda]}, X^{r-1}Y]  \nonumber\\ &~~~ + h +O(\sqrt{p}). \label{F_1 T plus f_0}
\end{align}
We also have \begin{align*}
-a_pf_1 &= \frac{1-p}{3\tilde{c}}\left[g^0_{1,0}, \sum_{\substack{{1<j \leq r-2}\\{j\equiv 1~\text{mod}~(p-1)}}}\binom{r}{j}X^{r-j}Y^j-S_0X^{r-p}Y^p\right]\\
&~~~+ \frac{1}{3\tilde{c}}\sum_{\lambda\in\mathbb{F}_p^\times}\left[g^0_{1,[\lambda]}, \sum_{\substack{{1<j\leq r-2}\\{j\equiv 1 ~\text{mod}~(p-1)}}}\binom{r}{j}X^{r-j}Y^j -S_0X^{r-p}Y^p\right].
\end{align*}
By Proposition \ref{proposition 0.6} part \textit{(1)}, we have $S_0 \equiv 3-r$ mod $p^{t+1}$. As $v(\tilde{c})\leq t$, we can write
\begin{align}
-a_pf_1&= \frac{1-p}{3\tilde{c}}\left[g^0_{1,0}, \sum_{\substack{{1<j \leq r-2}\\{j\equiv 1~\text{mod}~(p-1)}}}\binom{r}{j}X^{r-j}Y^j\right] +\frac{1}{3\tilde{c}}\sum_{\lambda\in\mathbb{F}_p^\times}\left[g^0_{1,[\lambda]}, \sum_{\substack{{1<j\leq r-2}\\{j\equiv 1 ~\text{mod}~(p-1)}}}\binom{r}{j}X^{r-j}Y^j \right]\nonumber \\
&~~~ -\frac{3-r}{3\tilde{c}}\sum_{\lambda\in\mathbb{F}_p}[g^0_{1,[\lambda]}, X^{r-p}Y^p] + O(p). \label{F_1 tau <= t -a_pf_1}
\end{align}
Now combining the radius $1$ terms from equations \eqref{F_1 T plus f_0} and \eqref{F_1 tau <= t -a_pf_1}, we get
\begin{align}
T^+f_0 -a_pf_1 &= \frac{1-p}{3\tilde{c}}\left[g^0_{1,0}, \sum_{\substack{{1\leq j\leq r-2}\\{j\equiv 1~\text{mod}~(p-1)}}}\binom{r}{j}X^{r-j}Y^j\right] + \frac{1}{3\tilde{c}}\sum_{\lambda\in\mathbb{F}_p^\times}\left[g^0_{1,[\lambda]}, \sum_{\substack{{1\leq j\leq r-2}\\{j\equiv 1~\text{mod}~(p-1)}}} \binom{r}{j}X^{r-j}Y^j\right]\nonumber \\
&~~~+\frac{3-r}{3\tilde{c}}\sum_{\lambda\in\mathbb{F}_p}[g^0_{1,[\lambda]}, X^{r-1}Y- X^{r-p}Y^p]+ h + O(\sqrt{p}). \label{F_1 T plus f_0 + a_pf_1}
\end{align}
We now compute $(T-a_p)\Psi_g.$ Using Lemma \ref{chi}, as $\tilde{\tau} \leq t$, we have
\begin{align*}
(T-a_p)\Psi_g &= \sum_{\mu\in\mathbb{F}_p^\times}[\mu]^{-2}\left([gg^0_{1,[\mu]}\alpha, Y^r]+a_p[gg^0_{1,[\mu]}, X^{r-3}Y^3]+p^{t+1}h_\chi+ O(p^{\tilde{\tau} +2})\right)\\
&= \sum_{\mu\in\mathbb{F}_p^\times}[\mu]^{-2}[g, ([\mu]X+ Y)^r] + a_p\sum_{\mu\in\mathbb{F}_p^\times}[\mu]^{-2}[gg^0_{1,[\mu]}, X^{r-3}Y^3] + O(p^{\tilde{\tau} +1})\\
&= (p-1)\left[g, \sum_{\substack{{1\leq j\leq r-2}\\{j\equiv 1~\text{mod}~(p-1)}}}\binom{r}{j}X^{r-j}Y^j\right] + a_p\sum_{\mu\in\mathbb{F}_p^\times}[\mu]^{-2}[gg^0_{1,[\mu]}, X^{r-3}Y^3] + O(p^{\tilde{\tau} +1}),
\end{align*}
by \eqref{sum of roots of 1}.
Using the above computation, we have
\begin{align*}
&(T-a_p)f_\infty = \frac{1}{3\tilde{c}(1-p)}\sum_{\lambda\in\mathbb{F}_p^\times}\Bigg((p-1)\Bigg[g^0_{1,[\lambda]}, \sum_{\substack{{1\leq j\leq r-2}\\{j\equiv 1~\text{mod}~(p-1)}}}\binom{r}{j}X^{r-j}Y^j\Bigg]\\
&+  a_p\sum_{\lambda\in\mathbb{F}_p^\times}[\mu]^{-2}[g^0_{1,[\lambda]}g^0_{1,[\mu]}, X^{r-3}Y^3]  \Bigg) + \frac{1}{3\tilde{c}}\Bigg((p-1)\Bigg[g^0_{1,0}, \sum_{\substack{{1\leq j\leq r-2}\\{j\equiv 1~\text{mod}~(p-1)}}}\binom{r}{j}X^{r-j}Y^j\Bigg]\\
& + a_p\sum_{\mu\in\mathbb{F}_p^\times}[\mu]^{-2}[g^0_{1,0}g^0_{1,[\mu]}, X^{r-3}Y^3]  \Bigg) + O(p).
\end{align*}
Rearranging the terms, we finally have
\begin{align}
(T- &a_p)f_\infty \nonumber \\
&= \frac{-1}{3\tilde{c}}\left((1-p)\left[g^0_{1,0},\sum_{\substack{{1\leq j\leq r-2}\\{j\equiv 1~\text{mod}~(p-1)}}}\binom{r}{j}X^{r-j}Y^j\right] + \sum_{\lambda\in\mathbb{F}_p^\times}\left[g^0_{1,[\lambda]}, \sum_{\substack{{1\leq j\leq r-2}\\{j\equiv 1~\text{mod}~(p-1)}}}\binom{r}{j}X^{r-j}Y^j\right]\right) \nonumber \\
& ~~~+ \frac{a_p}{3\tilde{c}}\left(\sum_{\mu\in\mathbb{F}_p^\times}[\mu]^{-2}[g^0_{1,0}g^0_{1,[\mu]}, X^{r-3}Y^3]+ \frac{1}{1-p}\sum_{\lambda\in\mathbb{F}_p^\times}\sum_{\mu\in\mathbb{F}_p^\times}[\mu]^{-2}[g^0_{1,[\lambda]}g^0_{1,[\mu]}, X^{r-3}Y^3]\right) + O(p). \label{F_1 (T-a_p)f_infty}
\end{align}
Combining everything together, from equations \eqref{F_1 T plus f_0 + a_pf_1} and \eqref{F_1 (T-a_p)f_infty}, in radius $1$ we have
\begin{align}
T^+f_0 - a_pf_1 + h_{\infty,1} = \frac{3-r}{3\tilde{c}}\sum_{\lambda\in\mathbb{F}_p}[g^0_{1,[\lambda]}, X^{r-1}Y - X^{r-p}Y^p] + h + O(\sqrt{p}). \label{F_1 tau <= t radius 1}
\end{align}

Finally, in radius $2$ we compute $T^+f_1 + h_{\infty, 2}$, where $h_{\infty, 2}$ is the radius $2$ part of $(T-a_p)f_\infty$. We have
\begin{align}
T^+\Phi_g &= \sum_{\mu\in\mathbb{F}_p}\left[gg^0_{1,[\mu]}, \sum_{\substack{{1<j\leq r-2}\\{j\equiv 1~\text{mod}~p-1}}}\binom{r}{j}X^{r-j}(-[\mu]X+ pY)^j - S_0X^{r-p}(-[\mu]X+ pY)^p\right] \nonumber \\
&= \sum_{\mu\in\mathbb{F}_p^\times}\left[gg^0_{1,[\mu]}, \sum_{i= 0}^r p^i(-[\mu])^{1-i}\sum_{\substack{{1< j\leq r-2}\\{j\equiv 1~\text{mod}~(p-1)}}}\binom{j}{i}\binom{r}{j}X^{r-j}Y^j + [\mu]S_0X^r\right]\nonumber \\
&~~~+ \left[gg^0_{1,0}, \sum_{\substack{{1<j\leq r-2}\\{j\equiv 1~\text{mod}~(p-1)}}}p^j\binom{r}{j}X^{r-j}Y^j\right] + O(p^{t+2}), \label{F_1 T plus psi_g}
\end{align}
since $v(S_0) \geq t$, by Proposition \ref{proposition 0.6} part \textit{(1)}.

By using Lemma \ref{lemma 0.1.} for the $\mu = 0$ terms in \eqref{F_1 T plus psi_g}, we get 
\begin{align*}
T^+\Phi_g = \sum_{\mu\in\mathbb{F}_p^\times}\left[gg^0_{1,[\mu]}, \sum_{i= 0}^r p^i(-[\mu])^{1-i}\sum_{\substack{{1< j\leq r-2}\\{j\equiv 1~\text{mod}~(p-1)}}}\binom{j}{i}\binom{r}{j}X^{r-j}Y^j + [\mu]S_0X^r\right] + O(p^{t+2}).
\end{align*}
By the definition of $S_0$, the coefficient of $[gg^0_{1,[\mu]}, X^r]$ for $\mu \neq 0$ in $T^+\Phi_g$ vanishes since
\begin{align*}
 -[\mu]\sum_{\substack{{1<j\leq r-2}\\{j\equiv 1~\text{mod}~(p-1)}}}\binom{r}{j} + [\mu]S_0 = 0.
\end{align*}

By Proposition \ref{proposition 0.6}, parts \textit{(2)}, \textit{(3)}, \textit{(4)} and \textit{(5)}, the coefficients of $[gg^0_{1[\mu]}, X^{r-1}Y]$, $[gg^0_{1,[\mu]}, X^{r-2}Y^2]$, $[g^0_{1,[\mu]}, X^{r-3}Y^3]$, $[g^0_{1,[\mu]}, X^{r-i}Y^i]$ (for $i\geq 4$) in $T^+\Phi_g$, for $\mu\neq 0$, are
\begin{align*}
p \sum_{\substack{{1<j\leq r-2}\\{j\equiv 1~\text{mod}~(p-1)}}}j\binom{r}{j} &\equiv 0~\text{mod}~p^{t+2},\\
 -[\mu]^{-1}p^2\sum_{\substack{{1<j\leq r-2}\\{j\equiv 1~\text{mod}~(p-1)}}}\binom{j}{2}\binom{r}{j}& \equiv 0 ~\text{mod}~p^{t+2},\\
 [\mu]^{-2}p^3\sum_{\substack{{1<j\leq r-2}\\{j\equiv 1~\text{mod}~(p-1)}}}\binom{j}{3}\binom{r}{j}&\equiv \frac{[\mu]^{-2}\binom{r}{3}p^3}{1-p}~\text{mod}~p^{t+2},\\
(-[\mu])^{1-i}p^i\sum_{\substack{{1<j\leq r-2}\\{j\equiv 1~\text{mod}~(p-1)}}}\binom{j}{i}\binom{r}{j} &\equiv 0 ~\text{mod}~p^{t+4}, \forall i \geq 4,
\end{align*}
respectively.
As $v(a_p\tilde{c})\leq t+\frac{3}{2}$, we get that 
\begin{align}
 \frac{-1}{3a_p\tilde{c}}T^+\Phi_g &= \frac{-p^3\binom{r}{3}}{3a_p\tilde{c}(1-p)}\sum_{\mu\in\mathbb{F}_p^\times}[\mu]^{-2}[gg^0_{1,[\mu]}, X^{r-3}Y^3] + O(\sqrt{p}) \nonumber\\
 &= \frac{-a_p}{3\tilde{c}(1-p)}\sum_{\mu\in\mathbb{F}_p^\times}[\mu]^{-2}[gg^0_{1,[\mu]}, X^{r-3}Y^3] + O(\sqrt{p}), \label{F_1 tau <= t T plus Phi_g }
\end{align}
since \begin{align*}
\frac{-p^3\binom{r}{3}}{3a_p\tilde{c}(1-p)} = \frac{pa_p\tilde{c}-a_p^2}{3a_p\tilde{c}(1-p)}\equiv \frac{-a_p}{3\tilde{c}(1-p)} ~\text{mod}~p.
\end{align*}
So by equation \eqref{F_1 tau <= t T plus Phi_g }, we get
 \begin{align}
T^+f_1 = \frac{-a_p}{3\tilde{c}}\left(\frac{1}{1-p}\sum_{\lambda\in\mathbb{F}_p^\times}\sum_{\mu\in\mathbb{F}_p^\times}[\mu]^{-2}[g^0_{1,[\lambda]}g^0_{1,[\mu]}, X^{r-3}Y^3]+ \sum_{\mu\in\mathbb{F}_p^\times}[\mu]^{-2}[g^0_{1,0}g^0_{1,[\mu]}, X^{r-3}Y^3]\right) + O(\sqrt{p}). \label{F_1 T plus f_1}
\end{align}
Therefore from equations \eqref{F_1 (T-a_p)f_infty} and \eqref{F_1 T plus f_1}, in radius $2$ we have \begin{align}
T^+f_1 + h_{\infty, 2} = O(\sqrt{p}). \label{F_1 tau <= t radius 2}
\end{align}

So putting everything together by equations \eqref{F_1, tau <= t, T minus f_0}, \eqref{F_1, tau <= t, radius 0}, \eqref{F_1 tau <= t radius 1} and \eqref{F_1 tau <= t radius 2}, we have 
\begin{align*}
(T-a_p)f &= \frac{3-r}{3\tilde{c}}\sum_{\lambda\in\mathbb{F}_p}[g^0_{1,[\lambda]}, X^{r-1}Y - X^{r-p}Y^p]- [1, X^{r-1}Y - X^{r-p}Y^p] + h + O(\sqrt{p})\\
&= \frac{3-r}{3\tilde{c}}\sum_{\lambda\in\mathbb{F}_p}[g^0_{1,[\lambda]}, \theta X^{r-p-1}]- [1,\theta X^{r-p-1}] + h + O(\sqrt{p}),
\end{align*}
where $h$ is an integral combination of terms of the form $[g, X^{r}],$ for some $g\in G$. Now the term $h$ dies in ind$_{KZ}^{G}Q$. Moreover, 
by part (ii) of Lemma~\ref{generator},
%by Lemma 8.5 of \cite{Bhattacharya-Ghate}, 
we see that $\theta X^{r-p-1}$ is identified with $X\in J_2 = V_1\otimes D$. So the image of $(T-a_p)f$ in ind$_{KZ}^GQ$ actually lies in ind$_{KZ}^GJ_2$ and is given by
\begin{align*}
\overline{\frac{3-r}{3\tilde{c}}}\sum_{\lambda\in\mathbb{F}_p}[g^0_{1,[\lambda]}, X] - [1, X] = \left(\frac{1}{\tilde{\lambda}}T - 1 \right)[1, X],
\end{align*} 
where $\tilde{\lambda} = \overline{\frac{3}{3-r}\left(\frac{a_p^2 - \binom{r}{3}p^3}{pa_p}\right)}$. 

If $\tilde{\tau} < t$, then $\tilde{\lambda}^{-1} =0$, so we have that $F_2 = 0$ because $\overline{(T-a_p)f} = [1, X]$ which is zero in $\bar{\Theta}_{k,a_p}$ and $[1, X]$ generates ind$_{KZ}^GJ_2$. 

If $\tilde{\tau} = t$, then $F_2$ is a quotient of $\pi(1,\tilde{\lambda},\omega)$  which is the same as saying that $F_2$ is a quotient of $\pi(1,\lambda,\omega)$, by \eqref{lambda = tilde(lambda)}.
\end{proof}

We now prove a lemma which will be used to show that $F_2 = 0 $, when $\tau > t + \frac{1}{2}$ (see Proposition~\ref{F_1 when tau > t + 0.5}).

%\vspace{5mm}

\begin{lemma}{\label{Technical lemma about F_1}}
Let $T$ be the Hecke operator and $J_2 = V_1\otimes D$. If $h = \sum\limits_{\lambda\in\mathbb{F}_p}[g^0_{1,[\lambda]}, Y] \in\emph{ind}_{KZ}^GJ_2$, then $h\not\in T(\emph{ind}_{KZ}^GJ_2)$. 
\end{lemma}

\begin{proof}
%Let $I_0 = 0$ and $I_l = \lbrace[\lambda_0] + [\lambda_1]p + \ldots + [\lambda_{l-1}]p^{l-1} : \lambda_i \in \mathbb{F}_p \rbrace$, for $l>0$. Let 
%\[ g^0_{n,\lambda} = \begin{pmatrix}
%p^n & \lambda\\
%0 & 1
%\end{pmatrix}~~~\text{and}~~~
%g^1_{m,\mu} =\begin{pmatrix}
%1 & 0\\
%p\mu & p^{m+1}
%\end{pmatrix},
%\] 
%where $n$, $m \geq 0$, $\lambda \in I_n$ and $\mu\in I_m$.

Recall that the vertices of the Bruhat-Tits tree for $G=$ GL$_2(\mathbb{Q}_p)$ are represented by the matrices $g^0_{n,\lambda}$ and $g^1_{m,\mu}$ for $n$, $m \geq 0$, $\lambda \in I_n$ and $\mu\in I_m$ (cf. Section~\ref{Hecke}).
Let $f \in \ind_{KZ}^G J_2$. If $f$ is supported mod $KZ$ on matrices of the first kind, then we say that $f$ is supported on the positive side of the tree and if $f$ is supported mod $KZ$ on matrices of the second kind, then we say that $f$ is supported on the negative side of the tree. 

Suppose there exists an $f \in$ ind$_{KZ}^G J_2$ such that $Tf =h$. Let $f = f_+ + f_-$, where $f_+$ and $f_-$ are the parts of $f$ supported mod $KZ$ on the positive and the negative side of the tree respectively.

We first show that  supp$(Tf_-)$ mod $KZ$ does not contain matrices of the form $g^0_{n,\lambda}$, for $n\geq 1$ and $\lambda \in I_n$. As $T = T^+ + T^-$, we first claim that supp$(T^+f_-)$ mod $KZ$ does not contain matrices of the form $g^0_{n,\lambda}$, for $n \geq 1$ and $\lambda \in I_n$.
Indeed, if supp$(T^+f_-)$ mod $KZ$ contains some matrix $g^0_{n, \gamma}$, %for $n\geq 1$ and $\gamma\in I_n$, 
then we have, by \eqref{T^+},
\begin{align}\begin{pmatrix}
1 & 0\\
p\mu & p^{m+1}
\end{pmatrix}
\begin{pmatrix}
p & [\lambda]\\
0 & 1
\end{pmatrix}
= \begin{pmatrix}
p^n & \gamma\\
0 & 1
\end{pmatrix}kz, \label{T plus of f minus}
\end{align}
for some $m\geq 0$, $n \geq 1$, $\mu\in I_m$, $\lambda\in\mathbb{F}_p$, $\gamma\in I_n$, $k\in K = \GL_2(\mathbb{Z}_p)$ and $z= \left(\begin{smallmatrix}
z & 0\\
0 & z
\end{smallmatrix} \right) \in Z$. Comparing the valuations of the determinants in equation \eqref{T plus of f minus}, we get $v(\det(z)) = 2v(z)= m-n+2$, so 
\begin{align}v(z) = \frac{m-n+2}{2}. \label{v(z)}
\end{align}
 Multiplying by $(g^0_{n,\gamma})^{-1}$ on both sides of equation \eqref{T plus of f minus} and further multiplying out the matrices, we get
\begin{align}
%p^{-n}\begin{pmatrix}
%1 & -\gamma\\
%0 & p^n
%\end{pmatrix}
%\begin{pmatrix}
%1& 0\\
%p\mu & p^{m+1}
%\end{pmatrix}
%\begin{pmatrix}
%p & [\lambda]\\
%0 & 1
%\end{pmatrix} &= kz \nonumber\\
%p^{-n}\begin{pmatrix}
%1-p\gamma\mu & -\gamma p^{m+1}\\
%p^{n+1}\mu & p^{m+n+1}
%\end{pmatrix}
%\begin{pmatrix}
%p & [\lambda]\\
%0 & 1
%\end{pmatrix} &= kz\nonumber\\
p^{-n}\begin{pmatrix}
p-p^2\gamma\mu & [\lambda] - p[\lambda]\gamma\mu - \gamma p^{m+1}\\
p^{n+2}\mu & p^{n+1}\mu [\lambda] + p^{m+n+1}
\end{pmatrix} &= kz. \label{kz}
\end{align}
Multiplying equation \eqref{kz} by $z^{-1}$ and comparing the valuation of the $(1,1)$-th entry, we get
\begin{align*}
v(k_{(1,1)}) = -n + 1 - \frac{m-n+2}{2} = \frac{-m-n}{2},
\end{align*}
by \eqref{v(z)}. As $m\geq 0$ and $n \geq 1$, we get that $v(k_{(1,1)}) < 0$, which is a contradiction since $k_{(1,1)}\in \mathbb{Z}_p$, proving the claim.

Next we claim that supp$(T^-f_-)$ mod $KZ$ does not contain matrices of the form $g^0_{n,\gamma}$, for $n \geq 0$ and $\gamma\in I_n$. Indeed if supp$(T^-f_-)$ mod $KZ$ does contain such a matrix, then we have, by \eqref{T^-},
\begin{align}
\begin{pmatrix}
1 & 0\\
p\mu & p^{m+1}
\end{pmatrix}
\begin{pmatrix}
1 & 0\\
0 & p
\end{pmatrix} & = \begin{pmatrix}
p^n & \gamma\\
0 & 1 
\end{pmatrix}kz, \label{T of f minus}
\end{align} 
where $m$, $n\geq 0$, $\mu \in I_m$, $\gamma\in I_n$, $k \in$ GL$_2(\mathbb{Z}_p)$ and $z\in Z$.
Comparing the valuations of the determinants in equation \eqref{T of f minus}, we get
\begin{align}
v(z) = \frac{m-n+2}{2}. \label{v((z))}
\end{align}
Multiplying both sides of equation \eqref{T of f minus} by $(g^0_{n, \gamma})^{-1}$ and further multiplying out the matrices, we get 
\begin{align}
%p^{-n}\begin{pmatrix}
%1 & -\gamma\\
%0 & p^n
%\end{pmatrix}
%\begin{pmatrix}
%1 & 0\\
%p\mu & p^{m+1}
%\end{pmatrix}
%\begin{pmatrix}
%1 & 0 \\
%0 & p
%\end{pmatrix} & = kz \nonumber\\
%p^{-n}\begin{pmatrix}
%1-p\gamma\mu & -\gamma p^{m+1}\\
%p^{n+1}\mu & p^{m+n+1}
%\end{pmatrix}
%\begin{pmatrix}
%1 & 0\\
%0 & p
%\end{pmatrix} &= kz\nonumber\\
p^{-n}\begin{pmatrix}
1-p\gamma\mu & -\gamma p^{m+2}\\
p^{n+1}\mu & p^{m+n+2}  \label{(kz)}
\end{pmatrix} & = kz.
\end{align}
Multiplying equation \eqref{(kz)} by $z^{-1}$ and comparing the valuation of the $(1,1)$-th entry, we get
\begin{align*}
v(k_{(1,1)}) = -n - \frac{m-n+2}{2} = \frac{-m-n}{2} -1,
\end{align*}
by equation \eqref{v((z))}. As $m$, $n \geq 0$, we get that $v(k_{(1,1)}) < 0$, which is a contradiction since $k_{(1,1)} \in \mathbb{Z}_p$, proving the claim.
Therefore supp$(Tf_-)$ mod $KZ$ does not contain matrices of the form $g^0_{n,\mu}$, for $n\geq 1$.

Let $f_{n,+}$ denote the radius $n$ part of $f_+$. We now claim that $f_+ = f_{0,+}$.
Suppose not, then there exists an $m\geq 1$ such that $f_{m,+} \neq 0$ and $f_{l,+} = 0$, for all $l>m$. From our initial assumption we know that
\begin{align}
Tf_- + Tf_+ = h. \label{Tf- plus Tf+ = h}
\end{align}
As $m$ is the non-zero maximum radius that appears in $f_+$, the radius $m+1$ part of the left hand
side of \eqref{Tf- plus Tf+ = h} is $T^+f_{m,+}$, since supp$(Tf_-)$ mod $KZ$ does not contain 
$g^0_{n,\lambda}$ for $n\geq 1$.  Comparing the radius  $m+1 \geq 2$ parts of 
\eqref{Tf- plus Tf+ = h}, we get
\begin{align}
T^+f_{m,+} = 0, \label{T plus of f(m,+) = 0 }
\end{align}
since $h$ lies in radius $1 < 2$.
 Write
$f_{m,+} = \sum\limits_{\mu}[g^0_{m,\mu}, A_{\mu}X + B_{\mu}Y]$, for some $A_\mu$, $B_\mu \in \bar{\mathbb{F}}_p$. Now
\begin{align*}
T^+f_{m,+} & =  \sum_{\lambda\in\mathbb{F}_p}\sum_{\mu}[g^0_{m,\mu}g^0_{1,[\lambda]}, A_\mu X + B_\mu (-[\lambda]X +pY)] \\
&= \sum_{\lambda\in\mathbb{F}_p}\sum_{\mu}[g^0_{m+1,\mu + p^m[\lambda]}, (A_\mu -[\lambda]B_\mu)X ].
\end{align*}
By \eqref{T plus of f(m,+) = 0 }, we have
\begin{align*}
A_\mu - [\lambda]B\mu = 0,
\end{align*}
for all $\mu\in I_m$ and $\lambda\in\mathbb{F}_p$. Substituting $\lambda = 0$, $1$ in the above equation we get $A_{\mu} = B_{\mu} = 0$, for all $\mu\in I_m$, implying $f_{m,+} = 0$, which is a contradiction, proving out claim that $f_+ = f_{0,+}$. 

Write $f_{0,+} = [1, AX + BY]$, for some $A$, $B \in \bar{\mathbb{F}}_p$. From \eqref{Tf- plus Tf+ = h}, we get in radius $1$ that
\begin{align*}
T^+f_{0,+} = h,
\end{align*}
since supp$(Tf_-)$ mod $KZ$ does not contain vertices of the form $g^0_{1,\lambda}$, for $\lambda \in I_1$. So we have 
\begin{align*}
T^+f_{0,+} &= \sum_{\lambda\in\mathbb{F}_p}[g^0_{1,[\lambda]}, AX + B(-[\lambda]X + pY)]\\
&= \sum_{\lambda\in\mathbb{F}_p}[g^0_{1,[\lambda]}, (A - [\lambda]B)X].
\end{align*}
As $T^+f_{0,+} = h$, we get a contradiction because $h$ is supported on functions of the form $[g, v]$ where $v$ has only $Y$ terms. Therefore $h \not\in T(\text{ind}_{KZ}^GJ_2)$.
\end{proof}
%\vspace{5mm}

\begin{prop}
{\label{F_1 when tau > t + 0.5}} 
Let $p>3$, $r \geq 2p+1$, $r = 3 + n(p-1)p^t $, with $t = v(r-3)$ and $v(a_p) = \frac{3}{2}$. Let $c=\frac{a_p^2-(r-2)\binom{r-1}{2}p^3}{pa_p}$ and $\tau = v(c)$. If $\tau > t + \frac{1}{2}$, then $F_2 = 0$. 
\end{prop}

\begin{proof}
Define the function $f=f_0 + f_\infty$, with
\begin{align*}
f_0 &=\frac{(p-1)(r-2)}{p^2(3-r)}\left[1,\sum_{\substack{{2\leq j< r-1}\\{j\equiv 2~\text{mod}~(p-1)}}}\binom{r-1}{j}X^{r-j}Y^j\right] + \frac{r-2}{2p}[1, X^{r-2}Y^2],\\
f_\infty &= \sum_{\lambda\in \mathbb{F}_p^\times} \xi_{g^0_{1,[\lambda]}} + (1-p)\xi_{g^0_{1,0}}, 
\end{align*}
where $\xi$ is as in Lemma~\ref{lemma 4.2}.

We compute $(T-a_p)f$ in several steps.

Computation in radius $-1$:
\begin{align}
T^-f_0 &= \frac{(p-1)(r-2)}{p^2(3-r)}\left[\alpha, \sum_{\substack{{2\leq j < r-1}\\{j \equiv 2~\text{mod}~(p-1)}}}\binom{r-1}{j}(pX)^{r-j}Y^j\right] + \frac{r-2}{2p}[\alpha, (pX)^{r-2}Y^2] \nonumber\\
& = O(p^2), \label{F_1 tau > t + 0.5 radius -1}
\end{align}
since $p^{r-j}\binom{r-1}{j} = p\left(p^{r-j-1}\binom{r-j}{r-j-1}\right)$, and $r-j-1 \geq p-1 > 3$ for all $j$ as above, so by Lemma~\ref{lemma 0.2}, we get $p^{r-j-1}\binom{r-1}{r-j-1} \equiv 0$ mod $p^{t+3}$. 

Computation in radius $0$:
\begin{align}
-a_p f_0 &= \frac{-a_p(p-1)(r-2)}{p^2(3-r)}\left[1, \sum_{\substack{{2\leq j < r-1}\\{j \equiv 2~\text{mod}~(p-1)}}}\binom{r-1}{j}X^{r-j}Y^j\right] + O(\sqrt{p}). \label{F_1 tau > t+ 0.5 -a_pf_0}
\end{align}
Let $h_{0,\infty}$ be the radius $0$ part of $(T-a_p)f_\infty$. By Lemma \ref{lemma 4.2}, we get
\begin{align}
h_{0,\infty} &= \frac{a_p(r-2)}{p^2(3-r)}\left(\sum_{\lambda\in\mathbb{F}_p^\times}[g^0_{1,[\lambda]}\alpha, XY^{r-1}] + (1-p)[g^0_{1,0}\alpha, XY^{r-1}]\right) \nonumber\\
&= \frac{a_p(r-2)}{p^2(3-r)}\left((p-1)\left[1, \sum_{\substack{{2\leq j \leq r-1}\\{j \equiv 2~\text{mod}~(p-1)}}}\binom{r-1}{j}(X)^{r-j}Y^j\right] + (1-p)[1, XY^{r-1}]\right) \nonumber \\
&= \frac{a_p(r-2)(p-1)}{p^2(3-r)}\left[1, \sum_{\substack{{2\leq j < r-1}\\{j \equiv 2~\text{mod}~(p-1)}}}\binom{r-1}{j}X^{r-j}Y^j\right]. \label{F_1 tau > t + 0.5 T minus f_1}
\end{align} 
So in radius $0$, from equations \eqref{F_1 tau > t+ 0.5 -a_pf_0} and \eqref{F_1 tau > t + 0.5 T minus f_1}, we get
\begin{align}
-a_pf_0 + h_{0,\infty} &= O(\sqrt{p}). \label{F_1 tau > t + 0.5 radius 0}
\end{align}

Computation in radius $1$:
\begin{align*}
T^+f_0 &= \frac{(p-1)(r-2)}{p^2(3-r)}\sum_{\lambda\in\mathbb{F}_p}\left[g^0_{1,[\lambda]}, \sum_{\substack{{2\leq j < r-1}\\{j \equiv 2~\text{mod}~(p-1)}}}\binom{r-1}{j}X^{r-j}(-[\lambda]X + pY)^j\right] \\
&~~~+ \frac{r-2}{2p}\sum_{\lambda\in\mathbb{F}_p}[g^0_{1,[\lambda]}, X^{r-2}(-[\lambda]X + pY)^2]\\
&= \frac{(p-1)(r-2)}{p^2(3-r)}\sum_{\lambda\in\mathbb{F}_p^\times}\left[g^0_{1,[\lambda]}, \sum_{\substack{{2\leq j < r-1}\\{j \equiv 2~\text{mod}~(p-1)}}}\binom{r-1}{j}X^{r-j}\sum_{i=0}^j\binom{j}{i}(-[\lambda]X)^{j-i}(pY)^i\right] \\
&~~~+ \frac{(p-1)(r-2)}{p^2(3-r)}\left[g^0_{1,0}, \sum_{\substack{{2\leq j < r-1}\\{j\equiv 2 ~\text{mod}~(p-1)}}}p^j\binom{r-1}{j}X^{r-j}Y^j\right]\\
&~~~ + \frac{r-2}{2p}\sum_{\lambda\in\mathbb{F}_p^\times}[g^0_{1,[\lambda]}, [\lambda]^2X^r - 2[\lambda]pX^{r-1}Y] + O(p)\\
 &= \frac{(p-1)(r-2)}{p^2(3-r)}\sum_{\lambda\in\mathbb{F}_p^\times}\left[g^0_{1,[\lambda]},\sum_{i = 0}^r(-[\lambda])^{2-i}p^i\sum_{\substack{{2\leq j < r-1}\\{j\equiv 2~\text{mod}~(p-1)}}}\binom{j}{i}\binom{r-1}{j}X^{r-i}Y^i\right]\\
&~~~+\frac{(p-1)(r-2)}{(3-r)}\left[g^0_{1,0},\binom{r-1}{2}X^{r-2}Y^2\right] + \frac{r-2}{2p}\sum_{\lambda\in\mathbb{F}_p^\times}[g^0_{1,[\lambda]}, [\lambda]^2X^r-2[\lambda]pX^{r-1}Y] + O(p),
\end{align*}
since by Lemma \ref{lemma 0.2}, $p^j\binom{r-1}{j} \equiv 0$ mod $p^{t+ 3}$ for all $2 < j< r-1$.

By Proposition \ref{Proposition 0.11}, part \textit{(1)}, we can write 
\begin{align}
\sum_{\substack{{2\leq j < r-1}\\{j\equiv 2 ~\text{mod}~(p-1)}}}\binom{r-1}{j}\equiv \frac{p(3-r)}{2} + \frac{p^2r(3-r)}{2} ~\text{mod}~p^{t+3}. \label{Propostion 3.7 part 1'}
\end{align}

Now by \eqref{Propostion 3.7 part 1'} and Proposition \ref{Proposition 0.11}, parts \textit{(2)}, \textit{(3)} and \textit{(4)}, the coefficients of $[g^0_{1,[\lambda]}, X^r]$, $[g^0_{1,[\lambda]}, X^{r-1}Y]$, $[g^0_{1,[\lambda]}, X^{r-2}Y^2]$ and $[g^0_{1,[\lambda]}, X^{r-i}Y^i]$ for $i\geq 3$ and $\lambda\neq  0$ in $T^+f_0$ are
\begin{align*}
\frac{[\lambda]^2(p-1)(r-2)}{p^2(3-r)}\sum_{\substack{{2 \leq j < r-1}\\{j \equiv 2 ~\text{mod}~(p-1)}}}\binom{r-1}{j} + \frac{[\lambda]^2(r-2)}{2p} 
&\equiv \frac{[\lambda]^2(p-1)(r-2)}{p^2(3-r)}\left(\frac{p(3-r)}{2} + \frac{p^2r(3-r)}{2}\right) \\
&~~~+ \frac{[\lambda]^2(r-2)}{2p} ~\text{mod}~p\\
& \equiv \frac{[\lambda]^2(r-2)(1-r)}{2}~\text{mod}~p,\\
\frac{-[\lambda](p-1)(r-2)}{p(3-r)}\sum_{\substack{{2 \leq j < r-1}\\{j \equiv 2 ~\text{mod}~(p-1)}}}j\binom{r-1}{j} - [\lambda](r-2) &\equiv \frac{-[\lambda](p-1)(r-2)}{p(3-r)}\left(\frac{p(r-1)(3-r)}{1-p}\right)\\
&~~~- [\lambda](r-2) ~\text{mod}~ p\\
&= [\lambda](r-2)^2 ~\text{mod}~p,\\
\frac{(p-1)(r-2)}{3-r}\sum_{\substack{{2 \leq j < r-1}\\{j \equiv 2 ~\text{mod}~(p-1)}}}\binom{j}{2}\binom{r-1}{j} &\equiv \frac{(p-1)(r-2)}{3-r}\left(\frac{\binom{r-1}{2}}{1-p}\right) ~\text{mod}~p\\
&=  \frac{-(r-2)\binom{r-1}{2}}{3-r}~\text{mod}~p,\\
\frac{(-[\lambda])^{2-i}(p-1)(r-2)p^i}{p^2(3-r)} \sum_{\substack{{2 \leq j < r-1}\\{j \equiv 2 ~\text{mod}~(p-1)}}} \binom{j}{i}\binom{r-1}{j} &\equiv 0 ~\text{mod}~p,
\end{align*}
respectively. Note that the coefficients of $[g^0_{1,[\lambda]}, X^r]$ and $[g^0_{1,[\lambda]}, X^{r-1}Y]$ are integral. So we get 
\begin{align}
T^+f_0 & = \frac{-(r-2)\binom{r-1}{2}}{3-r}\left(\sum_{\lambda\in\mathbb{F}_p^\times}[g^0_{1,[\lambda]}, X^{r-2}Y^2] + (1-p)[g^0_{1,0}, X^{r-2}Y^2]\right) + h + O(p), \label{F_1 tau > t+ 0.5 T plus f_0}
\end{align}
where $h$ is an integral linear combination of the terms of the form $[g, X^r]$ and $[g, X^{r-1}Y]$, for some $g\in G$.

Let $h_{1,\infty}$ be the radius $1$ part of $(T-a_p)f_\infty$. By Lemma \ref{lemma 4.2}, we get
\begin{align}
h_{1,\infty} &=  \frac{a_p^2}{p^3(3-r)}\left(\sum_{\lambda\in\mathbb{F}_p^\times}[g^0_{1,[\lambda]}, X^{r-2}Y^2 + (r-3)X^pY^{r-p}] + (1-p)[g^0_{1,0}, X^{r-2}Y^2 + (r-3)X^pY^{r-p}]\right). \label{F_1 tau > t+ 0.5 T h_infty}
\end{align}
>From equations \eqref{F_1 tau > t+ 0.5 T plus f_0} and \eqref{F_1 tau > t+ 0.5 T h_infty}, in radius $1$ we get
\begin{align}
T^+f_0 + h_{1,\infty} &= \left(\frac{a_p^2-(r-2)\binom{r-1}{2}p^3}{p^3(3-r)}\right)\left(\sum_{\lambda\in\mathbb{F}_p^\times}[g^0_{1,[\lambda]}, X^{r-2}Y^2] + (1-p)[g^0_{1,0}, X^{r-2}Y^2]\right)\nonumber\\
& ~~~ -\frac{a_p^2}{p^3}\sum_{\lambda\in\mathbb{F}_p}[g^0_{1,[\lambda]}, X^pY^{r-p}] + h + O(p). \label{F_1 tau > t+ 0.5 radius 1}
\end{align}
So from equations \eqref{F_1 tau > t + 0.5 radius -1}, \eqref{F_1 tau > t + 0.5 radius 0} and \eqref{F_1 tau > t+ 0.5 radius 1}, we get
\begin{align*}
(T-a_p)f & = \frac{ca_p}{p^2(3-r)}\left(\sum_{\lambda\in\mathbb{F}_p^\times}[g^0_{1,[\lambda]}, X^{r-2}Y^2] + (1-p)[g^0_{1,0}, X^{r-2}Y^2]\right) -\frac{a_p^2}{p^3}\sum_{\lambda\in\mathbb{F}_p}[g^0_{1,[\lambda]}, X^{r-p}Y^p]  \nonumber\\
&~~~+ h + O(\sqrt{p}). \label{F_1 tau > t + 0.5 (T-a_p)f}
\end{align*}

Now, the image of $h$ in ind$_{KZ}^GQ$ is $0$. Also as $\tau > t + \frac{1}{2}$, the image of $(T-a_p)f$ in ind$_{KZ}^GQ$ is 
\begin{align}
\overline{(T-a_p)f} & = \frac{-a_p^2}{p^3}\sum_{\lambda\in\mathbb{F}_p}[g^0_{1,[\lambda]}, X^{r-p}Y^p]\quad\text{mod}~ X_{r-1}  \nonumber \\
&= \frac{-a_p^2}{p^3}\sum_{\lambda\in\mathbb{F}_p}[g^0_{1,[\lambda]}, X^{r-p}Y^p - XY^{r-1}]\quad\text{mod}~ X_{r-1},
\end{align}
since $XY^{r-1}\in X_{r-1}$. Now it follows from part (ii) of Lemma~\ref{generator}, that %Lemma 8.5 in \cite{Bhattacharya-Ghate}, 
$\theta Y^{r-p-1} = X^pY^{r-p} - XY^{r-1}$ is identified with $Y\in J_2 = V_1\otimes D$. 
%Note that $J_2$ for us is actually $J_0$ in \cite{Bhattacharya-Ghate}. 
So the image of $\overline{(T-a_p)f}$ in ind$_{KZ}^GQ$ actually lies in ind$_{KZ}^GJ_2$ and is given by
\begin{align*}
\overline{(T-a_p)f} &= \frac{-a_p^2}{p^3}\sum_{\lambda\in\mathbb{F}_p}[g^0_{1,[\lambda]}, Y].
\end{align*}
By Lemma \ref{Technical lemma about F_1}, we get that $\overline{(T-a_p)f}\not\in T(\text{ind}_{KZ}^GJ_2)$, so it is non-zero in $\frac{\text{ind}_{KZ}^GJ_2}{T}$. But $\overline{(T-a_p)f}$ maps to zero in $\bar{\Theta}_{k,a_p}$. So the kernel of the map $\frac{\text{ind}_{KZ}^GJ_2}{T} \twoheadrightarrow F_2$ is not zero. Since $\frac{\text{ind}_{KZ}^GJ_2}{T}$ is irreducible, the kernel is all of $ \frac{\text{ind}_{KZ}^GJ_2}{T}$. Therefore $F_2 = 0$.
\end{proof}

We remark that when $t=0$ we know $F_2 =0$ always, by Proposition \ref{Structure of Q}, part \textit{(1)}, whereas  Proposition \ref{F_1 when tau > t + 0.5} only shows that $F_2 = 0$ when $\tau >\frac{1}{2}$.  

\vspace{5mm}

\section{Behaviour of $F_3$}
\label{sectionF3}

Finally, we study the structure of the subquotient $F_3$ of $\bar\Theta_{k,a_p}$ which is the hardest to treat.
Recall that $\ind_{KZ}^G J_3 \twoheadrightarrow F_3$, where $J_3 = V_{p-2} \otimes D^2$.

One might expect that $F_3 =0$ for $\tau < t + \frac{1}{2}$, but have not been able to prove this. However, 
for our purposes it will suffice to show that $F_3 = 0$ for $\tau \leq t$.
%\vspace{5mm}

\begin{prop}{\label{F_2 when tau =< t}}
Let $p >3$, $r \geq 2p+1$, $r = 3+n(p-1)p^t$, with $t =v(r-3)$ and $v(a_p) = \frac{3}{2}$. Let $c =\frac{a_p^2-(r-2)\binom{r-1}{2}p^3}{pa_p}$ and $\tau = v(c)$. If $\tau \leq t$, then $F_3 = 0$. 
\end{prop}
\begin{proof} The reader should compare the proof with the proof of Theorem 8.6 in \cite{Bhattacharya-Ghate}.

Note that $\tau \leq t$ forces $t$ to be at least $1$, since $\tau \geq \frac{1}{2}$.

 Define 
\begin{align*}
f_0 &= \frac{p-1}{a_pc}\left[1, \sum_{\substack{{2\leq j<r-1}\\{j\equiv 2~\text{mod}~(p-1)}}}\beta_jX^{r-j}Y^j\right],\\
f_\infty &= \frac{1}{c}\sum_{\lambda\in\mathbb{F}_p^\times}\chi_{g^0_{1,[\lambda]}}' + \frac{r(p-1)}{pc}\phi_{g^0_{1,0}},
\end{align*} 
where the $\beta_j$ are defined in the proof of Proposition \ref{beta}, and $\chi'$, $\phi$ are as in Lemma \ref{chi prime} and Lemma \ref{phi}, respectively. Let $f = f_0 + f_\infty$.

Computation in radius $-1$:
\begin{align}
T^-f_0 &= \frac{p-1}{a_pc}\left[\alpha, \sum_{\substack{{2\leq j<r-1}\\{j\equiv 2~\text{mod}~(p-1)}}}\beta_j(pX)^{r-j}Y^j\right] \nonumber\\
& = O(p^2), \label{F_2 tau <= t radius -1}
\end{align} 
since $r-j \geq p > 4$, for all $j$ as above, so by Proposition \ref{beta}, part \textit{(1)}, we have $p^{r-j}\beta_j \equiv p^{r-j}\binom{r}{j} ~\text{mod}~p^{t+4}$ and by Lemma \ref{lemma 0.1.}, we have $p^{r-j}\binom{r}{j} \equiv 0$ mod $p^{t+4}$.

Computation in radius $0$:
\begin{align}
-a_pf_0 = \frac{1-p}{c}\left[1, \sum_{\substack{{2\leq j < r-1}\\{j\equiv 2~\text{mod}~ (p-1)}}}\beta_jX^{r-j}Y^j\right]. \label{F_2 tau < = t  -a_pf_0}
\end{align}
Let $h_{0,\infty}$ be the radius $0$ part of $(T-a_p)f_{\infty}$. By Lemma \ref{chi prime} and Lemma \ref{phi}, we get
\begin{align}
h_{0,\infty}&= \frac{1}{c}\sum_{\lambda\in\mathbb{F}_p^\times}\left[g^0_{1,[\lambda]}\alpha, [\lambda]^{-1}Y^r\right] - \frac{r(p-1)}{pc}[g^0_{1,0}\alpha, pXY^{r-1}] + h_{\phi} + O(p)\nonumber\\
&= \frac{1}{c}\sum_{\lambda\in\mathbb{F}_p^\times}\left[1, [\lambda]^{-1}([\lambda]X + Y)^r\right]-\frac{r(p-1)}{c}[1, XY^{r-1}]  + h_{\phi} + O(p)\nonumber\\
&= \frac{p-1}{c}\left[1,  \sum_{\substack{{2\leq j< r-1}\\{j\equiv 2~\text{mod}~(p-1)}}}\binom{r}{j}X^{r-j}Y^j\right] + h_{\phi} + O(p), \label{F_2 tau < = t h_(0, infinity)}
\end{align}
by \eqref{sum of roots of 1}, where $h_{\phi}$ is an integral linear combination of the terms of the form $[g, X^{r-1}Y]$ and $[g, XY^{r-1}]$. So in radius $0$, by equations \eqref{F_2 tau < = t  -a_pf_0}, \eqref{F_2 tau < = t h_(0, infinity)} and the definition of $\beta_j$ in \eqref{def beta_j} for $j \neq 2$, $2p$, we get
\begin{align}
-a_pf_0 + h_{0, \infty} = \frac{1-p}{c}\left[1, \left(\beta_2 - \binom{r}{2}\right)X^{r-2}Y^2 + \left(\beta_{2p}-\binom{r}{2p}\right)X^{r-2p}Y^{2p}\right] + h_{\phi} + O(p). \label{F_2 tau < = t -a_pf_0 + h_(0, infinity)}
\end{align}
 By the definition of $\beta_2$ in \eqref{def beta_2}, we have
\begin{align*}
\beta_2-\binom{r}{2} &= \frac{-1}{2}\left(\sum_{\substack{{2<j<r-1}\\{j\equiv 2~\text{mod}~(p-1)}}}2b'j\binom{r}{j} + r(r-1)\right)\\
&\equiv -\frac{1}{2}\left(\sum_{\substack{{2<j\leq r-1}\\{j\equiv 2~\text{mod}~(p-1)}}}j\binom{r}{j}\right)~ \text{mod}~p^{t+1},
\end{align*} 
since $2b'\equiv 1$ mod $p^{t+1}$. By Proposition \ref{Proposition 0.10}, part \textit{(2)}, we get
\begin{align}
\beta_2 -\binom{r}{2} &\equiv \frac{r(r-3)}{2} ~\text{mod}~p^{t+1}. \label{beta_2 - r choose 2}
\end{align}
By the definition of $\beta_{2p}$ in \eqref{def beta_2p}, we get
\begin{align*}
\beta_{2p}-\binom{r}{2p} & =-\sum_{\substack{{2<j<r-1}\\{j\equiv 2~\text{mod}~(p-1)}\\{j\neq 2p}}}\binom{r}{j} -\beta_2-\binom{r}{2p}\\
&\equiv -\sum_{\substack{{2<j\leq r-1}\\{j\equiv 2~\text{mod}~(p-1)}}}\binom{r}{j} + r-\binom{r}{2}-\frac{r(r-3)}{2}~\text{mod}~p^{t+1},
\end{align*}
by \eqref{beta_2 - r choose 2}. Now, by Proposition \ref{Proposition 0.10}, part \textit{(1)}, we get  
\begin{align}
\beta_{2p} -\binom{r}{2p} \equiv r-3 -\frac{r(r-3)}{2} =\frac{(r-3)(2-r)}{2} \equiv \frac{3-r}{2}~\text{mod}~p^{t+1}, \label{beta_2p -  r choose 2p}
\end{align} since $2-r \equiv -1$ mod $p$ as $t \geq 1$. Since $v(c)\leq t$, by using equations \eqref{beta_2 - r choose 2} and \eqref{beta_2p -  r choose 2p} in \eqref{F_2 tau < = t -a_pf_0 + h_(0, infinity)}, in radius $0$, we have
\begin{align}
-a_pf_0 + h_{0,\infty} & = \frac{r-3}{2c}[1, rX^{r-2}Y^2- X^{r-2p}Y^{2p}] + h_{\phi} + O(p) \nonumber\\
&= \frac{r-3}{2c}[1, 3X^{r-2}Y^2 - X^{r-2p}Y^{2p}] + h_{\phi} + O(p) \nonumber\\
&= \frac{r-3}{2c}[1, 2X^{r-2}Y^2 + \theta X^{r-p-2}Y + \theta X^{r-2p -1}Y^p] + h_{\phi} + O(p), \label{F_2 tau < =t radius 0}
\end{align}
since $r\equiv 3$ mod $p$. 
%and 
%\begin{align*}3X^{r-2}Y^2 - X^{r-2p}Y^{2p} &= 2X^{r-2}Y^2 + X^{r-2}Y^2 - X^{r-p-1}Y^{p+1} + X^{r-p-1}Y^{p+1} - X^{r-2p}Y^{2p}\\
%& = 2X^{r-2}Y^2 + \theta X^{r-p-2}Y + \theta X^{r-2p -1}Y^p.
%\end{align*}

Computation in radius $1$:
\begin{align*}
T^+f_0 &= \frac{p-1}{a_pc}\sum_{\lambda\in\mathbb{F}_p}\left[g^0_{1,[\lambda]}, \sum_{\substack{{2\leq j< r-1}\\{j\equiv 2~\text{mod}~(p-1)}}}\beta_jX^{r-j}(-[\lambda]X+pY)^j\right]\\ 
 &= \frac{p-1}{a_pc}\sum_{\lambda\in\mathbb{F}_p^\times}\left[g^0_{1,[\lambda]}, \sum_{i = 0}^r(-[\lambda])^{2-i}p^i \sum_{\substack{{2\leq j< r-1}\\{j\equiv 2~\text{mod}~(p-1)}}}\beta_j\binom{j}{i}X^{r-i}Y^i\right]\\
&~~~+ \frac{p-1}{a_pc}\left[g^0_{1,0}, \sum_{\substack{{2\leq j< r-1}\\{j\equiv 2~\text{mod}~(p-1)}}}p^j\beta_jX^{r-j}Y^j\right].
\end{align*}

By Proposition \ref{beta}, parts \textit{(2)}, \textit{(3)} and \textit{(4)}, the coefficients of $[g^0_{1,[\lambda]}, X^r]$, $[g^0_{1,[\lambda]}, X^{r-1}Y]$, $[g^0_{1,[\lambda]}, X^{r-2}Y^2]$, $[g^0_{1,[\lambda]}, X^{r-3}Y^3]$ and $[g^0_{1,[\lambda]}, X^{r-i}Y^i]$ for $i\geq 4$ and $\lambda\neq 0$ in $T^+f_0$ are
\begin{align*}
\frac{[\lambda]^2(p-1)}{a_pc}\sum_{\substack{{2\leq j< r-1}\\{j\equiv 2~\text{mod}~(p-1)}}}\beta_j &\equiv 0 ~\text{mod}~\sqrt{p},\\
\frac{-p[\lambda](p-1)}{a_pc}\sum_{\substack{{2\leq j< r-1}\\{j\equiv 2~\text{mod}~(p-1)}}}j\beta_j &\equiv 0~\text{mod}~\sqrt{p},\\
\frac{p^2(p-1)}{a_pc}\sum_{\substack{{2\leq j< r-1}\\{j\equiv 2~\text{mod}~(p-1)}}}\binom{j}{2}\beta_j &\equiv 0~\text{mod}~\sqrt{p},\\
\frac{p^3(-[\lambda])^{-1}(p-1)}{a_pc}\sum_{\substack{{2\leq j< r-1}\\{j\equiv 2~\text{mod}~(p-1)}}}\binom{j}{3}\beta_j &\equiv \frac{-[\lambda]^{-1}p^3\binom{r}{3}}{a_pc}~\text{mod}~p,\\
\frac{p^i(-[\lambda])^{2-i}(p-1)}{a_pc}\sum_{\substack{{2\leq j< r-1}\\{j\equiv 2~\text{mod}~(p-1)}}}\binom{j}{i}\beta_j &\equiv 0~ \text{mod}~p,
\end{align*}
respectively. 
The coefficient of $[g^0_{1,0}, X^{r-j}Y^j]$ for $2<j<r-1$ and $j\equiv 2$ mod $(p-1)$ in $T^+f_0$ is
\begin{align*}
\left(\frac{p-1}{a_pc}\right)p^j\beta_j \equiv 0~\text{mod}~p^2,
\end{align*}
since by Proposition \ref{beta}, part \textit{(1)} and Lemma \ref{lemma 0.1.}, $p^j\beta_j \equiv p^j \binom{r}{j} \equiv 0$ mod $p^{t+ 4}$.

The coefficient of $[g^0_{1,0}, X^{r-2}Y^2]$ in $T^+f_0$ is 
\begin{align*}
\left(\frac{p-1}{a_pc}\right)p^2\beta_2 \equiv \frac{p^2(p-1)\binom{r}{2}}{a_pc}~\text{mod}~\sqrt{p},
\end{align*}
by Proposition \ref{beta}, part \textit{(1)}.
So we can write
\begin{align*}
T^+f_0 =   - \frac{p^3\binom{r}{3}}{a_pc}\sum_{\lambda\in\mathbb{F}_p^\times}\left[g^0_{1,[\lambda]}, [\lambda]^{-1}X^{r-3}Y^3\right] + \frac{p^2(p-1)\binom{r}{2}}{a_pc}[g^0_{1,0}, X^{r-2}Y^2] + O(\sqrt{p}).
\end{align*}

Let $h_{1,\infty}$ be the radius $1$ part of $(T-a_p)f_{\infty}$. By Lemma \ref{chi prime} and Lemma \ref{phi}, we get
\begin{align*}
h_{1,\infty} &= \frac{a_p}{c}\sum_{\lambda\neq 0}[g^0_{1,[\lambda]}, [\lambda]^{-1}X^{r-3}Y^3]- \frac{a_pr(p-1)}{pc}[g^0_{1,0}, X^{r-2}Y^2] + h'_{\phi} + O(p),
\end{align*} where $h'_{\phi}$ is an integral linear combination of the terms of the form $[g, X^{r-1}Y]$ and $[g, XY^{r-1}]$, for some $g\in G$.

\noindent The coefficient of $[g^0_{1,[\lambda]}, X^{r-3}Y^3]$ for $\lambda\neq 0$ in $T^+f_0 + h_{1,\infty}$ is 
\begin{align*}
[\lambda]^{-1}\left(\frac{a_p}{c}-\frac{\binom{r}{3}p^3}{a_pc}\right) = \frac{[\lambda]^{-1}}{c}\left(\frac{a_p^2-\binom{r}{3}p^3}{a_p}\right) =\frac{[\lambda]^{-1}p\tilde{c}}{c} \equiv 0 ~\text{mod}~p,
\end{align*}
since $\tau = \tilde{\tau}$ by Lemma \ref{Lemma 5.1}, part \textit{(iii)}.

\noindent The coefficient of $[g^0_{1,0}, X^{r-2}Y^2]$ in  $T^+f_0 + h_{1,\infty}$ is
\begin{align*}
\frac{p^2(p-1)\binom{r}{2}}{a_pc}-\frac{a_pr(p-1)}{pc} &= \frac{p-1}{c}\left(\frac{p^3\binom{r}{2}-a_p^2r}{pa_p}\right)
\equiv \frac{3(p-1)}{c}\left(\frac{p^3-a_p^2}{pa_p}\right) \equiv 3\left(\frac{a_p^2-p^3}{pa_pc}\right)~\text{mod}~p,
\end{align*} 
since $r\equiv \binom{r}{2}\equiv 3$ mod $p^t$. Note that, by \eqref{v(a_p^2-p^3)} and our hypothesis $\tau \leq t$ we have $v(a_p^2-p^3) = \tau + \frac{5}{2}$, so the constant $3\left(\frac{a_p^2-p^3}{pa_pc}\right)$ is a unit.

\noindent So finally in radius $1$ we get 
\begin{align}
 T^+f_0 + h_{1,\infty} &= 3\left(\frac{a_p^2-p^3}{pa_pc}\right)[g^0_{1,0}, X^{r-2}Y^2] + h'_{\phi} + O(\sqrt{p}). \label{F_2 tau < = t radius 1}
\end{align} 

>From equations \eqref{F_2 tau <= t radius -1}, \eqref{F_2 tau < =t radius 0} and \eqref{F_2 tau < = t radius 1}, we get
\begin{equation}
\begin{gathered}
(T-a_p)f  =  \frac{r-3}{2c}[1, 2X^{r-2}Y^2 +\theta X^{r-p-2}Y + \theta X^{r-2p-1}Y^p] + 3\left(\frac{a_p^2 - p^3}{pa_pc}\right)[g^0_{1,0}, X^{r-2}Y^2] \\
                + \> h + O(\sqrt{p}), 
\end{gathered}
\label{F_2 tau < = t (T-a_p)f}
\end{equation}
where $h = h_{\phi} + h'_{\phi}$.

Now, $h$ dies in ind$_{KZ}^GQ$. The image of $X^{r-2}Y^2$ in $Q$ is the same as that of $X^{r-2}Y^2-XY^{r-1}$, which is 
\begin{align}
\theta (X^{r-p-2}Y + \ldots + Y^{r-p-1}) \equiv \theta\left(\frac{r-p-2}{p-1}X^{r-p-2}Y + Y^{r-p-1}\right)\quad \text{mod}~V_r^{**}. \label{image of X^{r-2}Y^2 mod V_r**}
\end{align}
By parts (iii) and (ii)  of Lemma~\ref{generator}, %Lemma 8.5, part (ii) in \cite{Bhattacharya-Ghate}, 
$\theta X^{r-p-2}Y$ and $\theta Y^{r-p-1}$ map to $X^{p-2}$ and $0$ in $J_3 = V_{p-2}\otimes D^2$, respectively. 
%Note that $J_1$ in \cite{Bhattacharya-Ghate} is $J_3$ for us. 
Therefore 
\begin{align}
X^{r-2}Y^2 \mapsto (2-r)X^{p-2} \label{image of X^{r-2}Y^2 in J_2}
\end{align}
in $J_3$. As $t\geq 1$, we get that $X^{r-2}Y^2$ maps to $-X^{p-2}$ in $J_3$. Also $\theta X^{r-2p-1}Y^p$ maps to $X^{p-2}$ in $J_3$ because $\theta X^{r-2p-1}Y^p \equiv \theta X^{r-p-2}Y$ mod $\theta ^2$, so $2X^{r-2}Y^2 +\theta X^{r-p-2}Y + \theta X^{r-2p-1}Y^p $ maps to $0$ in $J_3$.
So the image of $\overline{(T-a_p)f}$ in ind$_{KZ}^GJ_3$ is given by
\begin{align*}
\overline{3\left(\frac{p^3-a_p^2}{pa_pc}\right)}[g^0_{1,0}, X^{p-2}].
\end{align*}
 But $\overline{(T-a_p)f}$ maps to zero in $\bar{\Theta}_{k,a_p}$. Therefore $F_3 = 0$ since $[g^0_{1,0}, X^{p-2}]$ generates ind$_{KZ}^GJ_3$.
\end{proof}

%\vspace{5mm}
\begin{remark} We make some remarks to show that our computations in the proof of   Proposition \ref{F_2 when tau =< t} are consistent with previous results. 
By Proposition \ref{Structure of Q}, part \textit{(2)}, as $t\geq 1$, we have $\frac{V_r^*}{V_r^{**}} \subset Q$, where $\frac{V_r^*}{V_r^{**}}$ is the non-trivial extension of $J_3 = V_{p-2}\otimes D^2$ by $J_2 = V_1\otimes D$.

Suppose $\tau < t$. Then the first term on the right hand side of  \eqref{F_2 tau < = t (T-a_p)f} dies mod $p$, so by equation \eqref{image of X^{r-2}Y^2 mod V_r**}, we have $\overline{(T-a_p)f} = [g^0_{1,0}, v]$, where $v = \overline{3\left(\frac{a_p^2-p^3}{pa_pc}\right)}\theta(-X^{r-p-2}Y + Y^{r-p-1}) \in \frac{V_r^*}{V_r^{**}}$. 
 % with a non-zero projection to \emph{ind}$_{KZ} ^GJ_2$. 
As $\frac{V_r^*}{V_r^{**}}$ is the non-trivial extension of $J_3$ by $J_2$ and $v$ projects to a non-zero vector in $J_3$, we see that $v$ generates $\frac{V_r^*}{V_r^{**}}$,  so $[g^0_{1,0}, v]$ generates $\ind_{KZ}^G\frac{V_r^*}{V_r^{**}}$, since  $[g^0_{1,0}, v]$ is an elementary function (i.e., supported mod $KZ$ on one matrix). This shows that $\ind_{KZ}^G\frac{V_r^*}{V_r^{**}} \twoheadrightarrow F_{2,3}$ is the zero map. In particular $F_2 = 0$ because $F_2 \subset F_{2,3}$, and we recover Proposition \ref{F_1 when tau <= t}, part \textit{(i)}, showing $F_2$ vanishes for $\tau < t$. 

This argument does not work for $\tau = t$. Indeed when $\tau = t$, by \eqref{F_2 tau < = t (T-a_p)f}, we see that $\overline{(T-a_p)f}$ is supported 
mod $KZ$ in radii $0$ and $1$, so is not an elementary function and so does not necessarily generate $\ind_{KZ}^G\frac{V_r^*}{V_r^{**}}$. So $F_2$ does not necessarily vanish when $\tau = t$. This is consistent with Proposition \ref{F_1 when tau <= t}, part \textit{(ii)}.
\end{remark}

%\vspace{5mm}

\begin{prop}\label{F_2 when tau < t + 1}
Let $p>3$, $r \geq 2p+1$, $r = 3+n(p-1)p^t$, with $t = v(r-3)$ and $v(a_p)= \frac{3}{2}$. Let $c=\frac{a_p^2-(r-2)\binom{r-1}{2}p^3}{pa_p}$ and $\tau = v(c)$. If $\tau < t+1$, then the map from $\mathrm{ind}_{KZ}^GJ_3$ to $F_3$ factors through the image of $T$, i.e., $\pi(p-2, 0, \omega^2)\twoheadrightarrow F_3$.
\end{prop}
\begin{proof}
If $v(2-r) > 0$, then $t = v(r-3) = 0$ and $\tau = \frac{1}{2}$, i.e., $v\left(a_p^2-\binom{r-1}{2}(r-2)p^3\right) = 3$. So condition $(\star)$ holds in the main result (Theorem 1.1) of \cite{Bhattacharya-Ghate} and the proposition follows from the first part of the proof of Theorem 9.1 in \cite{Bhattacharya-Ghate}. Note that $J_3$ for us is $J_1$ in \cite{Bhattacharya-Ghate}.

So throughout this proof we will assume that $v(2-r) = 0$.

Let
\begin{equation} 
\begin{gathered}
S_{r-1} = \sum\limits_{\substack{{2\leq j < r-1}\\{j\equiv 2~\text{mod}~(p-1)}}}\binom{r-1}{j}, \\
 T_{r-1} =  \sum\limits_{\substack{{2\leq j < r-1}\\{j\equiv 2~\text{mod}~(p-1)}}}j\binom{r-1}{j}. 
 \end{gathered} \label{S(r-1) and T(r-1)}
 \end{equation}
 
Define   
\begin{align*}
f_0 &= \frac{1-p}{ca_p}\left[1, \sum_{\substack{{2 \leq j <r-1}\\{j\equiv 2  ~\text{mod}~(p-1)}}}\binom{r-1}{j}X^{r-j}Y^j\right] + \left(\frac{p-1}{ca_p}\right)S_{r-1}[1,X^{r-2}Y^2]\\
&~~~+ \left(\frac{2S_{r-1}-T_{r-1}}{ca_p}\right)[1, X^{r-2}Y^2-X^{r-p-1}Y^{p+1}],\\
f_1 &= \frac{-1}{cp(2-r)}\left(\sum_{\lambda \in\mathbb{F}_p^\times}[g^0_{1,[\lambda]},X^{r-2}Y^2-(r-2)XY^{r-1}] +(1-p)[g^0_{1,0},X^{r-2}Y^2-(r-2)XY^{r-1}]\right)\\
&~~~+ \frac{1-p}{cp(2-r)}\left(\sum_{\lambda\in\mathbb{F}_p^\times}\left[g^0_{1,[\lambda]}, \sum_{\substack{{2\leq j\leq r-p}\\{j\equiv 2~\text{mod}~(p-1)}}}\mkern- 16 mu\binom{r}{j}X^{r-j}Y^j\right] +(1-p)\left[g^0_{1,0}, \sum_{\substack{{2\leq j\leq r-p}\\{j\equiv 2~\text{mod}~(p-1)}}}\mkern- 16 mu\binom{r}{j}X^{r-j}Y^j\right]\right)\\
&~~~ -\frac{3(3-r)}{2c(2-r)}\left(\sum_{\lambda\in\mathbb{F}_p^\times }[g^0_{1,[\lambda]}, X^{r-2}Y^2] +(1-p)[g^0_{1,0}, X^{r-2}Y^2]\right),\\
f_\infty &= \frac{pr-2}{(2-r)^2}\left(\sum_{\lambda\in\mathbb{F}_p^\times}\xi''_{g^0_{1,[\lambda]}} +(1-p)\xi''_{g^0_{1,0}}\right) +\frac{1}{r-2}\sum_{\mu\in\mathbb{F}_p^\times}\left(\sum_{\lambda\in\mathbb{F}_p^\times}\psi'_{g^0_{1,[\lambda]}}+(1-p)\psi'_{g^0_{1,0}}\right),
\end{align*}
where $\xi''$ and $\psi'$ are as in Lemma \ref{xi''} and Lemma \ref{psi'}. Let $f=f_0+f_1+f_\infty$.

Computation in radius $-1$: 
\begin{align}
T^-f_0 &= \frac{(1-p)}{ca_p}\left[\alpha, \sum_{\substack{{2 \leq j <r-1}\\{j\equiv 2  ~\text{mod}~(p-1)}}}\binom{r-1}{j}(pX)^{r-j}Y^j\right] + \left(\frac{p-1}{ca_p}\right)S_{r-1}[\alpha,(pX)^{r-2}Y^2]
\nonumber \\
&~~~+ \left(\frac{2S_{r-1}-T_{r-1}}{ca_p}\right)[\alpha, (pX)^{r-2}Y^2-(pX)^{r-p-1}Y^{p+1}] \nonumber \\
&= O(p), \label{F_2  tau < t+1 radius -1 }
\end{align}
since $r-j-1 \geq p-1 > 3$ for all $j$ as above so by Lemma \ref{lemma 0.2}, $p\left(p^{r-j-1}\binom{r-1}{r-j-1}\right)\equiv 0~\text{mod}~ p^{t+4}$ and by Proposition \ref{Proposition 0.11}, part \textit{(1)} and \textit{(2)}, we have $v(S_{r-1}) = v(T_{r-1}) = t +1$, so $p^{r-2}S_{r-1} \equiv p^{r-p-1}(2S_{r-1} - T_{r-1})\equiv0 ~\text{mod}~p^{t+4}$.

Computation in radius $0$:
\begin{align}
-a_pf_0 & =\frac{p-1}{c}\left[1, \sum_{\substack{{2 \leq j <r-1}\\{j\equiv 2  ~\text{mod}~(p-1)}}}\binom{r-1}{j}X^{r-j}Y^j\right] + \left(\frac{1-p}{c}\right)S_{r-1}[1,X^{r-2}Y^2] \nonumber\\
&~~~ - \left(\frac{2S_{r-1}-T_{r-1}}{c}\right)[1, X^{r-2}Y^2-X^{r-p-1}Y^{p+1}] \nonumber\\
& = \frac{p-1}{c}\left[1, \sum_{\substack{{2 \leq j <r-1}\\{j\equiv 2  ~\text{mod}~(p-1)}}}\binom{r-1}{j}X^{r-j}Y^j\right] + O(p^\delta), \label{F_2 tau < t+1 -a_pf_0}
\end{align}
for $\delta :=t+1-\tau > 0$, because $v(S_{r-1}) = v(T_{r-1}) = t +1$ by Proposition \ref{Proposition 0.11}, part \textit{(1)} and \textit{(2)} and the fact that $v(c) =\tau < t+1$.
Now
\begin{align*}
T^-f_1 &= \frac{-1}{cp(2-r)}\bigg(\sum_{\lambda \in\mathbb{F}_p^\times}[g^0_{1,[\lambda]}\alpha,(pX)^{r-2}Y^2-(r-2)pXY^{r-1}]  \\
&~~~+(1-p)[g^0_{1,0}\alpha,(pX)^{r-2}Y^2-(r-2)pXY^{r-1}]\bigg)\\
&~~~+ \frac{1-p}{cp(2-r)}\left(\sum_{\lambda\in\mathbb{F}_p^\times}\left[g^0_{1,[\lambda]}\alpha, \sum_{\substack{{2\leq j\leq r-p}\\{j\equiv 2~\text{mod}~(p-1)}}}\binom{r}{j}(pX)^{r-j}Y^j\right]\right.\\&~~~\left. +(1-p)\left[g^0_{1,0}\alpha, \sum_{\substack{{2\leq j\leq r-p}\\{j\equiv 2~\text{mod}~(p-1)}}}\binom{r}{j}(pX)^{r-j}Y^j\right]\right)\\
&~~~-\frac{3(3-r)}{2c(2-r)}\left(\sum_{\lambda\in\mathbb{F}_p^\times}[g^0_{1,[\lambda]}\alpha, (pX)^{r-2}Y^2] +(1-p)[g^0_{1,0}\alpha, (pX)^{r-2}Y^2]\right)\\
&= -\frac{1}{c}\left(\sum_{\lambda\in\mathbb{F}_p^\times}[1, X([\lambda]X + Y)^{r-1}]+(1-p)[1, XY^{r-1}]\right) + O(p^2),
\end{align*} 
since $r-2 > t +4$, by Lemma \ref{lemma 0.1.}, $p^{r-j}\binom{r}{j} \equiv 0 ~\text{mod}~p^{t+4}$ for all $j$ as above, and $v(c) < t+1$. 
By \eqref{sum of roots of 1}, we have
\begin{align}
T^-f_1 = \frac{1-p}{c}\left[1, \sum_{\substack{{2\leq j<r-1}\\{j\equiv 2~\text{mod}~(p-1)}}}\binom{r-1}{j}X^{r-j}Y^j\right] + O(p^2). 
\label{F_2 tau < t+1 T minus f_1}
\end{align}
So in radius $0$ by equations \eqref{F_2 tau < t+1 -a_pf_0} and \eqref{F_2 tau < t+1 T minus f_1}, we get
\begin{align}
-a_pf_0 +T^-f_1 = O(p^{\delta'}), \label{F_2 tau < t+1 radius 0}
\end{align}
with $\delta':= \min(\delta, 2)$.

Computation in radius $1$:
\begin{align*}
T^+f_0 &=\left( \frac{1-p}{ca_p}\right)\sum_{\lambda \in\mathbb{F}_p}\left[g^0_{1,[\lambda]}, \sum_{\substack{{2\leq j < r-1}\\{j \equiv 2~\text{mod}~(p-1)}}}\binom{r-1}{j}X^{r-j}(-[\lambda]X+pY)^j\right]\\
 &~~~+\left(\frac{p-1}{ca_p}\right)S_{r-1}\sum_{\lambda\in\mathbb{F}_p}[g^0_{1,[\lambda]}, X^{r-2}(-[\lambda]X+pY)^2]\\
 &~~~+\left( \frac{2S_{r-1}-T_{r-1}}{ca_p}\right)\sum_{\lambda\in\mathbb{F}_p}[g^0_{1,[\lambda]}, X^{r-2}(-[\lambda]X+pY)^2-X^{r-p-1}(-[\lambda]X+pY)^{p+1}]\\
 &= \left(\frac{1-p}{ca_p}\right)\sum_{\lambda \in\mathbb{F}_p^\times}\left[g^0_{1,[\lambda]},\sum_{i =0}^{r} (-[\lambda])^{2-i}p^i \sum_{\substack{{2\leq j < r-1}\\{j \equiv 2~\text{mod}~(p-1)}}}\binom{j}{i}\binom{r-1}{j}X^{r-i}Y^i\right]\\
 &~~~ + \left(\frac{1-p}{ca_p}\right)\left[g^0_{1,0}, p^2\binom{r-1}{2}X^{r-2}Y^2\right] + \left(\frac{p-1}{ca_p}\right)S_{r-1}\sum_{\lambda\in\mathbb{F}_p^\times}[g^0_{1,[\lambda]}, [\lambda]^2X^r - 2p[\lambda]X^{r-1}Y]\\
 &~~~ + \left(\frac{2S_{r-1}-T_{r-1}}{ca_p}\right)\sum_{\lambda\in\mathbb{F}_p^\times}[g^0_{1,[\lambda]}, [\lambda]p(p-1)X^{r-1}Y] + O(\sqrt{p}),
\end{align*}
by Lemma \ref{lemma 0.2}, $p^j\binom{r-1}{j} \equiv 0$ mod $p^{t+3}$ for $2< j < r-1$ and using the fact that $v(S_{r-1}) = v(T_{r-1}) =t + 1$. 

By the definition of $S_{r-1}$, $T_{r-1}$ and Proposition \ref{Proposition 0.11}, part \textit{(3)}, part \textit{(4)} and the fact that $v(S_{r-1}) = v(T_{r-1}) = t+1$, the coefficients of $[g^0_{1,[\lambda]}, X^r]$, $[g^0_{1,[\lambda]}, X^{r-1}Y]$, $[g^0_{1,[\lambda]}, X^{r-2}Y^2]$ and $[g^0_{1,[\lambda]}, X^{r-i}Y^i]$ for $i\geq 3$ and $\lambda \neq 0$ in $T^+f_0$ are
\begin{align*}
\left(\frac{1-p}{ca_p}\right)[\lambda]^2(S_{r-1} - S_{r-1}) &= 0,\\
\frac{p(p-1)}{ca_p}[\lambda](T_{r-1} - 2S_{r-1} + 2S_{r-1} - T_{r-1}) &= 0,\\
\left(\frac{1-p}{ca_p}\right)p^2\sum_{\substack{{2\leq j < r-1}\\{j\equiv 2~\text{mod}~(p-1)}}}\binom{j}{2}\binom{r-1}{j} &\equiv \frac{p^2\binom{r-1}{2}}{ca_p}~\text{mod}~\sqrt{p},\\
\left(\frac{1-p}{ca_p}\right)(-[\lambda])^{2-i}p^i \sum_{\substack{{2\leq j < r-1}\\{j\equiv 2~\text{mod}~(p-1)}}}\binom{j}{i}\binom{r-1}{j} &\equiv 0~\text{mod}~\sqrt{p},
\end{align*}
respectively. Therefore we get 
\begin{align}
T^+f_0 = \frac{p^2\binom{r-1}{2}}{ca_p}\left(\sum_{\lambda\in\mathbb{F}_p^\times}[g^0_{1,[\lambda]}, X^{r-2}Y^2+(1-p)[g^0_{1,0},X^{r-2}Y^2]\right) + O(\sqrt{p}). \label{F_2 tau < t+1 T plus f_0}
\end{align}
Now 
\begin{align}
-a_pf_1  &= \frac{a_p}{cp(2-r)}\left(\sum_{\lambda \in\mathbb{F}_p^\times}[g^0_{1,[\lambda]},X^{r-2}Y^2-(r-2)XY^{r-1}] +(1-p)[g^0_{1,0},X^{r-2}Y^2-(r-2)XY^{r-1}]\right) \nonumber\\
&~~~+ \frac{(p-1)a_p}{cp(2-r)}\left(\sum_{\lambda\in\mathbb{F}_p^\times}\left[g^0_{1,[\lambda]}, \sum_{\substack{{2\leq j\leq r-p}\\{j\equiv 2~\text{mod}~(p-1)}}}\binom{r}{j}X^{r-j}Y^j\right] \right.\nonumber\\
&~~~\left. +(1-p)\left[g^0_{1,0}, \sum_{\substack{{2\leq j\leq r-p}\\{j\equiv 2~\text{mod}~(p-1)}}}\binom{r}{j}X^{r-j}Y^j\right]\right)  + O(\sqrt{p}), \label{F_2 tau < t +1 -a_pf_1 }
\end{align}
since $v\left(\frac{3(3-r)a_p}{2c(2-r)}\right) > \frac{1}{2}$ as $(2-r)$ is a unit.

Let $h_{1, \infty}$ be the radius $1$ part of $(T-a_p)f_\infty$. By Lemma \ref{xi''} and Lemma \ref{psi'}, we get 
\begin{align*}
h_{1,\infty} &=\frac{pr-2}{(2-r)^2}\left(\frac{a_p}{pc}\right)\left(\sum_{\lambda\in\mathbb{F}_p^\times}[g^0_{1,[\lambda]},(2-r)XY^{r-1}] +(1-p)[g^0_{1,0},(2-r)XY^{r-1}]\right)\\
&~~~+ \frac{1}{r-2}\left(\frac{a_p}{pc}\right)\sum_{\mu\in\mathbb{F}_p^\times}\left(\sum_{\lambda\in\mathbb{F}_p^\times}[g^0_{1,[\lambda]},[\mu]^{-1}([\mu]X+Y)^r]+(1-p)[g^0_{1,0},[\mu]^{-1}([\mu]X+Y)^r]\right) + O(p^\epsilon),
\end{align*}
for $\epsilon = \min(\delta, \frac{1}{2}) > 0$.
By \eqref{sum of roots of 1}, we can rewrite the above as
\begin{align}
h_{1,\infty} &=\frac{pr-2}{(2-r)^2}\left(\frac{a_p}{pc}\right)\left(\sum_{\lambda\in\mathbb{F}_p^\times}[g^0_{1,[\lambda]},(2-r)XY^{r-1}] +(1-p)[g^0_{1,0},(2-r)XY^{r-1}]\right) \nonumber \\
&~~~+ \frac{p-1}{r-2}\left(\frac{a_p}{pc}\right)\left(\sum_{\lambda\in\mathbb{F}_p^\times}\left[g^0_{1,[\lambda]},\sum_{\substack{{2\leq j \leq r-1}\\{j\equiv 2~\text{mod}~(p-1)}}}\binom{r}{j}X^{r-j}Y^j\right]\right. \nonumber\\
&~~~\left.+(1-p)\left[g^0_{1,0}, \sum_{\substack{{2\leq j \leq r-1}\\{j\equiv 2~\text{mod}~(p-1)}}}\binom{r}{j}X^{r-j}Y^j\right]\right) + O(p^\epsilon) \nonumber\\
&=-\frac{a_p}{pc}\left(\sum_{\lambda\in\mathbb{F}_p^\times}[g^0_{1,[\lambda]}, XY^{r-1}]+(1-p)[g^0_{1,0}, XY^{r-1}]\right)\nonumber \\
&~~~+\frac{a_p(p-1)}{(r-2)pc}\left(\sum_{\lambda\in\mathbb{F}_p^\times}\left[g^0_{1,[\lambda]}, \sum_{\substack{{2\leq j\leq r-p}\\{j\equiv 2 ~\text{mod}~(p-1)}}}\mkern- 18 mu\binom{r}{j}X^{r-j}Y^j\right]+(1-p)\left[g^0_{1,0},\sum_{\substack{{2\leq j\leq r-p}\\{j\equiv 2 ~\text{mod}~(p-1)}}}\mkern -18 mu\binom{r}{j}X^{r-j}Y^j\right]\right) \nonumber\\
&~~~+O(p^\epsilon). \label{F_2 tau < t+1 h_{1,infty}}
\end{align}
So finally in radius $1$ by \eqref{F_2 tau < t+1 T plus f_0}, \eqref{F_2 tau < t +1 -a_pf_1 } and \eqref{F_2 tau < t+1 h_{1,infty}}, we get 
\begin{align}
T^+f_0 -a_pf_1 + h_{1,\infty} & =\left(\frac{p^2\binom{r-1}{2}}{ca_p}+\frac{a_p}{cp(2-r)}\right)\left(\sum_{\lambda\in\mathbb{F}_p^\times}[g^0_{1,[\lambda]}, X^{r-2}Y^2]+(1-p)[g^0_{1,0}, X^{r-2}Y^2]\right) \mkern-7 mu +\mkern-5 mu O(p^\epsilon)\nonumber\\
&=\frac{1}{2-r}\left(\sum_{\lambda\in\mathbb{F}_p^\times}[g^0_{1,[\lambda]}, X^{r-2}Y^2]+(1-p)[g^0_{1,0}, X^{r-2}Y^2]\right) + O(p^\epsilon)\nonumber\\
&= \frac{1}{2-r}\sum_{\lambda\in \mathbb{F}_p}[g^0_{1,[\lambda]}, X^{r-2}Y^2] +O(p^\epsilon). \label{F_2 tau < t+1 radius 1}
\end{align}
Computation in radius $2$:
\begin{align*}
T^+f_1 &= -\frac{1}{cp(2-r)}\sum_{\lambda \in\mathbb{F}_p^\times}\sum_{\mu\in\mathbb{F}_p}[g^0_{2, [\lambda]+p[\mu]}, X^{r-2}(-[\mu]X+pY)^2 - (r-2)X(-[\mu]X+pY)^{r-1}]\\
&~~~-\frac{(1-p)}{cp(2-r)}\sum_{\mu\in\mathbb{F}_p}[g^0_{2,p[\mu]}, X^{r-2}(-[\mu]X+pY)^2 - (r-2)X(-[\mu]X+pY)^{r-1}]\\
&~~~+\frac{(1-p)}{cp(2-r)}\sum_{\lambda\in\mathbb{F}_p^\times}\sum_{\mu\in\mathbb{F}_p}\left[g^0_{2,[\lambda]+p[\mu]}, \sum_{\substack{{2\leq j\leq r-p}\\{j\equiv 2(p-1)}}}\binom{r}{j}X^{r-j}(-[\mu]X+pY)^j\right]\\
&~~~+\frac{(1-p)^2}{cp(2-r)}\sum_{\mu\in\mathbb{F}_p}\left[g^0_{2, p[\mu]}, \sum_{\substack{{2\leq j\leq r-p}\\{j\equiv 2~\text{mod}~(p-1)}}}\binom{r}{j}X^{r-j}(-[\mu]X+pY)^j\right]\\
&~~~ -\frac{3(3-r)}{2c(2-r)}\sum_{\lambda\in\mathbb{F}_p^\times}\sum_{\mu\in\mathbb{F}_p}[g^0_{2,[\lambda]+p[\mu]}, X^{r-2}(-[\mu]X+pY)^2]\\
&~~~ -\frac{3(3-r)(1-p)}{2c(2-r)}\sum_{\mu\in\mathbb{F}_p}[g^0_{2,p[\mu]}, X^{r-2}(-[\mu]X+pY)^2].
\end{align*}
By using Lemma \ref{lemma 0.2} in the first sum, Lemma \ref{lemma 0.1.} for $\mu =0$ in the third  and fourth sums and by our assumption $v(c) < t+1$, we get
\begin{align*}
T^+f_1 &= \frac{-1}{cp(2-r)}\sum_{\lambda\in\mathbb{F}_p^\times}\sum_{\mu\in\mathbb{F}_p^\times}\left[g^0_{2,[\lambda] + p[\mu]}, [\mu]^2(3-r)X^r -[\mu]pr(3-r)X^{r-1}Y \right. \\
&~~~ \left. + \> p^2\left(1- (r-2)\binom{r-1}{2}\right)X^{r-2}Y^2\right] \\
&~~~ +\frac{(p-1)}{cp(2-r)}\sum_{\mu\in\mathbb{F}_p^\times}\left[g^0_{2, p[\mu]}, [\mu]^2(3-r)X^r -[\mu]pr(3-r)X^{r-1}Y \right. \\
&~~~ \left. + \> p^2\left(1- (r-2)\binom{r-1}{2}\right)X^{r-2}Y^2\right] \\
&~~~ -\frac{1}{cp(2-r)}\left(\sum_{\lambda\in\mathbb{F}_p^\times}[g^0_{2,[\lambda]}, p^2X^{r-2}Y^2] + (1-p)[g^0_{2,0}, p^2X^{r-2}Y^2]\right)\\
&~~~+ \frac{(1-p)}{cp(2-r)}\sum_{\lambda\in\mathbb{F}_p^\times}\sum_{\mu\in\mathbb{F}_p^\times}\left[g^0_{2, [\lambda] + p[\mu]}, \sum_{i=0}^r(-[\mu])^{2-i}p^i\sum_{\substack{{2\leq j \leq r-p}\\{j\equiv 2~\text{mod}~(p-1)}}}\binom{j}{i}\binom{r}{j}X^{r-i}Y^i\right]\\
&~~~+ \frac{(1-p)^2}{cp(2-r)}\sum_{\mu\in\mathbb{F}_p^\times}\left[g^0_{2,  p[\mu]}, \sum_{i=0}^r(-[\mu])^{2-i}p^i\sum_{\substack{{2\leq j \leq r-p}\\{j\equiv 2~\text{mod}~(p-1)}}}\binom{j}{i}\binom{r}{j}X^{r-i}Y^i\right]\\
&~~~+ \frac{(1-p)}{cp(2-r)}\left(\sum_{\lambda\in\mathbb{F}_p^\times}\left[g^0_{2,[\lambda]}, p^2\binom{r}{2}X^{r-2}Y^2\right] + (1-p)\left[g^0_{2, 0}, p^2\binom{r}{2}X^{r-2}Y^2\right]\right)\\
&~~~- \frac{3(3-r)}{2c(2-r)}\sum_{\mu\in\mathbb{F}_p^\times}\left(\sum_{\lambda\in\mathbb{F}_p^\times}[g^0_{2, [\lambda] + p[\mu]},[\mu]^2X^r] + (1-p)[g^0_{2, p[\mu]}, [\mu]^2X^r]\right) + O(p).
\end{align*}
Recall $\delta = t+1 -\tau > 0$. By Proposition \ref{Proposition 0.10}, part \textit{(1)}, \textit{(2)}, \textit{(3)}, \textit{(4)}, \textit{(5)} and Lemma \ref{lemma 0.2}, the coefficients of $[g^0_{2,[\lambda]+p[\mu]}, X^r]$, $[g^0_{2,[\lambda]+p[\mu]}, X^{r-1}Y]$, $[g^0_{2,[\lambda]+p[\mu]}, X^{r-2}Y^2]$, $[g^0_{2,[\lambda]+p[\mu]}, X^{r-3}Y^3]$ and $[g^0_{2,[\lambda]+p[\mu]}, X^{r-i}Y^i]$ for $i\geq 4$, $\lambda \neq 0$, $\mu\neq 0$ in $T^+f_1$ are
\begin{align*}
\frac{[\mu]^2(r-3)}{cp(2-r)} &+ \frac{[\mu]^2(1-p)}{cp(2-r)}\sum_{\substack{{2\leq j \leq r-p}\\{j\equiv 2~\text{mod}~(p-1)}}}\binom{r}{j} -\frac{3[\mu]^2(3-r)}{2c(2-r)}\\
&\equiv \frac{[\mu]^2(r-3)}{cp(2-r)} + \frac{[\mu]^2(1-p)}{cp(2-r)}\left((3-r) + \frac{5(3-r)p}{2(1-p)}\right) -\frac{3[\mu]^2(3-r)}{2c(2-r)} ~\text{mod}~p^\delta \\
&= \frac{[\mu]^2(r-3)}{cp(2-r)}(1-1) + \frac{[\mu]^2(3-r)}{c(2-r)}\left(-1 + \frac{5}{2} - \frac{3}{2}\right)~\text{mod}~p^\delta\\
&= 0 ~\text{mod}~p^\delta,\\
\frac{[\mu]r(3-r)}{c(2-r)} & -\frac{[\mu](1-p)}{c(2-r)}\sum_{\substack{{2\leq j \leq r-p}\\{j\equiv 2~\text{mod}~(p-1)}}}j\binom{r}{j} \\
& \equiv \frac{[\mu]r(3-r)}{c(2-r)} -\frac{[\mu](1-p)}{c(2-r)}\left(2\binom{r}{2} + r(3-r) - (r-1)\binom{r}{r-1}\right)~\text{mod}~p^\delta\\
&\equiv 0~\text{mod}~p^\delta, \\
\frac{-p}{c(2-r)} &\left(1- (r-2)\binom{r-1}{2}\right)  + \frac{p(1-p)}{c(2-r)}\sum_{\substack{{2\leq j \leq r-p}\\{j\equiv 2~\text{mod}~(p-1)}}}\binom{j}{2}\binom{r}{j}\\
&\equiv \frac{-p}{c(2-r)} \left(1- (r-2)\binom{r-1}{2}\right) + \frac{p(1-p)}{c(2-r)}\left(\binom{r}{2} - \binom{r-1}{2}\binom{r}{r-1}\right)~\text{mod}~p^{\delta +1}\\
&\equiv \frac{-p}{c(2-r)}\left(\frac{(3-r)(r^2- 2r + 2)}{2}\right) + \frac{p(1-p)}{c(2-r)}\left(\frac{r(r-1)(3-r)}{2}\right)~\text{mod}~p^{\delta + 1}\\
& \equiv 0~\text{mod}~p^\delta,\\
\frac{[\mu]^{-1}p^2(p-1)}{c(2-r)}&\sum_{\substack{{2\leq j \leq r-p}\\{j\equiv 2~\text{mod}~(p-1)}}}\binom{j}{3}\binom{r}{j}\\
& \equiv \frac{[\mu]^{-1}p^2(p-1)}{c(2-r)}\left(\frac{\binom{r}{3}}{p-1}-\binom{r-1}{3}\binom{r}{r-1}\right)~\text{mod}~p\\
&\equiv \frac{[\mu]^{-1}p^2\binom{r}{3}}{c(2-r)}~\text{mod}~p,\\
\frac{(1-p)}{cp(2-r)}& (-[\mu])^{2-i}p^i\left(-\binom{r-1}{i}\binom{r}{r-1}\right)\\
&\equiv 0~\text{mod}~p,
\end{align*}
respectively.

The coefficients of $[g^0_{2,p[\mu]}, X^{r}]$, $[g^0_{2,p[\mu]}, X^{r-1}Y]$, $[g^0_{2,p[\mu]}, X^{r-2}Y^2]$, $[g^0_{2,p[\mu]}, X^{r-3}Y^3]$ and $[g^0_{2,p[\mu]}, X^{r-i}Y^i]$ for $i\geq 4$, $\mu\neq 0$ in $T^+f_1$ are $(1-p)$ times the coefficients above. %, i.e., they are
%\begin{align*}
%(1-p)\left(\frac{[\mu]^2(r-3)}{cp(2-r)} + \frac{[\mu]^2(1-p)}{cp(2-r)}\sum_{\substack{{2\leq j \leq r-p}\\{j\equiv 2~\text{mod}~(p-1)}}}\binom{r}{j} -\frac{3[\mu]^2(3-r)}{2c(2-r)}\right) &\equiv 0~\text{mod}~p^\delta,\\
%(1-p)\left(\frac{[\mu]r(3-r)}{c(2-r)}  -\frac{[\mu](1-p)}{c(2-r)}\sum_{\substack{{2\leq j \leq r-p}\\{j\equiv 2~\text{mod}~(p-1)}}}j\binom{r}{j}\right) &\equiv 0~\text{mod}~p^\delta,\\
%(1-p)\left(\frac{-p}{c(2-r)} \left(1- (r-2)\binom{r-1}{2}\right)  + \frac{p(1-p)}{c(2-r)}\sum_{\substack{{2\leq j \leq r-p}\\{j\equiv 2~\text{mod}~(p-1)}}}\binom{j}{2}\binom{r}{j}\right) &\equiv 0~\text{mod}~p^\delta,\\
%(1-p)\left(\frac{[\mu]^{-1}p^2(p-1)}{c(2-r)}\sum_{\substack{{2\leq j \leq r-p}\\{j\equiv 2~\text{mod}~(p-1)}}}\binom{j}{3}\binom{r}{j}\right) &\equiv \frac{[\mu]^{-1}p^2(1-p)\binom{r}{3}}{c(2-r)}\\&\text{mod}~p,\\
%\frac{(1-p)^2}{cp(2-r)}(-[\mu])^{2-i}p^i\left(-\binom{r-1}{i}\binom{r}{r-1}\right) &\equiv 0~\text{mod}~p,
%\end{align*}
%respectively. 
So we have 
\begin{align*}
T^+f_1 &= \left(\frac{(1-p)p\binom{r}{2}}{c(2-r)}-\frac{p}{c(2-r)}\right)\left(\sum_{\lambda\in\mathbb{F}_p^\times}[g^0_{2,[\lambda]}, X^{r-2}Y^2]+(1-p)[g^0_{2,0}, X^{r-2}Y^2]\right) \\
&~~~+\frac{p^2\binom{r}{3}}{c(2-r)}\sum_{\mu\in\mathbb{F}_p^\times}\left(\sum_{\lambda\in\mathbb{F}_p^\times}[g^0_{2,[\lambda]+p[\mu]}, [\mu]^{-1}X^{r-3}Y^3]+(1-p)[g^0_{2,p[\mu]}, [\mu]^{-1}X^{r-3}Y^3]\right) 
+ O(p^{\tilde{\delta}}),
\end{align*}
where $\tilde{\delta} =\min(\delta, 1) > 0$.

Let $h_{2,\infty}$ be the radius $2$ part of $(T-a_p)f_{\infty}$. By Lemma \ref{xi''} and Lemma \ref{psi'}, we get\begin{align*}
h_{2,\infty} &=\frac{2-pr}{(2-r)^2}\left(\frac{a_p^2}{p^2c}\right)\left(\sum_{\lambda\in\mathbb{F}_p^\times}[g^0_{2,[\lambda]}, X^{r-2}Y^2] +(1-p)[g^0_{2,0}, X^{r-2}Y^2]\right)\\
&~~~ +\frac{1}{r-2}\left(\frac{a_p^2}{pc}\right)\sum_{\mu\in\mathbb{F}_p^\times}\left(\sum_{\lambda\in\mathbb{F}_p^\times}[g^0_{[\lambda]+p[\mu]},[\mu]^{-1}X^{r-3}Y^3]+(1-p)[g^0_{2,p[\mu]},[\mu]^{-1}X^{r-3}Y^3]\right) + O(p^\epsilon),
\end{align*}
where $\epsilon = \min(\delta, \frac{1}{2}) > 0$.

The coefficient of $[g^0_{2,[\lambda]}, X^{r-2}Y^2]$ for $\lambda\neq 0$ in $T^+f_1 + h_{2,\infty}$ is 
\begin{align*}
\frac{(1-p)p\binom{r}{2}}{c(2-r)} &- \frac{p}{c(2-r)} +\frac{2-pr}{(2-r)^2}\left(\frac{a_p^2}{p^2c}\right)\\
&= \frac{(1-p)p\binom{r}{2}}{c(2-r)} -\frac{p}{c(2-r)}  +\frac{a_p^2}{cp^2(2-r)} +\frac{(1-p)ra_p^2}{cp^2(2-r)^2}\\
&=  \frac{(1-p)r}{p^2c(2-r)^2}\left(a_p^2 - \binom{r-1}{2}p^3\right) + \frac{a_p^2-p^3}{p^2c(2-r)} \\
& \equiv 0 ~\text{mod}~p^\epsilon,
\end{align*}
for $\epsilon = \min(\delta, \frac{1}{2}) > 0$, since $v(a_p^2-p^3)\geq \min(\tau+ \frac{5}{2} , t+3)$ by \eqref{v(a_p^2-p^3)} and $v\left(a_p^2 - \binom{r-1}{2}p^3\right)\geq \min(\tau+ \frac{5}{2}, t+3)$ by a similar argument, and $\tau < t+1$.

\noindent Similarly, the coefficient of $[g^0_{2,0}, X^{r-2}Y^2]$ in $T^+f_1 + h_{1,\infty}$ is $(1-p)$ times the computation above, so is $0$ mod $p^\epsilon$.
%\begin{align*}
%(1-p)\left(\frac{(1-p)p\binom{r}{2}}{c(2-r)} - \frac{p}{c(2-r)} +\frac{2-pr}{(2-r)^2}\left(\frac{a_p^2}{p^2c}\right)\right)\equiv 0~\text{mod}~p^\epsilon.
%\end{align*}

\noindent The coefficient of $[g^0_{2,[\lambda]+p[\mu]}, X^{r-3}Y^3]$ for $\lambda\neq 0$, $\mu\neq 0$ in $T^+f_1 + h_{2,\infty}$ is 
  \begin{align*}
 [\mu]^{-1}\left(\frac{p^2\binom{r}{3}}{c(2-r)}+\frac{a_p^2}{(r-2)pc}\right)
 = \frac{[\mu]^{-1}}{pc(r-2)}\left(a_p^2-p^3\binom{r}{3}\right) \equiv 0~\text{mod}~p,
\end{align*}
since as before $v\left(a_p^2-\binom{r}{3}p^3\right)\geq \min(\tau + \frac{5}{2}, t+3)$ and $\tau < t+1$.

\noindent Similarly, the coefficient of $[g^0_{2,p[\mu]}, X^{r-3}Y^3]$ for $\mu\neq 0$ in $T^+f_1 + h_{2,\infty}$ is $(1-p)$ times the computation above, so is $0$ mod $p$.
%\begin{align*}
%&[\mu]^{-1}(1-p)\left(\frac{p^2\binom{r}{3}}{c(2-r)}+\frac{a_p^2}{(r-2)pc}\right) \equiv 0~\text{mod}~p.
%\end{align*}
Therefore  in radius $2$ we get 
\begin{align}
T^+f_1 + h_{2,\infty} = O(p^\epsilon),\label{F_2 tau < t+1 radius 2}
\end{align}
with $\epsilon > 0$.

By \eqref{F_2  tau < t+1 radius -1 }, \eqref{F_2 tau < t+1 radius 0}, \eqref{F_2 tau < t+1 radius 1} and \eqref{F_2 tau < t+1 radius 2}, we get
 \begin{align*}
(T-a_p)f = \frac{1}{2-r}\sum_{\lambda\in\mathbb{F}_p}[g^0_{1,[\lambda]}, X^{r-2}Y^2] + O(p^\epsilon),
\end{align*}
with $\epsilon > 0$.

By \eqref{image of X^{r-2}Y^2 in J_2}, $X^{r-2}Y^2$ maps to $(2-r)X^{p-2}$ in $J_3 = V_{p-2} \otimes D^2$. So the image $\overline{(T-a_p)f}$ inside ind$_{KZ}^GJ_3$ is 
\begin{align*}
\sum_{\lambda\in\mathbb{F}_p}[g^0_{1,[\lambda]}, X^{p-2}] = T[1,X^{p-2}].
\end{align*}
As $\overline{(T-a_p)f}$ maps to zero in $\bar{\Theta}_{k,a_p}$ and $[1, X^{p-2}]$ generates ind$_{KZ}^GJ_3$, we get that $T(\text{ind}_{KZ}^GJ_3) \subset$ kernel$\left(\text{ind}_{KZ}^GJ_3 \twoheadrightarrow F_3\right)$.
Therefore, if $\tau < t+1$, then $ F_3$ is a quotient of $\frac{\text{ind}^G_{KZ}J_3}{T}$. 
\end{proof} 

Finally, we give the structure of $F_3$ for $\tau \geq t+1$. 
% \vspace{5mm}

\begin{prop} \label{F_2 when tau = > t + 1}
Let $p>3$, $r\geq 2p+1$, $r = 3 + n(p-1)p^t$, with $t=v(r-3)$ and $v(a_p) =\frac{3}{2}$. Let $c =\frac{a_p^2-(r-2)\binom{r-1}{2}p^3}{pa_p}$ and $\tau =v(c)$.
\begin{enumerate}[label =(\roman{*})]
\item  If $\tau = t+1$, then $F_3$ is a quotient of 
$$\frac{\mathrm{ind}^G_{KZ}J_3}{T^2-\bar{d}T+1},$$
where $$d =\frac{2c}{(2-r)(3-r)p}.$$
\item  If $ \tau > t+1$, then $F_3$ is a quotient of \begin{align*}
\frac{\mathrm{ind}^G_{KZ}J_3}{T^2+1}.
\end{align*}
\end{enumerate}
Thus in both cases $F_3$ is a quotient of $\pi(p-2,\lambda,\omega^2) \oplus \pi(p-2,\lambda^{-1}, \omega^2)$, where $$\lambda + \frac{1}{\lambda}= \bar{d}.$$
\end{prop}
\begin{proof} The idea of the proof is similar to the proof of Proposition~\ref{F_2 when tau < t + 1}.

Note that $\tau \geq t +1$ forces $v(2-r) = 0$. Indeed, $v(2-r) > 0$ implies $\tau = \frac{1}{2}$.

Let $S_{r-1} = \sum\limits_{\substack{{2\leq j <r-1}\\{j\equiv 2~\text{mod}~(p-1)}}}\binom{r-1}{j} $ and $T_{r-1} = \sum\limits_{\substack{{2\leq j <r-1}\\{j\equiv 2~\text{mod}~(p-1)}}}j\binom{r-1}{j}$ be as in \eqref{S(r-1) and T(r-1)}.
 
Define 
\begin{align*}
f_0 &= \frac{2(p-1)}{(2-r)(3-r)pa_p}\left(\left[1,\sum_{\substack{{2\leq j < r-1}\\{j \equiv 2~\text{mod}~(p-1)}}} \binom{r-1}{j}X^{r-j}Y^j\right] -S_{r-1}[1, X^{r-2}Y^2]\right)\\ 
&~~~-\frac{2(2S_{r-1} - T_{r-1})}{(2-r)(3-r)pa_p}[1, X^{r-2}Y^2 - X^{r-p-1}Y^{p+1}],\\
f_1 &= \frac{2}{(2-r)^2(3-r)p^2}\Bigg(\mkern-3mu\sum_{\lambda \in\mathbb{F}_p^\times}[g_{1,[\lambda]}^0, X^{r-2}Y^2\mkern -5mu-(r-2)XY^{r-1}]+(1-p)[g^0_{1,0}, X^{r-2}Y^2\mkern-5mu - (r-2)XY^{r-1}]\Bigg) \\
&~~~-\frac{2(1-p)}{(2-r)^2(3-r)p^2}\left(\sum_{\lambda\in\mathbb{F}_p^\times}\left[g^0_{1,[\lambda]}, \mkern-18mu\sum_{\substack{{2\leq j\leq r-p}\\{j\equiv 2~\text{mod}~(p-1)}}}\mkern-18mu\binom{r}{j}X^{r-j}Y^j\right]+(1-p)\left[g^0_{1,0}, \mkern-18mu\sum_{\substack{{2\leq j\leq r-p}\\{j\equiv 2~\text{mod}~(p-1)}}}\mkern-18mu\binom{r}{j}X^{r-j}Y^j\right]\right)\\
&~~~+ \frac{3}{(2-r)^2p}\left(\sum_{\lambda \in\mathbb{F}_p^\times}[g^0_{1,[\lambda]}, X^{r-2}Y^2] +(1-p)[g^0_{1,0}, X^{r-2}Y^2]\right),\\
f_\infty &= \frac{2(2-pr)}{(2-r)^3(3-r)}\left(\sum_{\lambda\in\mathbb{F}_p^\times} \xi'_{g^0_{1,[\lambda]}}+ (1-p)\xi'_{g^0_{1,0}}\right) + \frac{2}{(2-r)^2(3-r)}\sum_{\mu \in\mathbb{F}_p^\times}\left(\sum_{\lambda\in\mathbb{F}_p^\times}\psi_{g^0_{1,[\lambda]}} + (1-p)\psi_{g^0_{1,0}}\right),
\end{align*}
where $\xi'$ and $\psi$ are as in Lemma \ref{lemma 0.18} and Lemma \ref{Lemma 0.17.}, respectively.
Let $f = f_0 + f_1 + f_\infty$. 

The careful reader will notice that the functions $f_0$, $f_1$, $f_\infty$ and therefore $f$ above are $\frac{-2c}{(2-r)(3-r)p}$ times the corresponding functions defined in the proof of Proposition \ref{F_2 when tau < t + 1}, noting $\xi'' = \frac{p}{c}\xi'$ and $\psi' =\frac{p}{c}\psi$. Essentially, we have replaced
the constant $c$ by $(3-r)p$, which allows us to investigate the case
$\tau \geq t+1$ using the functions for $\tau < t + 1$. Thus, many
of the computations in this proof are similar to those in the proof of  Proposition~\ref{F_2 when tau < t + 1}. However, the computations must necessarily 
be more complicated, since now  
$\overline{(T-a_p)f}$  contributes in radii $0$ and $2$ and not just in radius $1$.

Computation in radius $-1$:
\begin{align}
T^-f_0 &= \frac{2(p-1)}{(2-r)(3-r)pa_p}\left(\left[\alpha,\sum_{\substack{{2\leq j < r-1}\\{j \equiv 2~\text{mod}~(p-1)}}} \binom{r-1}{j}(pX)^{r-j}Y^j\right] - S_{r-1}[\alpha, (pX)^{r-2}Y^2]\right) \nonumber\\
 &~~~ -\frac{2}{(2-r)(3-r)pa_p}(2S_{r-1} - T_{r-1})[\alpha, (pX)^{r-2}Y^2 - (pX)^{r-p-1}Y^{p+1}] \nonumber\\
 & = O(p), \label{F_2 tau > = t+1 radius -1}
\end{align}
since $r-j-1 \geq p-1 > 3$ for all $j$ as above, so by Lemma \ref{lemma 0.2}, $p\left(p^{r-j-1}\binom{r-1}{r-j-1}\right) \equiv 0 ~\text{mod}~p^{t+4}$ and by Proposition \ref{Proposition 0.11}, parts \textit{(1)} and \textit{(2)}, we have $v(S_{r-1}) = v(T_{r-1}) = t + 1$, so $p^{r-2}S_{r-1} \equiv p^{r-p-1}(2S_{r-1} - T_{r-1}) \equiv 0~\text{mod}~p^{t+4}$.

Computation in radius $0$:
\begin{align}
-a_pf_0 &= \frac{-2(p-1)}{(2-r)(3-r)p}\left(\left[1,\sum_{\substack{{2\leq j < r-1}\\{j \equiv 2~\text{mod}~(p-1)}}} \binom{r-1}{j}X^{r-j}Y^j\right] -S_{r-1}[1, X^{r-2}Y^2]\right) \nonumber\\ 
&~~~+\frac{2(2S_{r-1} - T_{r-1})}{(2-r)(3-r)p}[1, X^{r-2}Y^2 - X^{r-p-1}Y^{p+1}].\label{F_2 tau > = t+1 -a_pf_0}
\end{align}
Now
\begin{align*}
T^-f_1 &= \frac{2}{(2-r)^2(3-r)p^2}\left(\sum_{\lambda\in\mathbb{F}_p^\times}[g^0_{1,[\lambda]}\alpha, -(r-2)pXY^{r-1}] +(1-p)[g^0_{1,0},-p(r-2)XY^{r-1}]\right) \\
&~~~ -\frac{2(1-p)}{(2-r)^2(3-r)p^2}\left(\sum_{\lambda\in\mathbb{F}_p^\times}\left[g^0_{1,[\lambda]}\alpha, \sum_{\substack{{2\leq j\leq r-p}\\{j\equiv 2~\text{mod}~(p-1)}}}\binom{r}{j}(pX)^{r-j}Y^j\right] \right.\\
&~~~ +\left.(1-p)\left[g^0_{1,0}\alpha, \sum_{\substack{{2\leq j\leq r-p}\\{j\equiv 2~\text{mod}~(p-1)}}}\binom{r}{j}(pX)^{r-j}Y^j\right]\right) + O(p)\\
&= \frac{2}{(2-r)(3-r)p}\left(\sum_{\lambda\in\mathbb{F}_p^\times}[1, X([\lambda]X + Y)^{r-1}]+ (1-p)[1, XY^{r-1}]\right) + O(p),
\end{align*}
since $r-2 \geq t+4$ and by Lemma \ref{lemma 0.1.}, $p^{r-j}\binom{r}{r-j}\equiv 0$ mod $p^{t+4}$ since $r-j \geq 4$ for $2\leq j \leq r-p$.
By \eqref{sum of roots of 1}, we have
\begin{align}
T^-f_1 & = \frac{2(p-1)}{(2-r)(3-r)p}\left[1, \sum_{\substack{{2\leq j < r-1}\\{j\equiv 2~\text{mod}~(p-1)}}}\binom{r-1}{j}X^{r-j}Y^j\right] + O(p). \label{F_2 tau >= t+1 T minus f_1}
\end{align}
By equations \eqref{F_2 tau > = t+1 -a_pf_0}, \eqref{F_2 tau >= t+1 T minus f_1}, we get
\begin{align*}
-a_pf_0 + T^-f_1 &= \frac{2(p-1)S_{r-1}}{(2-r)(3-r)p}[1, X^{r-2}Y^2] +\frac{2(2S_{r-1} - T_{r-1})}{(2-r)(3-r)p}[1, X^{r-2}Y^2 - X^{r-p-1}Y^{p+1}] + O(p).
\end{align*}
By Proposition \ref{Proposition 0.11}, parts \textit{(1)} and \textit{(2)} mod $p^{t+2}$, in radius $0$, we get
\begin{align}
 -a_pf_0 + T^-f_1 &= \frac{-1}{2-r}[1, X^{r-2}Y^2] + \frac{2}{2-r}\left(1-\frac{r-1}{1-p}\right)[1, X^{r-2}Y^2 - X^{r-p-1}Y^{p+1}] +O(p) \nonumber \\
&= \frac{-1}{2-r}[1,X^{r-2}Y^2] + 2[1, X^{r-2}Y^2 - X^{r-p-1}Y^{p+1}] +O(p) \nonumber \\
&= \frac{-1}{2-r}[1, X^{r-2}Y^2] + 2[1, \theta X^{r-p-2}Y] + O(p). \label{F_2 tau >= t+1 radius 0}
\end{align}
 Computations in radius $1$:
 \begin{align*}
 T^+f_0 &= \frac{2(p-1)}{(2-r)(3-r)pa_p}\sum_{\lambda\in\mathbb{F}_p}\left[g^0_{1,[\lambda]}, \sum_{\substack{{2\leq j < r-1}\\{j \equiv 2~\text{mod}~(p-1)}}}\binom{r-1}{j}X^{r-j}(-[\lambda]X+pY)^j\right]\\
 &~~~-\frac{2(p-1)S_{r-1}}{(2-r)(3-r)pa_p}\sum_{\lambda\in\mathbb{F}_p}[g^0_{1,[\lambda]}, X^{r-2}(-[\lambda]X+pY)^2]\\
 &~~~-\frac{2(2S_{r-1}-T_{r-1})}{(2-r)(3-r)pa_p}\sum_{\lambda\in\mathbb{F}_p}[g^0_{1,[\lambda]}, X^{r-2}(-[\lambda]X+pY)^2-X^{r-p-1}(-[\lambda]X+pY)^{p+1}]\\
 &= \frac{2(p-1)}{(2-r)(3-r)pa_p}\sum_{\lambda\in\mathbb{F}_p^\times}\left[g^0_{1,[\lambda]}, \sum_{i=0}^r(-[\lambda])^{2-i}p^i\sum_{\substack{{2\leq j < r-1}\\{j\equiv 2~\text{mod}~(p-1)}}}\binom{j}{i}\binom{r-1}{j}X^{r-i}Y^i\right]\\
 &~~~ + \frac{2(p-1)}{(2-r)(3-r)pa_p}\left[g^0_{1,0}, p^2\binom{r-1}{2}X^{r-2}Y^2\right]\\
 &~~~ - \frac{2(p-1)S_{r-1}}{(2-r)(3-r)pa_p}\sum_{\lambda\in\mathbb{F}_p^\times}[g^0_{1,[\lambda]}, [\lambda]^2X^r -2p[\lambda]X^{r-1}Y] \\
 &~~~ - \frac{2(2S_{r-1}-T_{r-1})}{(2-r)(3-r)pa_p}\sum_{\lambda\in\mathbb{F}_p^\times}[g^0_{1,[\lambda]}, [\lambda]p(p-1)X^{r-1}Y] + O(\sqrt{p}),
 \end{align*}
 since, by Lemma~\ref{lemma 0.2}, for $2< j< r-1$, we have $p^j\binom{r-1}{j}\equiv 0~ \text{mod}~ p^{t+3}$, and by Proposition~ \ref{Proposition 0.11}, parts \textit{(1)} and \textit{(2)}, we have $v(S_{r-1}) = v(T_{r-1}) = t+1$.
 
By Proposition \ref{Proposition 0.11}, parts \textit{(3)} and \textit{(4)}, the coefficients of $[g^0_{1,[\lambda]}, X^r]$, $[g^0_{1,[\lambda]}, X^{r-1}Y]$, $[g^0_{1,[\lambda]}, X^{r-2}Y^2]$ and $[g^0_{1,[\lambda]}, X^{r-i}y^i]$ for $i \geq 3$ and $\lambda\neq 0$ in $T^+f_0$ are 
\begin{align*}
\frac{2[\lambda]^2(p-1)}{(2-r)(3-r)pa_p}(S_{r-1} - S_{r-1}) &= 0,\\
\frac{-2[\lambda](p-1)}{(2-r)(3-r)a_p}(T_{r-1} -2S_{r-1} + 2S_{r-1} - T_{r-1}) & = 0,\\
\frac{2(p-1)}{(2-r)(3-r)pa_p}\left(p^2\sum_{\substack{{2\leq j < r-1}\\{j\equiv 2~\text{mod}~(p-1)}}}\binom{j}{2}\binom{r-1}{j}\right) &\equiv \frac{-2p\binom{r-1}{2}}{(2-r)(3-r)a_p}~\text{mod}~\sqrt{p},\\
\frac{2(-[\lambda])^{2-i}(p-1)}{(2-r)(3-r)pa_p}\left(p^i\sum_{\substack{{2\leq j < r-1}\\{j \equiv 2~\text{mod}~(p-1)}}}\binom{j}{i}\binom{r-1}{j}\right) &\equiv 0 ~\text{mod}~\sqrt{p},
\end{align*} 
respectively. Therefore we get
 \begin{align}
T^+f_0 & = \frac{-2p\binom{r-1}{2}}{(2-r)(3-r)a_p}\left(\sum_{\lambda\in\mathbb{F}_p^\times}[g^0_{1,[\lambda]}, X^{r-2}Y^2] + (1-p)[g^0_{1,0}, X^{r-2}Y^2]\right) + O(\sqrt{p}). \label{F_2 tau >=  t+1 T plus f_0}
\end{align}
Now 
\begin{align}
-a_pf_1 &= \frac{-2a_p}{(2-r)^2(3-r)p^2}\Bigg(\sum_{\lambda\in\mathbb{F}_p^\times}[g^0_{1,[\lambda]}, X^{r-2}Y^2 - (r-2)XY^{r-1}]\nonumber\\
&~~~+ (1-p)[g^0_{1,0}, X^{r-2}Y^2 -(r-2)XY^{r-1}]\Bigg)\nonumber\\
&~~~+ \frac{2a_p(1-p)}{(2-r)^2(3-r)p^2}\left(\sum_{\lambda\in\mathbb{F}_p^\times}\left[g^0_{1,[\lambda]}, \sum_{\substack{{2\leq j \leq r-p}\\{j\equiv 2~\text{mod}~(p-1)}}}\binom{r}{j}X^{r-j}Y^j\right]\right.\nonumber\\
&~~~ + \left. (1-p)\left[g^0_{1,0}, \sum_{\substack{{2\leq j \leq r-p}\\{j\equiv 2~\text{mod}~(p-1)}}}\binom{r}{j}X^{r-j}Y^j\right]\right) + O(\sqrt{p}), \label{F_2 tau >= t+1 -a_pf_1 }
\end{align}
since $v\left(\frac{a_p}{p}\right) = \frac{1}{2} > 0$.

Let $h_{1,\infty}$ be the radius $1$ part of $(T-a_p)f_{\infty}$. By Lemma \ref{lemma 0.18} and Lemma \ref{Lemma 0.17.}, we get
\begin{align*}
h_{1,\infty} &= \frac{2(2-pr)}{(2-r)^3(3-r)}\left(\frac{a_p}{p^2}\right)\left(\sum_{\lambda \in\mathbb{F}_p^\times}[g^0_{1,[\lambda]}, (2-r)XY^{r-1} ] + (1-p)[g^0_{1,0}, (2-r)XY^{r-1}]\right) \\
&~~+ \frac{2}{(2-r)^2(3-r)}\left(\frac{a_p}{p^2}\right)\sum_{\mu \in\mathbb{F}_p^\times}\left(\sum_{\lambda \in\mathbb{F}_p^\times}[g^0_{1,[\lambda]},[\mu]^{-1}([\mu]X+Y)^r] +(1-p)[g^0_{1,0}, [\mu]^{-1}([\mu]X+Y)^r]\right)\\
&~~ + O(\sqrt{p}).
\end{align*}
By \eqref{sum of roots of 1}, we get
\begin{align*}
h_{1,\infty} &=\frac{2a_p(2-pr)}{(2-r)^2(3-r)p^2}\left(\sum_{\lambda \in\mathbb{F}_p^\times}[g^0_{1,[\lambda]}, XY^{r-1}] + (1-p)[g^0_{1,0}, XY^{r-1}]\right)\\
&~~+ \frac{2a_p(p-1)}{(2-r)^2(3-r)p^2}\left(\sum_{\lambda \in\mathbb{F}_p^\times}\left[g^0_{1,[\lambda]}, \mkern-18mu\sum_{\substack{{2\leq j \leq r-1}\\{j\equiv 2~\text{mod}~(p-1)}}}\mkern-18mu\binom{r}{j}X^{r-j}Y^j\right] + (1-p)\left[g^0_{1,0}, \mkern-18mu\sum_{\substack{{2\leq j \leq r-1}\\{j\equiv 2~\text{mod}~(p-1)}}}\mkern-18mu\binom{r}{j}X^{r-j}Y^j\right]\right)\\
&~~+ O(\sqrt{p}).
\end{align*}
Now combining the coefficient of $[g^0_{1,[\lambda]}, XY^{r-1}]$ from the first and second sums, we get
\begin{align}
h_{1,\infty} &=\frac{2a_p}{(2-r)(3-r)p^2}\left(\sum_{\lambda \in\mathbb{F}_p^\times}[g^0_{1,[\lambda]}, XY^{r-1}]+ (1-p)[g^0_{1,0}, XY^{r-1}]\right)\nonumber \\
&~~ + \frac{2a_p(p-1)}{(2-r)^2(3-r)p^2}\left(\sum_{\lambda \in\mathbb{F}_p^\times}\left[g^0_{1,[\lambda]}, \mkern -18mu\sum_{\substack{{2\leq j \leq r-p}\\{j\equiv 2~\text{mod}~(p-1)}}}\mkern-18mu\binom{r}{j}X^{r-j}Y^j\right] + (1-p)\left[g^0_{1,0}, \mkern-18mu\sum_{\substack{{2\leq j \leq r-p}\\{j\equiv 2~\text{mod}~(p-1)}}}\mkern-18mu\binom{r}{j}X^{r-j}Y^j\right]\right)\nonumber\\
&~~ + O(\sqrt{p}). \label{F_2 tau >= t+1 h_(1,infty)}
\end{align}
So in radius $1$, by \eqref{F_2 tau >=  t+1 T plus f_0}, \eqref{F_2 tau >= t+1 -a_pf_1 }, \eqref{F_2 tau >= t+1 h_(1,infty)}, we get
\begin{align}
&T^+f_0 -a_pf_1 + h_{1,\infty} \nonumber\\
& =\left(\frac{-2p\binom{r-1}{2}}{(2-r)(3-r)a_p}-\frac{2a_p}{(2-r)^2(3-r)p^2}\right)\left(\sum_{\lambda \in\mathbb{F}_p^\times}[g^0_{1,[\lambda]}, X^{r-2}Y^2]+ (1-p)[g^0_{1,0}, X^{r-2}Y^2]\right) + O(\sqrt{p}) \nonumber\\
&=\frac{-2c}{(2-r)^2(3-r)p}\left(\sum_{\lambda \in\mathbb{F}_p}[g^0_{1,[\lambda]}, X^{r-2}Y^2]\right) + O(\sqrt{p}). \label{F_2 tau >= t+1 radius 1}
\end{align}

Computation in radius $2$:
\begin{align*}
T^+f_1 &= \frac{2}{(2-r)^2(3-r)p^2}\sum_{\lambda \in\mathbb{F}_p^\times}\sum_{\mu\in\mathbb{F}_p}[g^0_{2, [\lambda]+p[\mu]}, X^{r-2}(-[\mu]X+pY)^2 - (r-2)X(-[\mu]X+pY)^{r-1}]\\
&~~~+ \frac{2(1-p)}{(2-r)^2(3-r)p^2}\sum_{\mu\in\mathbb{F}_p}[g^0_{2,p[\mu]}, X^{r-2}(-[\mu]X+pY)^2 - (r-2)X(-[\mu]X+pY)^{r-1}]\\
&~~~- \frac{2(1-p)}{(2-r)^2(3-r)p^2}\sum_{\lambda\in\mathbb{F}_p^\times}\sum_{\mu\in\mathbb{F}_p}\left[g^0_{2,[\lambda]+p[\mu]}, \sum_{\substack{{2\leq j\leq r-p}\\{j\equiv 2~\text{mod}~(p-1)}}}\binom{r}{j}X^{r-j}(-[\mu]X+pY)^j\right]\\
&~~~-\frac{2(1-p)^2}{(2-r)^2(3-r)p^2}\sum_{\mu\in\mathbb{F}_p}\left[g^0_{2, p[\mu]}, \sum_{\substack{{2\leq j\leq r-p}\\{j\equiv 2~\text{mod}~(p-1)}}}\binom{r}{j}X^{r-j}(-[\mu]X+pY)^j\right]\\
&~~~ +\frac{3}{(2-r)^2p}\sum_{\lambda\in\mathbb{F}_p^\times}\sum_{\mu\in\mathbb{F}_p}[g^0_{2,[\lambda]+p[\mu]}, X^{r-2}(-[\mu]X+pY)^2] \\
&~~~+ \frac{3(1-p)}{(2-r)^2p}\sum_{\mu\in\mathbb{F}_p}[g^0_{2,p[\mu]}, X^{r-2}(-[\mu]X+pY)^2].
\end{align*}
By using Lemma \ref{lemma 0.2} and the fact that $r-1 \geq t +5$ in the first and second sums, Lemma~ \ref{lemma 0.1.} for $\mu =0$ in the third and fourth sums, we get
\begin{align*}
T^+f_1&= \frac{2}{(2-r)^2(3-r)p^2}\sum_{\lambda\in\mathbb{F}_p^\times}\sum_{\mu\in\mathbb{F}_p^\times}\bigg[g^0_{2,[\lambda] + p[\mu]}, [\mu]^2(3-r)X^r - [\mu]pr(3-r)X^{r-1}Y \\
&~~~+ p^2\left(1-(r-2)\binom{r-1}{2}\right)X^{r-2}Y^2\bigg] \\
&~~~+\frac{2(1-p)}{(2-r)^2(3-r)p^2}\sum_{\mu\in\mathbb{F}_p^\times}\bigg[g^0_{2, p[\mu]},  [\mu]^2(3-r)X^r - [\mu]pr(3-r)X^{r-1}Y \\
&~~~+ p^2\left(1-(r-2)\binom{r-1}{2}\right)X^{r-2}Y^2 \bigg]\\
&~~~+ \frac{2}{(2-r)^2(3-r)p^2}\left(\sum_{\lambda\in\mathbb{F}_p^\times}[g^0_{2,[\lambda]}, p^2X^{r-2}Y^2] + (1-p)[g^0_{2,0}, p^2X^{r-2}Y^2]\right)\\
&~~~ - \frac{2(1-p)}{(2-r)^2(3-r)p^2}\sum_{\lambda\in\mathbb{F}_p^\times}\sum_{\mu\in\mathbb{F}_p^\times}\left[g^0_{2,[\lambda]+ p[\mu]}, \sum_{i=0}^r(-[\mu])^{2-i}p^i\sum_{\substack{{2\leq j\leq r-p}\\{j\equiv 2~\text{mod}~(p-1)}}}\binom{j}{i}\binom{r}{j}X^{r-i}Y^i\right]\\
&~~~ - \frac{2(1-p)^2}{(2-r)^2(3-r)p^2}\sum_{\mu\in\mathbb{F}_p^\times}\left[g^0_{2,p[\mu]}, \sum_{i=0}^r(-[\mu])^{2-i}p^i\sum_{\substack{{2\leq j\leq r-p}\\{j\equiv 2~\text{mod}~(p-1)}}}\binom{j}{i}\binom{r}{j}X^{r-i}Y^i\right]\\
&~~~ - \frac{2(1-p)}{(2-r)^2(3-r)p^2}\left(\sum_{\lambda\in\mathbb{F}_p^\times}\bigg[g^0_{2,[\lambda]}, p^2\binom{r}{2}X^{r-2}Y^2\bigg] + (1-p)[g^0_{2,0}, p^2\binom{r}{2}X^{r-2}Y^2]\right)\\
&~~~ + \frac{3}{(2-r)^2p}\sum_{\lambda\in\mathbb{F}_p^\times}\sum_{\mu\in\mathbb{F}_p^\times}[g^0_{2,[\lambda] + p[\mu]}, [\mu]^2X^r - 2p[\mu]X^{r-1}Y]\\
&~~~ + \frac{3(1-p)}{(2-r)^2p}\sum_{\mu\in\mathbb{F}_p^\times}[g^0_{2,p[\mu]}, [\mu]^2X^r -2p[\mu]X^{r-1}Y] + O(p).
\end{align*}
By Proposition \ref{Proposition 0.10}, parts \textit{(1)}, \textit{(2)}, \textit{(3)}, \textit{(4)}, \textit{(5)} and Lemma \ref{lemma 0.2}, the coefficients of $[g^0_{2,[\lambda] +p[\mu]}, X^r]$, $[g^0_{2,[\lambda] +p[\mu]}, X^{r-1}Y]$, $[g^0_{2,[\lambda] +p[\mu]}, X^{r-2}Y^2]$, $[g^0_{2,[\lambda] +p[\mu]}, X^{r-3}Y^3]$, $[g^0_{2,[\lambda] +p[\mu]}, X^{r-i}Y^i]$ for $i\geq 4$, $\lambda\neq 0$ and $\mu\neq 0$ in 
$T^+f_1$ are
\begin{align*}
\frac{2[\mu]^2}{(2-r)^2p^2} &-\frac{2[\mu]^2(1-p)}{(2-r)^2(3-r)p^2}\left(\sum_{\substack{{2\leq j\leq r-p}\\{j \equiv 2~\text{mod}~(p-1)}}}\binom{r}{j}\right) +\frac{3[\mu]^2}{(2-r)^2p}\\
&\equiv\frac{2[\mu]^2}{(2-r)^2p^2}-\frac{2[\mu]^2(1-p)}{(2-r)^2(3-r)p^2}\left(3-r + \frac{5(3-r)p}{2(1-p)}\right) + \frac{3[\mu]^2}{(2-r)^2p}~\text{mod}~\mathbb{Z}_p\\
&=\frac{2[\mu]^2}{(2-r)^2p^2}(1-1) - \frac{2[\mu]^2}{(2-r)^2p}(-1 + \frac{5}{2}- \frac{3}{2}) ~\text{mod}~\mathbb{Z}_p\\
&= 0 ~\text{mod}~\mathbb{Z}_p,\\
\frac{-2[\mu]r}{(2-r)^2p} &+\frac{2[\mu](1-p)}{(2-r)^2(3-r)p}\left(\sum_{\substack{{2\leq j \leq r-p}\\{j \equiv 2~\text{mod}~(p-1)}}}j\binom{r}{j}\right) -\frac{6[\mu]}{(2-r)^2}\\
&\equiv \frac{-2[\mu]r}{(2-r)^2p} + \frac{2[\mu](1-p)}{(2-r)^2(3-r)p}(r(3-r))  ~\text{mod}~\mathbb{Z}_p\\
&\equiv 0~\text{mod}~\mathbb{Z}_p,\\
\frac{2\left(1-(r-2)\binom{r-1}{2}\right)}{(2-r)^2(3-r)} &-\frac{2(1-p)}{(2-r)^2(3-r)}\left(\sum_{\substack{{2\leq j\leq r-p}\\{j\equiv 2~\text{mod}~(p-1)}}}\binom{j}{2}\binom{r}{j}\right)\\
&\equiv \frac{-r^3+ 5r^2 -8r + 6}{(2-r)^2(3-r)} - \frac{2(1-p)}{(2-r)^2(3-r)}\left(\binom{r}{2} - \binom{r-1}{2}r\right)~\text{mod}~p\\
& =\frac{(3-r)(r^2 - 2r + 2)}{(2-r)^2(3-r)} -\frac{2\binom{r}{2}(1-p)}{(2-r)^2(3-r)}\left(1- (r-2)\right)~\text{mod}~p\\
&\equiv\frac{r^2-2r+2}{(2-r)^2} - \frac{r(r-1)}{(2-r)^2}~\text{mod}~p\\
&=\frac{1}{2-r}~\text{mod}~p,\\
 \frac{2[\mu]^{-1}(1-p)p}{(2-r)^2(3-r)}&\left(\sum_{\substack{{2\leq j \leq r-p}\\{j\equiv 2~\text{mod}~(p-1)}}}\binom{j}{3}\binom{r}{j}\right)\\
 &\equiv \frac{2[\mu]^{-1}(1-p)p}{(2-r)^2(3-r)}\left(\frac{\binom{r}{3}}{p-1} -\binom{r-1}{3}r\right)~\text{mod}~p\\
 &\equiv \frac{-2[\mu]^{-1}p\binom{r}{3}}{(2-r)^2(3-r)}~\text{mod}~p,\\
 \frac{-2(-[\mu])^{2-i}(1-p)}{(2-r)^2(3-r)p^2}&\left(p^i\sum_{\substack{{2\leq j \leq r-p}\\{j\equiv 2~\text{mod}~(p-1)}}}\binom{j}{i}\binom{r}{j}\right)\\
 &\equiv \frac{-2(-[\mu])^{2-i}(1-p)}{(2-r)^2(3-r)p^2}\left(-p^i\binom{r-1}{i}r\right)~\text{mod}~p\\
 & \equiv 0~\text{mod}~p,
\end{align*}
respectively. 

The coefficients of $[g^0_{2,p[\mu]}, X^r]$, $[g^0_{2,p[\mu]}, X^{r-1}Y]$, $[g^0_{2,p[\mu]}, X^{r-2}Y^2]$, $[g^0_{2,p[\mu]}, X^{r-3}Y^3]$ and $[g^0_{2,p[\mu]}, X^{r-i}Y^i]$ for $i\geq 4$, $\mu\neq 0$ in $T^+f_1$ is $(1-p)$ times the coefficients above. %i.e., they are
%\begin{align*}
%(1-p)\left(\frac{2[\mu]^2}{(2-r)^2p^2} -\frac{2[\mu]^2(1-p)}{(2-r)^2(3-r)p^2}\left(\sum_{\substack{{2\leq j\leq r-p}\\{j \equiv 2~\text{mod}~(p-1)}}}\binom{r}{j}\right) +\frac{3[\mu]^2}{(2-r)^2p}\right) &\equiv 0~\text{mod}~\mathbb{Z}_p,\\
%(1-p)\left(\frac{-2[\mu]r}{(2-r)^2p} +\frac{2[\mu](1-p)}{(2-r)^2(3-r)p}\left(\sum_{\substack{{2\leq j \leq r-p}\\{j \equiv 2~\text{mod}~(p-1)}}}j\binom{r}{j}\right) -\frac{6[\mu]}{(2-r)^2}\right) &\equiv 0~\text{mod}~\mathbb{Z}_p,\\
%(1-p)\left(\frac{2\left(1-(r-2)\binom{r-1}{2}\right)}{(2-r)^2(3-r)} -\frac{2(1-p)}{(2-r)^2(3-r)}\left(\sum_{\substack{{2\leq j\leq r-p}\\{j\equiv 2~\text{mod}~(p-1)}}}\binom{j}{2}\binom{r}{j}\right)\right) &\equiv \frac{1}{2-r}~\text{mod}~p,\\
% \frac{2[\mu]^{-1}(1-p)^2p}{(2-r)^2(3-r)}\left(\sum_{\substack{{2\leq j \leq r-p}\\{j\equiv 2~\text{mod}~(p-1)}}}\binom{j}{3}\binom{r}{j}\right) &\equiv \frac{-2[\mu]^{-1}(1-p)p\binom{r}{3}}{(2-r)^2(3-r)}\\&~~~~~\qquad \qquad \text{mod}~p,\\
% \frac{-2(-[\mu])^{2-i}(1-p)^2}{(2-r)^2(3-r)p^2}\left(p^i\sum_{\substack{{2\leq j \leq r-p}\\{j\equiv 2~\text{mod}~(p-1)}}}\binom{j}{i}\binom{r}{j}\right)&\equiv 0~\text{mod}~p,
%\end{align*}
%respectively. 
So, finally we get
\begin{align}
T^+f_1 &= \frac{1}{2-r}\sum_{\lambda\in\mathbb{F}_p}\sum_{\mu\in\mathbb{F}_p^\times}[g^0_{2,[\lambda]+p[\mu]}, X^{r-2}Y^2] \nonumber\\
 &~~~+ \frac{2\left(1-(1-p)\binom{r}{2}\right)}{(2-r)^2(3-r)}\left(\sum_{\lambda\in\mathbb{F}_p^\times}[g^0_{2,[\lambda]}, X^{r-2}Y^2] + (1-p)[g^0_{2,0}, X^{r-2}Y^2]\right)\nonumber\\
 &~~~-\frac{2p\binom{r}{3}}{(2-r)^2(3-r)}\sum_{\mu\in\mathbb{F}_p^\times}\left(\sum_{\lambda\in\mathbb{F}_p^\times}[g^0_{2,[\lambda]+p[\mu]}, [\mu]^{-1}X^{r-3}Y^3] + (1-p)[g^0_{2,p[\mu]}, [\mu]^{-1}, X^{r-3}Y^3]\right)\nonumber\\
 &~~~ + h + O(p), \label{F_2 tau > = t+1 T plus f_1}
\end{align}
where $h$ is an integral linear combination of the terms of the form $[g, X^r]$ and $[g, X^{r-1}Y]$, for some $g\in G$.

Let $h_{2,\infty}$ be the radius $2$ part of $(T-a_p)f_\infty$. By Lemma \ref{lemma 0.18} and  Lemma \ref{Lemma 0.17.}, we have
\begin{align}
h_{2,\infty} &= \frac{2(2-pr)}{(2-r)^3(3-r)}\left(\frac{-a_p^2}{p^3}\right)\Bigg(\sum_{\lambda\in\mathbb{F}_p^\times}[g^0_{2,[\lambda]}, X^{r-2}Y^2 + (r-3)X^pY^{r-p}] \nonumber \\
&~~~+ (1-p)[g^0_{2,0}, X^{r-2}Y^2 + (r-3)X^pY^{r-p}] \Bigg)\nonumber \\
&~~~+ \frac{2}{(2-r)^2(3-r)}\left(\frac{a_p^2}{p^2}\right)\sum_{\mu \in\mathbb{F}_p^\times}\left(\sum_{\lambda\in\mathbb{F}_p^\times}[g^0_{2,[\lambda]+ p[\mu]}, [\mu]^{-1}X^{r-3}Y^3]+(1-p)[g^0_{2,p[\mu]}, [\mu]^{-1}X^{r-3}Y^3]\right)\nonumber\\
&~~~ + O(\sqrt{p}). \label{F_2 tau >= t+1 h(2,infty)}
\end{align}
By \eqref{F_2 tau > = t+1 T plus f_1} and \eqref{F_2 tau >= t+1 h(2,infty)}, the coefficient of $[g^0_{2,[\lambda]}, X^{r-2}Y^2]$ for $\lambda\neq 0$ in $T^+f_1 + h_{2,\infty}$ is 
\begin{align*}
&\frac{2\left(1-(1-p)\binom{r}{2}\right)}{(2-r)^2(3-r)} - \frac{2(2-pr)a_p^2}{(2-r)^3(3-r)p^3}\\
%&~~~= \frac{2}{(2-r)^2(3-r)}-\frac{2(1-p)\binom{r}{2}}{(2-r)^2(3-r)} - \frac{2a_p^2}{(2-r)^2(3-r)p^3} + \frac{2(1-p)ra_p^2}{(2-r)^3(3-r)p^3}\\
%&~~~=\frac{2}{(2-r)^2(3-r)p^3}(p^3-a_p^2) - \frac{2(1-p)}{(2-r)^3(3-r)p^3}\left((2-r)\binom{r}{2}p^3 + ra_p^2\right)\\
%&~~~=\frac{2(p^3-a_p^2)}{(2-r)^2(3-r)p^3} - \frac{2(1-p)r\left(a_p^2-\binom{r-1}{2}p^3\right)}{(2-r)^3(3-r)p^3}\\
&~~~= \frac{2\left(p^3-(pa_pc + (r-2)\binom{r-1}{2}p^3)\right)}{(2-r)^2(3-r)p^3} - \frac{2(1-p)r\left(pa_pc + (r-2)\binom{r-1}{2}p^3-\binom{r-1}{2}p^3\right)}{(2-r)^3(3-r)p^3}\\
&~~~ \equiv\frac{2\left(1-(r-2)\binom{r-1}{2}\right)}{(2-r)^2(3-r)} -\frac{2(1-p)r\binom{r-1}{2}(r-2-1)}{(2-r)^3(3-r)}~\text{mod}~\sqrt{p}\\
%&~~~\equiv \frac{r^2-2r+2}{(2-r)^2} - \frac{r(r-1)}{(2-r)^2}~\text{mod}~\sqrt{p}\\
&~~~= \frac{1}{2-r}~\text{mod}~\sqrt{p},
\end{align*}
since $v(c) \geq t+1$.

\noindent Similarly, the coefficient of $[g^0_{2,0}, X^{r-2}Y^2]$ in $T^+f_1+ h_{2,\infty}$ is $(1-p)$ times the coefficient above, i.e., 
is still  $\frac{1}{2-r}$ mod $\sqrt{p}$.
%\begin{align*}
%(1-p)&\left(\frac{2\left(1-(1-p)\binom{r}{2}\right)}{(2-r)^2(3-r)} - \frac{2(2-pr)a_p^2}{(2-r)^3(3-r)p^3}\right)\\
%& \equiv \frac{1}{2-r}~\text{mod}~\sqrt{p}.
%\end{align*}

\noindent By \eqref{F_2 tau > = t+1 T plus f_1} and \eqref{F_2 tau >= t+1 h(2,infty)}, the coefficient of $[g^0_{[\lambda]+p[\mu]}, [\mu]^{-1}X^{r-3}Y^3]$ for $\lambda\neq 0$ and $\mu\neq 0$ is 
\begin{align*}
&\frac{-2p\binom{r}{3}}{(2-r)^2(3-r)} + \frac{2a_p^2}{(2-r)^2(3-r)p^2}
%&~~~ = \frac{2\left(a_p^2 -\binom{r}{3}p^3\right)}{(2-r)^2(3-r)p^2}\\
  = \frac{2\left(pa_pc + (r-2)\binom{r-1}{2}p^3 - \binom{r}{3}p^3\right)}{(2-r)^2(3-r)p^2} 
 \equiv 0 ~\text{mod}~p.
\end{align*}
Similarly the coefficient of $[g^0_{2,p[\mu]}, [\mu]^{-1}X^{r-3}Y^3]$ for $\mu\neq 0$ in $T^+f_1 + h_{2,\infty}$ is $(1-p)$ times the coefficient above, i.e., is still $0$ mod $p$.
%\begin{align*}
%(1-p)&\left(\frac{-2p\binom{r}{3}}{(2-r)^2(3-r)} + \frac{2a_p^2}{(2-r)^2(3-r)p^2}\right)\\
%&\equiv 0~\text{mod}~p.
%\end{align*}

\noindent Therefore by \eqref{F_2 tau > = t+1 T plus f_1} and \eqref{F_2 tau >= t+1 h(2,infty)}, in radius $2$, we get
\begin{align}
T^+f_1 + h_{2,\infty} &= \frac{1}{2-r}\sum_{\lambda\in\mathbb{F}_p}\sum_{\mu\in\mathbb{F}_p}[g^0_{2,[\lambda]+p[\mu]}, X^{r-2}Y^2] \nonumber \\ 
&~~~ + \frac{2(2-pr)a_p^2}{(2-r)^3p^3}\sum_{\lambda\in\mathbb{F}_p}[g^0_{2,[\lambda]}, X^pY^{r-p}] + h + O(\sqrt{p}). \label{F_2 tau >= t+1 radius 2}
\end{align}

Putting everything together, from \eqref{F_2 tau > = t+1 radius -1}, \eqref{F_2 tau >= t+1 radius 0}, \eqref{F_2 tau >= t+1 radius 1} and \eqref{F_2 tau >= t+1 radius 2}, we finally get
\begin{align}
(T-a_p)f & = \frac{-1}{2-r}[1, X^{r-2}Y^2] + 2[1, \theta X^{r-p-2}Y] - \frac{2c}{(2-r)^2(3-r)p}\left(\sum_{\lambda \in\mathbb{F}_p}[g^0_{1,[\lambda]}, X^{r-2}Y^2]\right)\nonumber\\
&~~~ + \frac{1}{2-r}\sum_{\lambda\in\mathbb{F}_p}\sum_{\mu\in\mathbb{F}_p}[g^0_{2,[\lambda]+p[\mu]}, X^{r-2}Y^2]
 + \frac{2(2-pr)a_p^2}{(2-r)^3p^3}\sum_{\lambda\in\mathbb{F}_p}[g^0_{2,[\lambda]}, X^pY^{r-p}] + h + O(\sqrt{p}). \label{ F_2 tau >= t+1 (T-a_p)f}
\end{align}

We project to $Q = \frac{V_r}{X_{r-1} + V_r^{**}}$. All the polynomials occurring in \eqref{ F_2 tau >= t+1 (T-a_p)f} are in fact in the image of $V_r^*$ in $Q$. Indeed in $Q$, we have $X^{r-2}Y^2 = X^{r-2}Y^2-XY^{r-1} = \theta ((2-r)X^{r-p-2} + Y^{r-p-1})$, by \eqref{image of X^{r-2}Y^2 mod V_r**}, and $X^pY^{r-p} =  X^pY^{r-p} -XY^{r-1} = \theta Y^{r-p-1}$. 

By Proposition \ref{Structure of Q}, part \textit{(1)}, if $t =0$, then the image of $V_r^*$ in $Q$ is $J_3 = V_{p-2}\otimes D^2$, and by part \textit{(2)}, if $t>0$, then the image of $V_r^{*}$ in $Q$ is $\frac{V_r^*}{V_r^{**}}$ which has $J_3$ as a further quotient.
So in either case we may project to $J_3 = V_{p-2}\otimes D^2$.
Now $h$ maps to zero in ind$_{KZ}^G Q$ %therefore is zero in  ind$_{KZ}^G V_r^*/V_r^{**}$ 
and therefore to zero in ind$_{KZ}^G J_3$.  Similarly $X^pY^{r-p}$ maps to zero in $J_3$, since $\theta Y^{r-p-1}$ maps to $0$ in $J_3$,
%by  Lemma 8.5 in \cite{Bhattacharya-Ghate}. 
by part (ii) of Lemma~\ref{generator}. Also, by part (iii) of the same lemma, $\theta X^{r-p-2}Y$ maps to $X^{p-2}$ in $J_3$. 
%Note that $J_3$ for us is $J_1$ in \cite{Bhattacharya-Ghate}. 
Finally, by \eqref{image of X^{r-2}Y^2 in J_2}, $X^{r-2}Y^2$ maps to $(2-r)X^{p-2}$ in $J_3$. 

Therefore, 
%by \eqref{F_2 tau > = t+1 radius -1}, \eqref{F_2 tau >= t+1 radius 0}, \eqref{F_2 tau >= t+1 radius 1} and \eqref{F_2 tau >= t+1 radius 2}, 
the image of $(T-a_p)f$ in ind$_{KZ}^GJ_3$ is 
\begin{align*}
\overline{(T-a_p)f} & = (-1+2) [1, X^{p-2}] - \overline{\frac{2c}{(2-r)(3-r)p}}\sum_{\lambda\in\mathbb{F}_p}[g^0_{1,[\lambda]}, X^{p-2}] + \sum_{\lambda\in\mathbb{F}_p}\sum_{\mu\in\mathbb{F}_p}[g^0_{2,[\lambda] + p[\mu]}, X^{p-2}] \nonumber\\
&= (T^2-\bar{d}T + 1)[1, X^{p-2}], \label{F_2 tau >= t+1 (T-a_p)f}
\end{align*}
where $d =\frac{2c}{(2-r)(3-r)p}$. 
As $\overline{(T-a_p)f}$ maps to zero in $\bar{\Theta}_{k,a_p}$ and $[1, X^{p-2}]$ generates ind$_{KZ}^GJ_3$, we get that $(T^2-\bar{d}T + 1)(\text{ind}_{KZ}^GJ_3) \subset \ker (\text{ind}_{KZ}^GJ_3\twoheadrightarrow F_3)$.

Therefore if $\tau = t+1 $, then $F_3$ is a quotient of $\frac{\text{ind}_{KZ}^GJ_3}{T^2-\bar{d}T+1}$ and if $\tau > t+1$, then $F_3$ is a quotient of $\frac{\text{ind}_{KZ}^GJ_3}{T^2+1}$, since $\bar{d} = 0$.
\end{proof}

{\noindent \bf Acknowledgments:} The first author thanks Institut de Math\'ematiques de Jussieu, Max Planck Institute and Institut de Math\'ematiques de Bordeaux
for their hospitality in 2017 and 2018. The second author thanks Tata Institute of Fundamental Research for its support.
%The second author wishes to thank the first author for introducing him to
%> > the problem of mod p reductions of Galois representations and for his
%> > constant support throughout the project. 
Both authors thank Shalini Bhattacharya and Sandra Rozensztajn for useful discussions.

\end{document}